\documentclass[3p,final,times,sort&compress]{elsarticle}
\journal{EJOR}
\date{}
\let\today\relax
\makeatletter
\def\ps@pprintTitle{%
    \let\@oddhead\@empty
    \let\@evenhead\@empty
    \def\@oddfoot{\footnotesize\itshape
         {} \hfill\today}%
    \let\@evenfoot\@oddfoot
    }
\makeatother

\usepackage{amsmath, amssymb, amsthm}
\usepackage{mathtools}
\usepackage{multirow}
\usepackage{subcaption}
\usepackage{graphicx}
\graphicspath{{imagens/}}
\usepackage{tabularx,ragged2e}
\newcolumntype{C}{>{\Centering\arraybackslash}X} 
\usepackage{booktabs}
\usepackage{xcolor}
\usepackage[ruled, linesnumbered]{algorithm2e}
\usepackage[unicode,pdfpagelabels,colorlinks,bookmarksopen]{hyperref}

\usepackage[shortlabels]{enumitem}
\newlist{lista}{enumerate}{1}
\setlist[lista]{label={\upshape{(\alph*)}}, leftmargin=*,align=right, widest=w}

\newlist{listi}{enumerate}{1}
\setlist[listi]{label={\upshape(\roman*\upshape)},leftmargin=*,align=right, widest=iii}

\usepackage{array}
\usepackage{multirow}           
\usepackage{siunitx}

\newtheorem{theorem}{Theorem}[section]
\newtheorem{lema}[theorem]{Lemma}

\newtheorem{prop}[theorem]{Proposition}
\theoremstyle{definition}
\newtheorem{defi}[theorem]{Definition}
\newtheorem{assump}[theorem]{Assumption}
\theoremstyle{remark}
\newtheorem{remark}[theorem]{Remark}
\newtheorem{example}[theorem]{Example}

\usepackage[shortlabels]{enumitem}

\DeclareMathOperator*{\conv}{conv}
\newcommand{\convS}{\ensuremath{\conv(\mathcal{A})}}

\newcommand{\convA}{\ensuremath{\conv(\mathcal{A})}}
\newcommand{\ISet}{\ensuremath{\mathcal{I}}}

\newcommand\R{\mathbb{R}}
\newcommand{\Rm}{\mathbb{R}^m}
\newcommand{\Rn}{\mathbb{R}^n}

\newcommand\N{\mathbb{N}}

\DeclareMathOperator*{\argmin}{arg\, min}
\DeclareMathOperator*{\argmax}{arg\, max}

\newcommand\norm[1]{\left\lVert #1 \right\rVert}

\newcommand\abs[1]{\left\lvert #1 \right\rvert}

\newcommand{\A}{{\cal A}}

\newcommand{\x}{\bar{x}}
\newcommand{\xs}{x^{\star}}
\newcommand{\ys}{y^{\star}}
\newcommand{\fstar}{f^{\star}}

\usepackage[left,mathlines]{lineno}

\usepackage{autonum}
\begin{document}
\begin{frontmatter}

\title{First-order methods for the convex hull membership problem} 

\author[1]{Rafaela Filippozzi}
\ead{rafaela.filippozzi@gmail.com}

\affiliation[1]{organization={Department of Mathematics,  Federal University of Santa Catarina},
	city={Florianopolis},
	state={SC},
	postcode={88040-900}, 
	country={Brazil}
}

\author[1]{Douglas S.~Gon\c{c}alves\corref{cor1}}
\ead{douglas.goncalves@ufsc.br}

\author[ufscbnu]{Luiz-Rafael Santos}
\ead{l.r.santos@ufsc.br}

\affiliation[ufscbnu]{organization={Department of Mathematics,  Federal University of Santa Catarina},
	city={Blumenau},
	state={SC},
	postcode={89065-300},
	country={Brazil}
}
\cortext[cor1]{Corresponding author}


\begin{abstract}\small
	The convex hull membership problem (CHMP) consists in deciding whether a certain point belongs to the convex hull of a finite set of points, a decision problem with important applications in computational geometry and in foundations of linear programming. 
In this study, we review, compare and analyze first-order methods for CHMP, namely, Frank-Wolfe type methods, Projected Gradient methods and a recently  introduced geometric algorithm, called Triangle Algorithm (TA). We discuss the connections between this algorithm and Frank-Wolfe, showing that TA can be  interpreted as an inexact Frank-Wolfe. Despite this similarity, TA is strongly based on a theorem of alternatives known as distance duality. 
By using this theorem, we propose suitable stopping criteria for CHMP to be integrated into Frank-Wolfe type and Projected Gradient,  
specializing these methods to the membership decision problem. 
Interestingly, Frank-Wolfe integrated with such stopping criteria coincides with a greedy version of the Triangle Algorithm which is, in its turn, equivalent to an algorithm due to von Neumann.  
We report numerical experiments on random instances of CHMP, carefully designed to cover different scenarios, that indicate which algorithm is preferable according to the geometry of the convex hull and the relative position of the query point. 
Concerning potential applications, we present two illustrative examples, one related to linear programming feasibility problems and another related to image classification problems.

\end{abstract}


\end{frontmatter}

\section{Introduction} \label{sec:intro}
Let ${\cal A}\coloneqq \{v_1,v_2,\dots, v_n\} \subset \mathbb{R}^m$ be given and consider a point $p\in \mathbb{R}^m$. The \emph{convex hull membership problem} (CHMP) consists in deciding whether
 \begin{equation}
     p \in \convA, \label{eq:chmp}
 \end{equation}
where $\convA$ denotes the convex hull of ${\cal A}$. 
This problem is related to fundamental concepts in linear programming 
and finds important applications in  computational geometry. 

Throughout this paper, we denote by $e \in \Rn $ the vector whose $n$ components are all equal to one,  by $v^T p$  the Euclidean inner product between vectors $v,p \in \Rm$  and   by $\| \cdot\|$ its induced Euclidean norm.  We also denote the Euclidean distance between $v$ and $p$ as $d(v,p) \coloneqq  \| v - p \|$, the Euclidean ball centered in $v$ with radius $\rho$ as $B_\rho(v) \coloneqq \{p\in \R^m \mid d(v,p) < \rho \}$ and the convex combination or segment  between $v$ and $p$ as $[v,p]$. {The proofs  presented are three fold: for new results, when the proof is different from previous proofs in the literature or when the proof contributes to a subsequent discussion.}

Let $A \coloneqq [v_1 \ v_2 \  \cdots \ v_n ] \in \R^{m \times n}$ be the matrix where each column is  one of the $n$ points of $\mathcal{A}$.
One can see that  $p \in \convA$ if and only if $p$ is a convex combination of the columns of $A$. Thus,  we can formally describe the convex hull membership problem~\eqref{eq:chmp} as the following decision problem: 
\begin{equation}
\text{Is there any} \ x\in \Rn \text{ such that } Ax = p, e^T x =1, x \geq 0 ? \label{eq:lp}
\end{equation}
Problem~\eqref{eq:lp} is a linear programming feasibility problem whose affirmative answer ensures $p \in \convA$.   
Such feasibility problem can also be cast as the following quadratic programming problem:
\begin{equation}\label{prob:quad}
\begin{aligned}
\min_{x \in \Delta_n} & \quad \dfrac{1}{2} \| Ax - p \|^2 =: {\Phi(x)}, \\
\end{aligned}
\end{equation}
where $\Delta_n \coloneqq \{ x \in \Rn \mid  e^T x = 1,  x \ge 0 \}$ is the unit simplex in $\Rn$. Clearly, $p \in \convA$ when the optimal value of \eqref{prob:quad} is zero.
Furthermore, if $x$ is a feasible point of \eqref{prob:quad}, then $y \coloneqq  Ax \in \convA$. Therefore, another possible formulation of \eqref{eq:chmp} is given by
\begin{equation}\label{prob:proj}
\begin{aligned}
\min_{y \in \Rm} & \quad \dfrac{1}{2} \| y - p \|^2 =: {\Psi(y)} \\
\text{s.t.} & \quad y \in \convA, 
\end{aligned}
\end{equation}
which achieves a zero optimal value if and only if $p \in \convA$.

With  formulations \eqref{eq:lp}, \eqref{prob:quad} and \eqref{prob:proj} in hand, one could simply apply classic methods for linear or quadratic programming in order to solve CHMP. 
Nevertheless,  \cite{Kalantari}, inspired by geometric ideas, proposes a specific algorithm for CHMP, called \emph{Triangle Algorithm} (TA), which shows remarkable numerical results \citep{KalantariComparacao} in comparison with the Simplex method applied to \eqref{eq:lp} or with Frank-Wolfe (FW) \citep{fw1956} 
applied to \eqref{prob:quad}. 

TA was also successfully employed to solve other important convex hull related problems. For example,  \citet{Kalantari:2019a} considered the more general problem of separation of two compact convex sets (CHMP is a special case, when one of the sets is a singleton).  \citet{awasthi2020} showed that TA is the building block for an algorithm to enumerate all extreme points of convex hulls in high dimension. 
For further recent developments involving the Triangle Algorithm see the works of \citet{Kalantari:2014b,Kalantari:2019,Kalantari:2019b,Kalantari:2020a}. 
{In particular,  \citet{Kalantari:2019b} consider the \emph{Spherical} CHMP where all the points belonging to $\A$ have unit norm and the point $p$ is the origin. The authors showed not only the equivalence between CHMP and spherical CHMP, but also explained how approximate solutions for the former provide approximate solutions for the latter.  Furthermore, it was show that a variant of the Triangle Algorithm for spherical CHMP almost obtains $O(1/\varepsilon)$ complexity as opposed to $O(1/\varepsilon^2)$ complexity of the standard TA and Frank-Wolfe (see Section~\ref{sec:rel} for details).}

In general lines, the Triangle Algorithm iterations can be described as follows. For an iterate $p'\in \convA$, the algorithm selects  $v_i \in {\cal A}$, such that $d(p',v_i) \geq d(p,v_i)$. Such $v_i \in \A$ is called a \emph{pivot}. If a pivot does not exist, we can stop and declare $p\notin \convA$; see Theorem~\ref{teo:dd}. 
Otherwise, update the next iterate to the point lying in the segment $[p',v]$ which is closest to $p$. 
The iterations continue until 
$\| p' - p \| \leq \varepsilon$, in which case it is declared $p \in \convA$. Figure~\ref{fig:agiter} illustrates the scheme (where the sequence of iterates is denoted by $p',p'', \dots$). 
\begin{figure*}
\centering
  \begin{subfigure}[t]{.48\linewidth}
    \centering
    \includegraphics[width=.8\textwidth]{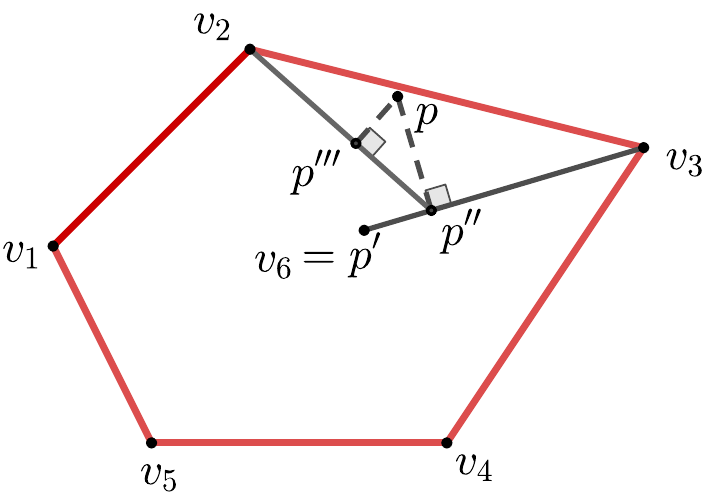}
    \caption{\label{fig:agiter-a} Triangle Algorithm when $p \in \convA$.}
  \end{subfigure}
  \begin{subfigure}[t]{.48\linewidth}
    \centering
    \includegraphics[width=.8\textwidth]{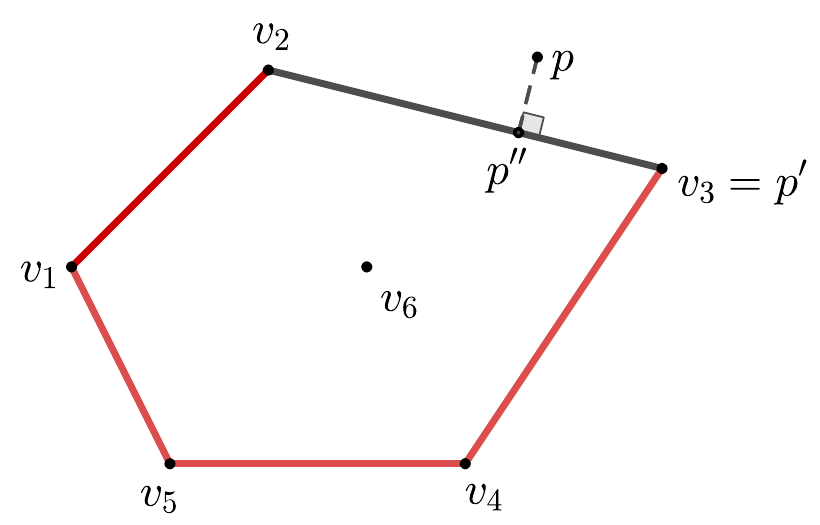}
    \caption{\label{fig:agiter-b} Triangle Algorithm when $p \notin \convA$.}
  \end{subfigure} 
	\caption{Iterations of the Triangle Algorithm.}	\label{fig:agiter}
\end{figure*}

The correctness of the Triangle algorithm follows from the following theorem of alternatives, called  \emph{distance duality theorem} and whose proof can be found in \citet[Theorem~4]{Kalantari}.

\begin{theorem}[Distance duality]\label{teo:dd}
For a given set ${\cal A}=\{v_1,v_2,\dots, v_n\} \subset \mathbb{R}^m$ and a point $p\in \mathbb{R}^m$, precisely one of the two conditions is satisfied:	
\begin{listi}
	\item  For all $p' \in \convA$, there exists $v_i\in {\cal A}$ such that $d(p',v_i) \geq d(p,v_i)$;
	\item There exists $p'\in \convA$ such that $d(p',v_i) < d(p,v_i)$, for all $i = 1,\dots , n.$
\end{listi}
\end{theorem}

Note that Theorem~\ref{teo:dd} implies that either for every $p' \in \convA$ there exists $v\in \cal{A}$ such that $v$ is closer to $p$ than to $p'$, or there exists $p' \in \convA$ such that the bisector hyperplane orthogonal to the segment  $[p',p]$  separates $p$ from $\convA$ 
(see Figure~1 in \cite{Kalantari}). 

\cite{Kalantari} proved that the first condition is valid if and only if $p \in \convA,$ and the second condition if and only if $p \notin \convA$. 
If the latter holds, we say that $p'$ is a \emph{witness} that $p \notin \convA$ because the hyperplane  
\begin{equation}\label{eq:hiperbis}
 H_{[p,p']}\coloneqq  \{y\in \Rm \mid (p - p')^T y = (p - p')^T (p'+p)/2\}
\end{equation} 
strictly separates $p$ from $\convA$. 
Furthermore, when $p \notin \convA$, for a $p$-witness $p'$, we have
\begin{equation}\label{eq:witproj}
    \dfrac{1}{2} \|p'-p\| \leq \Delta \leq \|p'-p\|,
\end{equation}
where $\Delta := \min \{ d(v,p) \mid v \in \convA \}$. Thus, the distance from $p$ to $\convA$ is approximated within a factor of two.


The distance duality theorem provides unique features to the Triangle algorithm when compared to first-order algorithms applied to the  optimization formulations of CHMP.  For instance, one can consider applying the classic FW to \eqref{prob:quad} or \eqref{prob:proj}. We shall show later (see Section~\ref{sec:rel}) that the subproblem that must be solved at each FW iteration  corresponds to find $v_i \in \argmin \{  (p'-p)^T v_j \mid v_j \in \A \}$ whereas $v_i \in \A$ is a pivot for TA if and only if $(p'-p)^T v_i \leq (\|p'\|^2 - \|p\|^2)/2$ (see Lemma~\ref{lema:pivo}).  
Thus, the first alternative of Theorem~\ref{teo:dd} relaxes this requirement of the ``optimal'' pivot chosen by FW. Moreover, the second alternative of Theorem~\ref{teo:dd} provides a smart stopping criterion, allowing us to state that $p \notin \convA$ without the need of solving \eqref{prob:quad} (or \eqref{prob:proj}) to $\varepsilon$-optimality.

Yet,  TA is similar to FW applied to \eqref{prob:proj} in the sense that it keeps its iterates as convex combinations of points of $\A$, adding at most one new point per iteration to this combination. What is more, we will show that the so-called pivot of the Triangle algorithm can be seen as an inexact solution to the FW subproblem. Not surprisingly, convergence and iteration complexity results for the Triangle algorithm (and its variants) follow closely those of FW methods: namely, in general the  convergence is usually sublinear, and linear convergence is only achieved under stronger assumptions.

The contributions of this paper are the following.
\begin{listi}
 \item We study the similarities between Frank-Wolfe and the Triangle algorithm:  
 previous works on the Triangle algorithm \citep[for instance, ][]{Kalantari,Kalantari:2019a} briefly mention FW for CHMP and contrast its iteration mechanism with TA. Here, we show that FW with a stopping criterion based on distance duality (useful in case $p \notin \text{conv}({\cal A})$) can be interpreted as a greedy version of the Triangle algorithm: such version is equivalent to an algorithm attributed to von Neumann \cite[as communicated by][]{dantzig92}.  
Moreover, we also show that a pivot from TA can be interpreted as an inexact solution of Frank-Wolfe subproblem, in the sense of \citep{jaggi}.  
    \item Based on this relationship, we propose 
    variants of Frank-Wolfe that take advantage of the distance duality. 
    \item Relying on distance duality, we devise suitable stopping criteria for the considered first-order methods employed to solve CHMP.
    \item We compare the performance of the new variants with their classic versions and other first-order methods such as projected gradient methods through comprehensive  numerical experiments on CHMP: even though previous works \citep{KalantariComparacao,awasthi2020} 
    have reported some numerical experiments comparing the Triangle Algorithm with Frank-Wolfe, applied to the optimization problem \eqref{prob:quad}, specialized stopping criteria for CHMP were not used for FW. Aiming at a fair comparison, we integrate FW (and the other first-order methods) with specialized stopping criteria.  What is more, as far as we know, none of these works have considered Away-Step Frank-Wolfe (or Frank-Wolfe) applied to \eqref{prob:proj} and Projected Gradient methods applied to \eqref{prob:quad} in the computational experiments. 

    \item We illustrate the usefulness of these algorithms in two potential applications: linear programming feasibility problems and image classification problems.
\end{listi}

For that, this paper is organized as follows. In Section~\ref{sec:triangle} we provide a formal description of the Triangle algorithm along with its main convergence and iteration complexity results. Section~\ref{sec:fw} is devoted to a brief review of the Frank-Wolfe method for minimization of convex smooth functions over convex compact sets and a variant known as Away-Step Frank-Wolfe. In Section~\ref{sec:rel}, we discuss similarities between  Triangle and Frank-Wolfe algorithms applied to CHMP.  In Section~\ref{sec:new}, by using distance duality in FW we arrive at a greedy variant of the Triangle Algorithm. 
Section~\ref{sec:spg} reviews the spectral projected gradient method and discuss its application to CHMP. 
In Section~\ref{sec:stop} we devise specific stopping criteria for Frank-Wolfe and projected gradient methods when applied to the CHMP. 
The numerical experiments reported in Section~\ref{sec:numerics} show that the new variants of FW and Triangle algorithms perform better than their classical counterparts. We report that in randomly generated problems covering different scenarios (as described in Section~\ref{sec:random}) with respect to the geometry of the convex hull and the relative position of the query point. Two potential applications of these algorithms, one in the linear feasibility problem and the other on image classification problem, are detailed in Section~\ref{sec:lpfeas} and Section~\ref{sec:mnist}, respectively. Conclusions are drawn in Section~\ref{sec:end}.

\section{Triangle Algorithm} \label{sec:triangle}

In this section, we formally describe the Triangle algorithm and review the main convergence results from \cite{Kalantari}. We also provide complementary results that shall be useful later.

The Triangle algorithm takes as input a finite set of points ${\A}= \{v_1,\dots, v_n\} \subset \Rm$, a point $p\in \Rm$, and a tolerance $\varepsilon \in (0,1)$. At the $k$-th iteration, for a given point $p_k \in \convA$, the algorithm looks for $v\in \A$ that satisfies $d(v,p)\leq d(v,p_k)$.  This $v$ is called a \emph{pivot}. More specifically, we say that $v$ is a pivot at $p_k$. By defining $R\coloneqq \max\{d(v_i,p)\mid v_i \in \A \}$,  if $d(p_k,p)\leq \varepsilon R$,   we say that $p_k$ is an $\varepsilon$-solution, and the algorithm stops, stating $p \in \convA$. Otherwise, the next iterate $p_{k+1}$ is the closest point to $p$ on the line segment $[p_k, v]$. 
If there exists no pivot at $p_k \in \convA$, that is, when $d(v_j,p)> d(v_j,p_k), \forall j=1,\dots, n$, we say that $p_k$ is a \emph{witness} that $p \notin \convA$. In this case, the orthogonal bisecting hyperplane to the line segment $[p,p_k]$ separates $p$ from $\convA$. 
This iterative scheme is summarized in Algorithm~\ref{alg:ag}.

\begin{algorithm}\small
	\DontPrintSemicolon
	\SetAlgoLined
		\KwData{${\cal A}, p, R, \varepsilon \in (0,1)$} 
		Choose $p_0 \in \argmin \{ d(v_j,p)\mid v_j \in \A \}$ 	\label{step:agchoisep0}  \;

		\For{$k= 0,1,2, \ldots$}{
			\lIf{ $d(p_k,p)< \varepsilon R$}{stop}  \label{step:agcriteriodeparada} 
			\lIf{ $d(v_j,p)> d(v_j,p_k), \forall j=1,\dots, n$}{stop: $p_k$ is a $p$-witness.\label{step:agwitness}}
			Choose $v_j \neq p_k$, such that $d(v_j,p) \leq d(v_j,p_k)$. \label{step:agpivo} \;
		Set $\bar{\gamma}_k \in \argmin \{d(p, (1-\gamma)p_k + \gamma v_j )\mid  {0 \le \gamma \le 1}\}$. 
		\label{step:agbaralpha} \\
		 $p_{k+1} \gets  (1-\bar{\gamma}_k)p_k + \bar{\gamma}_k v_j$. 
		}
	\caption{\textsc{Triangle Algorithm (TA)}} 	\label{alg:ag}
\end{algorithm}

\begin{remark}
\begin{listi}
\item  In step~\ref{step:agchoisep0} of Algorithm~\ref{alg:ag} we could have chosen any $p' \in \convA $ as $p_0$, however starting with a point of $\mathcal{A}$ closest to $p$ has technical reasons that will become clear ahead. 
\item If a $p$-witness is detected at step~\ref{step:agwitness}, according to Theorem~\ref{teo:dd}, we have  that $p \notin \convA$.

\item  If the Triangle Algorithm stops at step~\ref{step:agcriteriodeparada}, we have a point $p_k \in \convA$ whose distance is less than $\varepsilon R$ from $p$, thus we classify $p$ as an element of $\convA$.
\end{listi}
\end{remark}

The next lemma will be useful in analyzing the variations of the Triangle Algorithm. It features equivalent characterizations for a pivot that follow from simple algebraic manipulation. 
\begin{lema}\label{lema:pivo}
	Let $p_k \in \convA$, $v_j \in \mathcal{A}$ and $p \in \Rm$ be given. The following are equivalent:
	\begin{listi}
		\item $\norm{v_j-p} \leq \norm{v_j-p_k};$
		\item $2v_j^T(p_k-p) \leq \norm{p_k}^2- \norm{p}^2;$
		\item $(p_k-p)^T(v_j-p_k) \leq -\frac{1}{2}\norm{p_k-p}^2$;
		\item $(p_k-p)^T(v_j-p) \leq \frac{1}{2}\norm{p_k-p}^2.$
	\end{listi}
\end{lema}

Using this lemma we can establish the following auxiliary result. 
\begin{lema}\label{lem:reduce}
In every iteration $k$ of Algorithm~\ref{alg:ag}, if  $d(p_k,p) > \varepsilon R$ and if there exists a pivot $v_j$, then $\bar \gamma_k$ from step \ref{step:agbaralpha} has a closed formula given by 
\begin{equation}\label{eq:alpha}
\bar{\gamma}_k = -\frac{ (p_k -p)^T(v_j-p_k)} {\norm{v_j-p_k}^2} \in (0,1].
\end{equation}
Moreover, $d(p_{k+1},p) < d(p_k,p)$.
\end{lema}
\begin{proof}
It follows from simple calculations that the unconstrained minimizer of the convex quadratic function   $\phi(\gamma) \coloneqq d(p, (1-\gamma)p_k + \gamma v_j )^2$ required in step~\ref{step:agbaralpha} is given by  \eqref{eq:alpha}. Let us show that $\bar{\gamma}_k \in (0,1]$. From Lemma~\ref{lema:pivo}(iii), we have $$
\bar{\gamma}_k \geq \dfrac{1}{2}\dfrac{\| p_k - p\|^2}{\|v_j - p_k\|^2} > 0,
$$
where the last inequality follows from the hypothesis $d(p_k,p) > \varepsilon R >0$. On the other hand, from Cauchy-Schwarz inequality, we obtain
$$
\bar{\gamma}_k \leq \dfrac{\| p_k - p \|}{\| v_j - p_k \|} \leq \dfrac{\| p_k - p \|}{\| v_j - p \|} \leq 1,
$$
where the second inequality follows from step~\ref{step:agpivo}  and Lemma~\ref{lema:pivo}(i). The third inequality is ensured in the first iteration by step~\ref{step:agchoisep0}. Its validity in the subsequent iterations comes from $d(p_{k+1},p) < d(p_k,p)$ that we prove now. Note that 
\begin{align}
    d(p_{k+1},p)^2 & = \| p_{k+1} - p \|^2 =  \| (1-\bar{\gamma}_k)p_k + \bar{\gamma}_k v_j  - p \|^2 = \| \bar{\gamma}_k (v_j - p_k) + (p_k - p)\|^2  \\
    & = \bar{\gamma}_k^2 \| v_j - p_k \|^2 + \bar{\gamma}_k 2 (p_k-p)^T (v_j - p_k) + \| p_k - p\|^2 < \| p_k - p\|^2 = d(p_k,p)^2, \label{eq:alphaquad} 
\end{align}
where the last inequality follows from the fact that $\bar{\gamma}_k>0$ from \eqref{eq:alpha} is the strict unconstrained minimizer of the quadratic in \eqref{eq:alphaquad}. Therefore, $d(p_{k+1},p) < d(p_k,p)$ and $0 < \bar{\gamma}_k \leq 1$ in every iteration such that $d(p_k,p) > \varepsilon R$ and a pivot $v_j$ exists.
\end{proof}

As a consequence of Lemma~\ref{lem:reduce}, we have that the sequence of distances given by  $\delta_k\coloneqq d(p_k,p)$ is monotonically decreasing. We can indeed quantify such reduction by observing  that the inequalities  $d(p_k,p)\leq d(v,p)\leq d(v,p_k)$ imply that the points  $p_k,p$ and $v$ are non-collinear and thus result in a non-degenerate triangle, as depicted in Figure~\ref{fig:piorsimples-a}, where the angle  $\theta_k \coloneqq \angle pp_kv \in [0,\pi/2)$. Hence, 
we have 
\begin{equation}\label{eq:gapred}
\norm{p_{k+1}-p}^2=(1-\cos^2\theta_k)\norm{p_k-p}^2. 
\end{equation} 
 Now, in view of the pivot characterization (Lemma~\ref{lema:pivo}), we can see that if $\norm{v - p}= r$, $\cos \theta_k$  is minimum when $v$ is over the hyperplane $H \coloneqq \{ y \in \Rm \ \mid \ (p - p_k)^T y = (\| p \|^2 -\| p_k\|^2)/2 \}$.	In this case, $\cos \theta_k = \delta_k / 2r$; see an illustration in Figure~\ref{fig:piorsimples-b}. 

\begin{figure*}
	\centering
  \begin{subfigure}[t]{.48\linewidth}
    \centering
    \includegraphics[width=.8\textwidth]{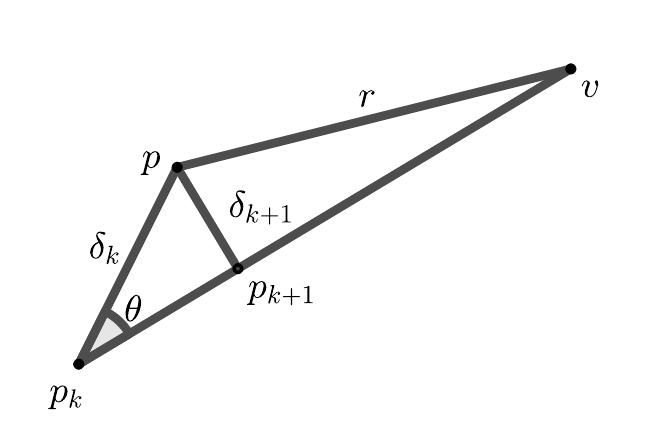}
    \caption{\label{fig:piorsimples-a} Iteration of Triangle Algorithm.}
  \end{subfigure}
  \begin{subfigure}[t]{.48\linewidth}
    \centering
    \includegraphics[width=.8\textwidth]{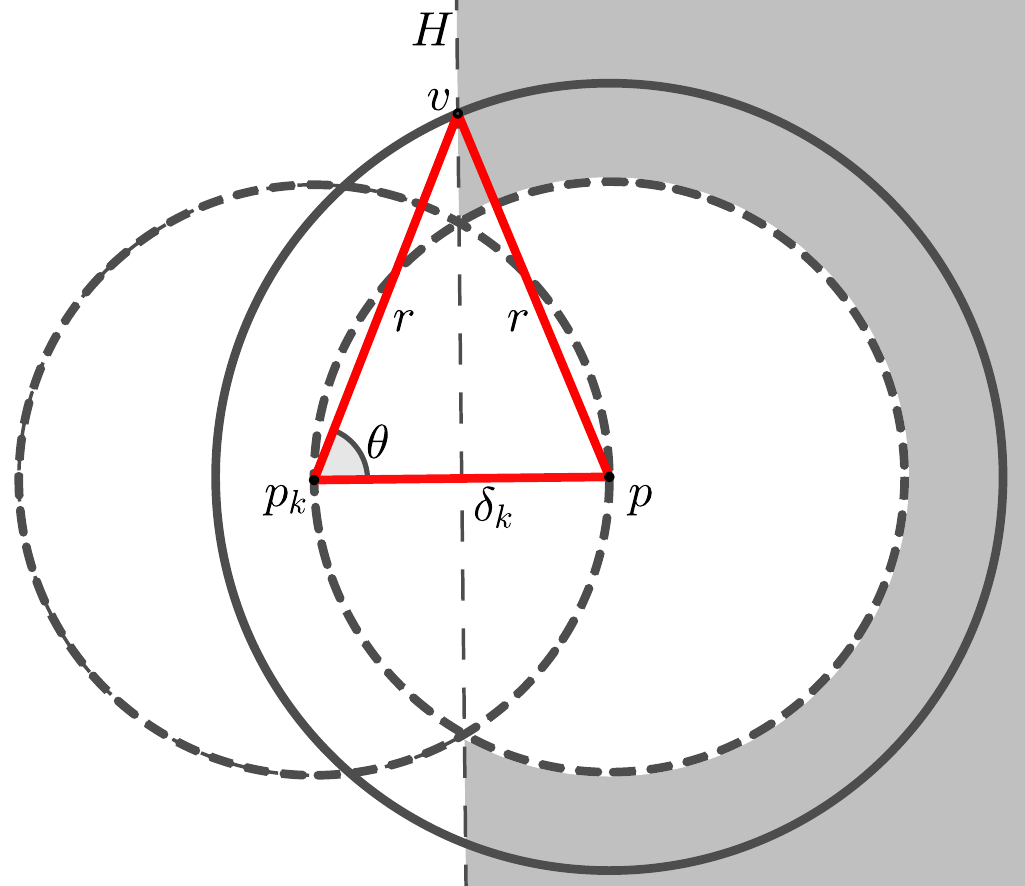}
    \caption{\label{fig:piorsimples-b}Worst case for $\cos \theta$ when $v$ is a pivot. The shaded area represents the region where a pivot can be found.}
  \end{subfigure} 
	\caption{Motivation for the proof of Theorem~\ref{lem:piorcaso}.}	\label{fig:piorsimples}
\end{figure*}

From these remarks and from   \eqref{eq:gapred} we have the next theorem.

\begin{theorem}[{\citet[Theorem 8]{Kalantari}}]
\label{lem:piorcaso}
	Let $p,p',v$ be distinct points in $\mathbb{R}^m$  and suppose that $d(v,p) \leq d(v,p')$. 
	Let $p''$ be the point in the segment $[p',v]$ that is closest to $p$. 
	Define $\delta\coloneqq  d(p',p)$,  $\delta'\coloneqq d(p'',p)$, $r\coloneqq d(v,p)$ and assume  $\delta \leq r$. 
	Then,
	\[\label{eq:piorcaso}
	\delta ' \leq \delta\sqrt{1- \frac{\delta^2}{4r^2}}.
	\]
\end{theorem}

{\begin{remark}\label{remark:complexpivosimples}
	Regarding Algorithm~\ref{alg:ag}, we have 
	\[\label{complexidadeag}
	\delta_{k+1} \leq \delta_k \sqrt{1- \frac{\delta_k^2}{4R^2}} \leq \delta_k \exp\left(-\frac{\delta_k^2}{8R^2}\right),
	\]
	 where the last inequality follows from $1+x \leq \exp(x)$; recall that $R=\max\{d(v_j,p) \mid v_j \in \A\}$.
	We also observe that, as long as $\delta_k > \varepsilon R$, it holds that $\exp\left(-\frac{\delta_k^2}{8R^2}\right) < \exp\left(-\frac{\varepsilon^2}{8}\right) < 1$.
\end{remark}}

Now, we can formally aggregate the complexity bound of  Algorithm~\ref{alg:ag} in the next result. 

	\begin{theorem}[Complexity of TA~{\cite[Theorem~9]{Kalantari}}]\label{teo:complex1}
	Algorithm~\ref{alg:ag}  correctly solves the convex hull membership problem \eqref{eq:chmp} with the following complexity:
		\begin{listi}
			\item If $p\in \convA$, given $\varepsilon>0, p_0\in \convA$, with $\delta_0= d(p,p_0) \leq  \min{ \{d(v_i,p) \mid i=1, \ldots, n\}}$,  
			the maximum number of iterations $K_{\varepsilon}$ to compute a point $p_{\varepsilon}\in \convA$ such that $d(p_{\varepsilon},p)< \varepsilon R$ satisfies
			\begin{equation}\label{eq:complexeps2}
			K_{\varepsilon} \leq \frac{48}{\varepsilon^2} = O(\varepsilon^{-2}).
			\end{equation}
			\item If $p\notin \convA$ the number of iterations $K_{\Delta}$ to compute a
			$p-$witness is such that  
			\begin{eqnarray}\label{eq:geometriaproble}
			K_{\Delta} \leq \frac{48R^2}{\Delta^2} = O\left(\frac{R^2}{\Delta^2}\right),
			\end{eqnarray}
			where  $\Delta = \min\{d(x,p) \mid x\in \convA\}$.
		\end{listi}
	\end{theorem}

We point out that when $p \not\in \convA$  Theorem~\ref{teo:complex1}(ii) states the iteration complexity of TA depending only on the geometry of the problem (namely, the constants $\Delta$ and $R$).
On the other hand, the $O(\varepsilon^{-2})$ iteration complexity in Theorem~\ref{teo:complex1}(i) implies that the convergence rate is only sublinear. In order to improve this complexity result, a more restrictive definition of pivot is presented next.
\begin{defi}[{Strict pivot \cite[Definition 8]{Kalantari}}]\label{def:pivoestrito}
	Let $p\in \Rm$ and $p_k \in \convA$. We say that $v\in \A$ is a \emph{strict pivot}, if  the angle $\angle p_kpv$  between the segments $[p_k,p]$ and $[p,v]$  is such that  $\angle p_kpv \geq \pi/2$.
\end{defi}
\begin{remark}\label{rem:strict}
Clearly, $v \in \A$ is strict pivot for $p_k$ whenever $(p_k - p)^T (v - p) \le 0$, which is a stronger condition than  the one in Lemma~\ref{lema:pivo}(iv).  
{We depict the geometric interpretation of a strict pivot in Figure~\ref{fig:pivoestrito}, where the  hyperplane $\bar{H} $ is defined as $\bar{H} \coloneqq  \{ y \in \Rm \mid (p - p_k)^T (y - p) = 0 \}$. Analogously to Remark~\ref{remark:complexpivosimples}, in the case of $v$ being a strict pivot, we have 
	\[
	\delta_{k+1} = \frac{\delta_k r}{\sqrt{r^2+\delta_k^2}} \leq \frac{\delta_k R}{\sqrt{R^2+\delta_k^2}} \leq \delta_k \sqrt{1- \frac{\delta_k^2}{2R^2}} \leq \delta_k \exp\left({-\frac{\delta_k^2}{4R^2}}\right). 
	\]}
\end{remark}

\begin{figure}
	\centering
	\includegraphics[scale=0.7]{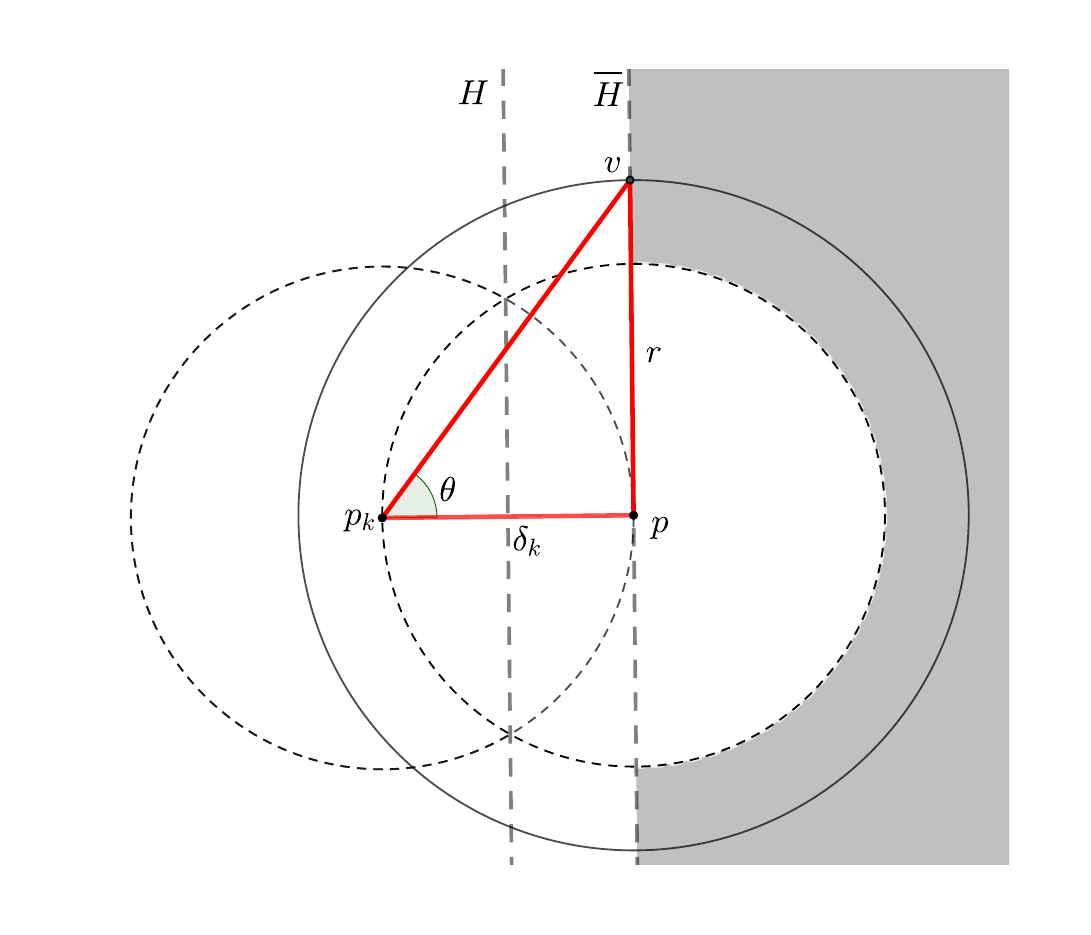}
	\caption{The shaded area represents the region where strict pivots can be found.}\label{fig:pivoestrito}
\end{figure}

Note that a distance duality theorem for strict pivots is also available and is presented in sequel.

\begin{theorem}[Distance duality for a strict pivot~{\cite[Theorem~10]{Kalantari}}]\label{teo:dual2} 
	Assume $p\notin \A$. Then, we have $p \in \convA$ if, and only if,	for each $p_k \in \convA$ there exists a strict pivot $v\in \A$.
\end{theorem}

Similarly to Theorem~\ref{teo:complex1}, the Triangle algorithm using strict pivots also requires  $\mathcal{O}(1/\varepsilon^2)$ iterations to find an $\varepsilon$-solution (albeit the  multiplicative constant is smaller; see \citet[Theorem~11]{Kalantari}), whenever $p$ belongs to $\convA$. Nonetheless, strict pivots shall be useful ahead to showcase a situation where the convergence rate is linear. First, note from \eqref{eq:gapred} that the reduction in $\delta_k = d(p_k, p)$, at each iteration of the Triangle Algorithm, depends on the angle $\theta_k = \angle pp_kv$ (see Figure \ref{fig:piorsimples}). Thus, 
if $\cos \theta_k$ stays bounded away from zero, it is possible to improve the complexity results for TA. 
Indeed, we formalize this fact, which was established in \cite{Kalantari},  with the aid of the next assumption.

\begin{assump}\label{assump:cfatorvisibi}
    There exists a constant $c> 0$ (called \emph{visibility factor})  such that  the following inequality is fulfilled
	\begin{equation}\label{eq:hip1}
		\sin \theta_k \le  \dfrac{1}{\sqrt{1+c}}, \quad \forall p_k \in \convS, 
		\end{equation}
	where $\theta_k = \angle pp_kv$ and $v$ is a  pivot at $p_k$.
\end{assump}

\begin{theorem}[{\citet[Theorem~13]{Kalantari}}]
\label{teo:complexalternative}
	Let $\delta_0= d(p_0,p),$ $p_0 \in \convS$ and assume that Assumption \ref{assump:cfatorvisibi} holds.
\begin{listi}
			\item If $p\in \convA$, then the number of iterations of the Triangle Algorithm to obtain $p_{\varepsilon} \in \convA$ such that $d(p_{\varepsilon},p) \le \varepsilon R$ is 
	\[
	\mathcal{O}\left( \frac{1}{c} \ln{\frac{\delta_0}{\varepsilon R}} \right).
	\]
	\item If $p\notin \convA$ the number of iterations of the Triangle Algorithm to obtain a $p$-witness is
	\[
	\mathcal{O}\left( \frac{1}{c} \ln{\frac{\delta_0}{\Delta}}\right).
	\]
	\end{listi}
\end{theorem}

Iteration complexities of Theorem~\ref{teo:complexalternative} are improvements in comparison with those of Theorem~\ref{teo:complex1}, as long as the visibility factor $c$ is bounded away from zero. As we shall see ahead, this is the case, for example, when $p$ belongs to the relative interior of $\convA$. Nevertheless, in other cases,  $c$ can be very close to $0$ and, as a consequence, the complexities of Theorem~\ref{teo:complexalternative} may be way worse than those of  Theorem~\ref{teo:complex1}. When this happens, TA may experience a zigzagging phenomenon as illustrated by the following example. 
\begin{example}\label{exe:quadrado}
Consider $\A = \{ v_0,v_1,v_2,v_3,v_4\} \subset \R^2$ where $\{v_0, v_1, v_2, v_3\}$ are vertices of the unit square and $v_4$ is an interior point. Assume that the point of $\A$  closest to the query point $p$ is $v_4$ (which was generated by shifting the center of the square a bit to the right). We analyze two instances depicted in Figure~\ref{fig:example-unitsquare}. In Figure~\ref{fig:example-unitsquare-a}, we have the case where the query point $p$ is the midpoint of an edge of $\convA$ whereas Figure~\ref{fig:example-unitsquare-b} shows the case where $p$  is moved to a position outside the square. In both cases,  the zigzagging behavior of TA shows up. We only plot the first $60$ iterations of TA, in each picture: it takes more than a million iterations to achieve $\|p_k - p \| \leq \varepsilon R$, with $\varepsilon=10^{-4}$ in case (a) and 81 iterations to find a witness in case (b). For $p\in \convA$, we have $\delta_0= 0.40$ and $R = 1.12$. 
We monitored the sequence $(\sin \theta_k)_{k\in \N}$ based on TA iterations in order to estimate an upper bound for \eqref{eq:hip1}. We notice that the values approach $1$ as $k$ grows; thus, for the points we computed, if the visibility factor $c$ exists, it is not greater  than $10^{-7}$. When $p\notin \convA$, we get $\delta_0= 0.45$ and $\Delta = 0.05$. Again, tracking $(\sin \theta_k)_{k\in N}$ allowed us to provide an upper bound for the visibility constant $c$ of $0.0044$. 
\end{example}

\begin{figure*}
\centering
  \begin{subfigure}[t]{.48\linewidth}
    \centering
    \includegraphics[height=.27\textheight]{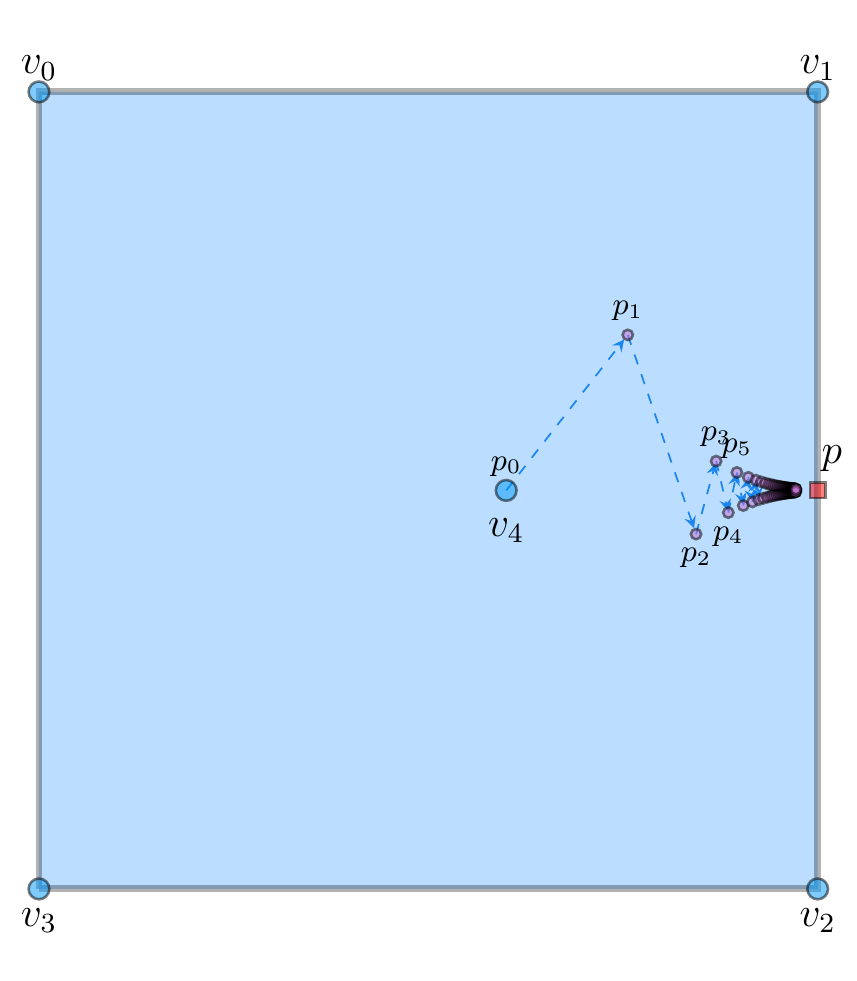}
    \caption{\label{fig:example-unitsquare-a} $p\in \convA$ over the boundary.}
  \end{subfigure}
  \begin{subfigure}[t]{.48\linewidth}
    \centering
    \includegraphics[height=.27\textheight]{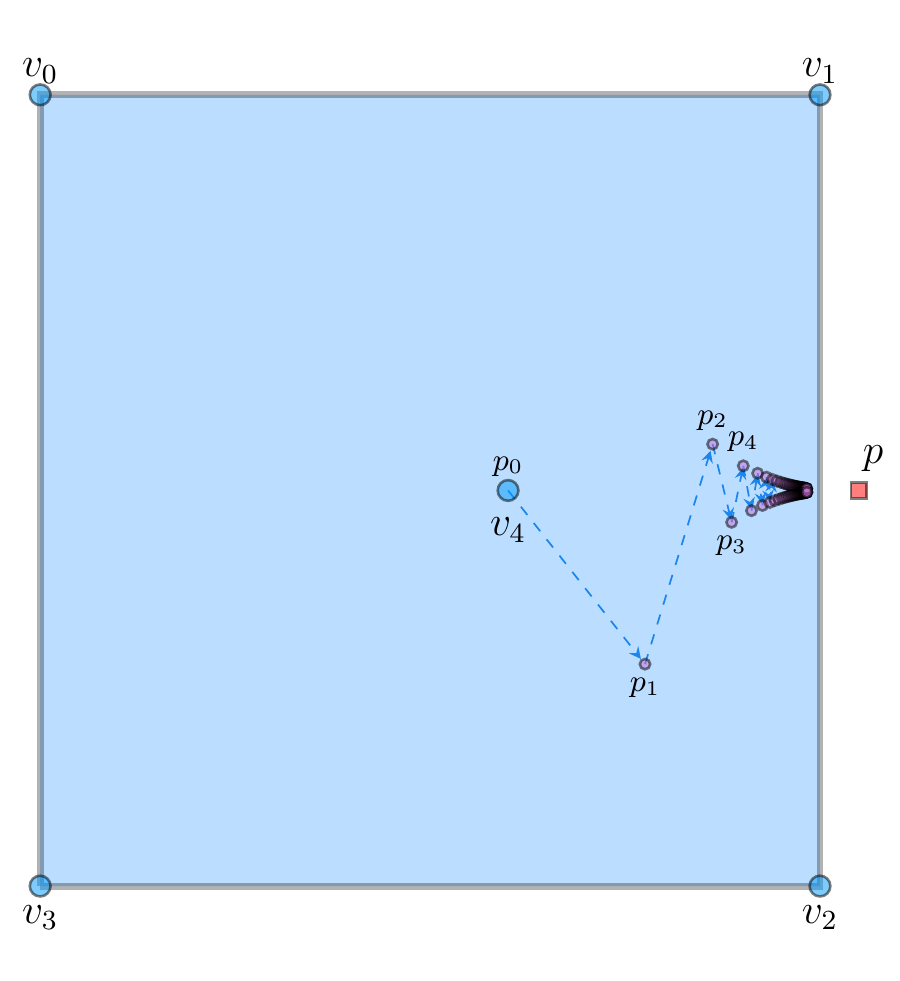}
    \caption{\label{fig:example-unitsquare-b} $p\notin \convA$ }
  \end{subfigure} 
	\caption{Unit squares of Example~\ref{exe:quadrado} with the first $60$ TA iterates.}	\label{fig:example-unitsquare}
\end{figure*}

Under the additional assumption that $p$ lies in the relative interior of $\convA$, and under a suitable choice of pivots in each iteration, it is possible to show that \eqref{eq:hip1} is satisfied with a constant $c$ whose lower bound depends only on the geometry of the problem.


\begin{assump}\label{assump:pnointerior}
$B_\rho(p)$ is contained in the relative interior of $\convA$. 
\end{assump}

Let $p_k$ be the current iterate such that $d(p,p_k) > \varepsilon R$. 
{
If Assumption~\ref{assump:pnointerior} holds for a given $\rho>0$,  then there exists a $q\in \convA$ that belongs to  the extension of the segment $[p_{k},p]$ (see Figure~\ref{fig:teoremapivoforte-a}) with $d(q,p)=\rho$. 
Now, if there exists a strict pivot $v$ with the property $(p - q)^T(v - q) \leq 0$, then
\begin{align}
\rho R \cos \angle qpv & \geq \| p - q \| \| v - p\| \cos \angle qpv = (q-p)^T (v-p)  \\
\ &= (q-p)^T (v- q + q - p) = \|q - p\|^2 - (p - q)^T(v - q) \geq \rho^2
\end{align}
which implies that 
\begin{equation}\label{eq:boundsintheta}
\sin \theta_k \leq \sin \angle qpv \leq \sqrt{1 - \rho^2/R^2} \leq \frac{1}{\sqrt{1 + \rho^2/R^2}}.
\end{equation}
}
{
We point out that \eqref{eq:boundsintheta} is necessary in the proof of the next theorem \citep[Theorem 14]{Kalantari}, and thus we consider the following assumption. 
}


\begin{assump}\label{assump:strong_pivot}
    Let  $q = p + \tau (p - p_k)$ for $\tau>0$ such that $q \in \convA$ and $d(q,p)=\rho$. 
    The Triangle Algorithm uses, at each iteration,  a strict pivot $v$ satisfying 
	\begin{equation}\label{eq:strongpivot}
		(p - q)^T(v - q) \leq 0.
	 \end{equation}
\end{assump}

{\begin{remark}
A strict pivot $v$ satisfying \eqref{eq:strongpivot} is represented in Figure~\ref{fig:teoremapivoforte-a}. 
We remark that a strict pivot does not necessarily satisfy condition~\eqref{eq:strongpivot}. 
In Figure~\ref{fig:teoremapivoforte-b}, the gray area represents the region where a strict pivot can be found. 
Note that $v_1,v_2$ and $v_3$ are strict pivots, but $v_2$ does not satisfy \eqref{eq:strongpivot} 
because it is in the \emph{wrong} side of the hyperplane $\bar{H}_q \coloneqq  \{ y \in \Rm \mid (p - q)^T (y - q) = 0 \}$. 
In fact, only strict pivots belonging to the dashed area satisfy the inequality \eqref{eq:strongpivot}.
\end{remark}}

\begin{figure*}
	\centering
	
	\begin{subfigure}[t]{.48\linewidth}
		\centering	
		\includegraphics[width=0.5\textwidth]{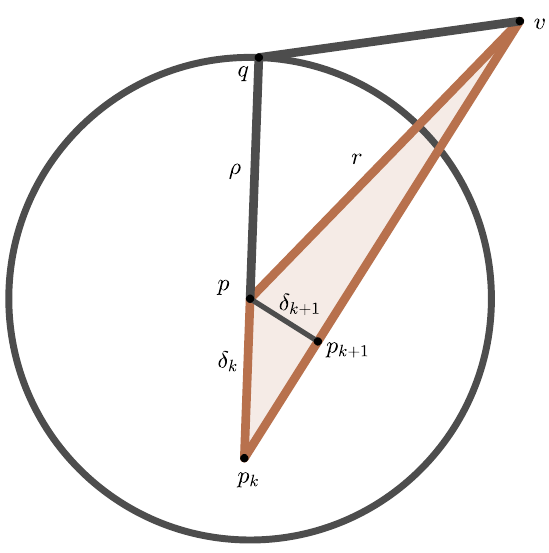}
		\caption{A strict pivot $v$ that satisfies  condition \eqref{eq:strongpivot}.}
		\label{fig:teoremapivoforte-a}
	\end{subfigure}
	\begin{subfigure}[t]{.48\linewidth}
    \centering
	\includegraphics[width=0.6\textwidth]{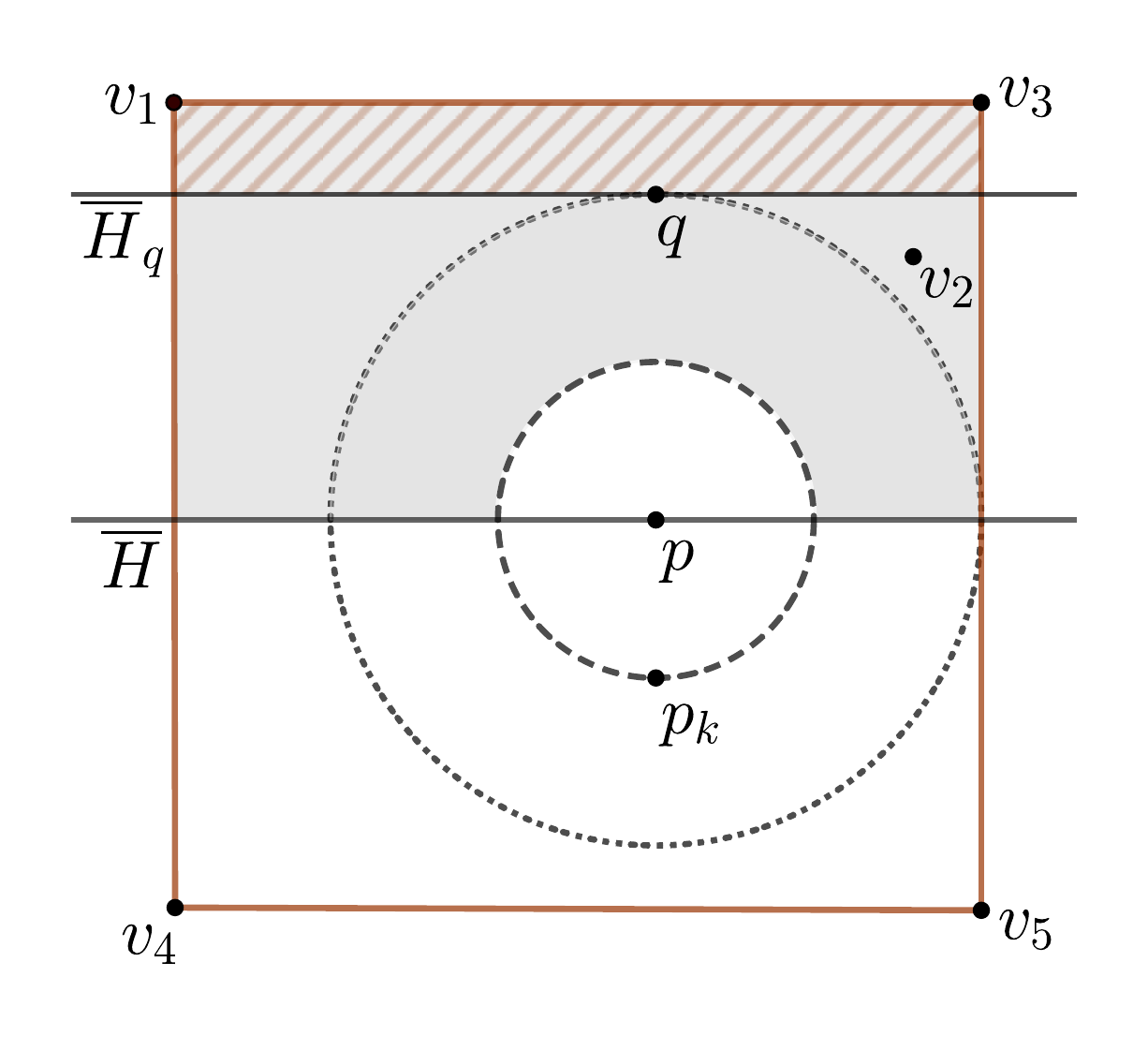}
	    \caption{\label{fig:teoremapivoforte-b} Shaded area is where strict pivots can be found while in the dashed area are the strict pivots for which \eqref{eq:strongpivot} holds.}
  \end{subfigure}

	  \caption{Representations of strict pivots.}
	\label{fig:teoremapivoforte}
\end{figure*}

\begin{theorem}\label{teo:complexp}
      Let  $R=\max\{d(v_i,v_j)\mid v_i,v_j\in \A\}$, $p_0\in \convA$, $\delta_0=d(p_0,p)$ and suppose that Assumptions \ref{assump:pnointerior} and \ref{assump:strong_pivot} hold. For any given $\varepsilon \in (0,1),$ the number of iterations that Algorithm~\ref{alg:ag} takes to compute $p_\varepsilon \in \convA$, such that $d(p,p_\varepsilon) \leq \varepsilon R$, is 
    $$\mathcal{O} \left( \frac{R^2}{\rho^2} \ln{\frac{\delta_0}{\varepsilon R}} \right).$$
\end{theorem}
{
\begin{proof}
It can be found in \citep[Theorem~14]{Kalantari}. Even though Assumption~\ref{assump:strong_pivot} is not explicit in the statement of Kalantari's theorem, 
it becomes necessary in the proof. 
\end{proof}
}

This result implies the linear convergence of Algorithm~\ref{alg:ag}, using strict pivots satisfying Assumption~\ref{assump:strong_pivot}, when $p \in B_{\rho}(p) \subset \operatorname{relint}(\convA)$.
In Section~\ref{sec:new} we shall see a variant of Algorithm~\ref{alg:ag} for which Assumption~\ref{assump:strong_pivot} holds in every iteration.

\section{Frank-Wolfe for CHMP}\label{sec:fw}
We kick-start this section presenting the Frank-Wolfe algorithm (FW) for convex constrained optimization. We then explore one of its variants called Away Step Frank-Wolfe (ASFW) for which a linear convergence rate can be obtained. Finally, we formally discuss connections between FW and the Triangle Algorithm when applied to the convex hull membership problem.

\subsection{Frank-Wolfe algorithm}
Let $f: \Rm \rightarrow \R$ be a continuously differentiable convex function whose gradient is $L$-Lipschitz and $C \subset \Rm$ a non-empty, convex and compact set. Following \cite{lacoste2015global}, we  assume that $C \coloneqq \convA$, where $\A\subset \Rm$ is a finite set of points. Consider the following optimization problem
\begin{equation}\label{prob:minf}
\begin{aligned}
\min_{x \in C} & \quad f(x).\\
\end{aligned}
\end{equation}

The Frank-Wolfe algorithm~\citep{fw1956} (also known as Conditional Gradient) can be used to solve  problem~\eqref{prob:minf}, and it is summarized in Algorithm~\ref{alg:fw}. 
FW is a first-order iterative method for smooth convex optimization which aims to solve \eqref{prob:minf} through a sequence of linearized subproblems. 

\begin{algorithm}
\small
    \DontPrintSemicolon
	\SetAlgoLined
	\KwData{$x_0 \in \A,\epsilon>0$} 
	\For{$k= 0,1,2, \ldots$}{
	Set $\x_k \in \argmin_{x\in \A}\nabla f(x_k)^T(x-x_k)$ \label{passosubprobfw} \;
	 $d_k\gets  \x_k -x_k$. \; 
	
		\lIf{$\nabla f(x_k)^T d_k \geq -\epsilon$} {stop} \label{step:FWstop}
		Choose $\gamma_k \in (0,1]$. \;
		$x_{k+1}\gets  x_k+ \gamma_k d_k$
	}	
	\caption{Frank-Wolfe (FW)}
	\label{alg:fw}
\end{algorithm}

\begin{remark}
{One of the main advantages of FW is that, in many applications, the subproblem of Step~\ref{passosubprobfw} either admits a closed form solution or is  much cheaper in terms of computational cost than the quadratic subproblems required by projected gradient or proximal methods, leading to good scalability \citep{jaggi}.  Moreover, it keeps the iterates as convex combination of few points of $\A$, which is an interesting feature in problems requiring sparse solutions \citep{Clarkson}.}
\end{remark}
\begin{remark}
{Let $\xs$ be a solution of \eqref{prob:minf}. Since $f$ is convex and differentiable, and by the definition of Step~\ref{passosubprobfw}, we have
\[
h(x_k)\coloneqq f(x_k)-f(\xs) \leq -\nabla f(x_k)^T(\xs-x_k)\leq -\min_{x \in \A}\nabla f(x_k)^T(x-x_k)\eqqcolon g(x_k),
\] 
that is, $g(x_k)$ provides an upper bound for the primal gap $h(x_k)$, therefore justifying the stopping criterion in Step~\ref{step:FWstop}.}
\end{remark}

{In the following, we briefly review the convergence analysis for Algorithm~\ref{alg:fw}, 
so that the comparison of its convergence rate with those of other algorithms can be clear. 
For further details, see \citep{jaggi,lacoste2015global,beck2017linearly}.}

Since $\nabla f$ is Lipschitz with constant $L>0$, we get that 
\[
f(x_{k+1}) = f(x_k + \gamma_k d_k)  \leq f(x_k) + \gamma_k  \nabla f(x_k)^T d_k  + \dfrac{L}{2} \gamma_k^2 \| d_k \|^2 \leq f(x_k) - \gamma_k g(x_k) + \gamma_k^2\frac{L}{2}D^2,
\]
where $D \coloneqq \max \{ \| x - y \|\mid x,y \in C \}$ is the diameter of $C$. From this inequality and the definition of $h(x_k)$, we obtain
$$
h(x_{k+1}) \leq (1 - \gamma_k)h(x_k) + \gamma_k^2\frac{L}{2}D^2.
$$
Then, it is possible to derive  \citep[see Theorem~1 in][]{jaggi} that, for $\gamma_k = \dfrac{2}{k+2}$, 
\begin{equation} \label{eq:sublinearconvFW}
h(x_k) \leq \frac{2LD^2}{k+2},
\end{equation}  
from which the sublinear convergence of FW is established\footnote{If $\bar{\gamma}_k$ is the exact line-search step-size, the previous analysis is still valid because $f(x_k+\bar{\gamma}_k d_k) \leq f(x_k+\frac{2}{k+2}d_k).$}. Unfortunately, this $\mathcal{O}(1/k)$ complexity is tight, as showed by \cite{canon1968}. 

Nevertheless, under strong convexity of $f$, it is possible to achieve linear convergence in certain cases. In fact, if we suppose the existence of $\mu>0$ such that, for all $x,y$, 
$$
f(y) \geq f(x) +  \nabla f(x)^T (y - x)  + \dfrac{\mu}{2}\|y - x \|^2,
$$
and define $e_k \coloneqq  \xs - x_k$, we have
\begin{align}
f(x_k + \lambda e_k) - f(x_k) & \geq \lambda \nabla f(x_k)^T e_k + \dfrac{\lambda^2}{2} \mu \| e_k \|^2 \geq -\dfrac{ \left( \nabla f(x_k)^T \hat{e}_k  \right)^2}{2\mu}, \quad \forall \lambda \in \R,
\end{align}
where $\hat{e}_k \coloneqq  e_k/\|e_k\|$. 
In particular, for $\lambda=1$, we get 
\begin{equation}\label{eq:minush}
    - h(x_k) \geq - \dfrac{\left( \nabla f(x_k)^T \hat{e}_k  \right)^2}{2\mu}.
\end{equation}

On the other hand, using again that $\nabla f$ is $L$-Lipschitz we obtain 
\begin{align}
f(x_k + \gamma_k d_k) - f(x_k) & \leq \gamma_k  \nabla f(x_k)^T d_k  + \dfrac{\gamma_k^2}{2} L \| d_k \|^2. 
\end{align}
By taking $\gamma_k \coloneqq  \argmin_{\gamma \in (0,\gamma_{\max}]} f(x_k + \gamma d_k)$ in Algorithm~\ref{alg:fw},  where $\gamma_{\max} \in (0,1]$ is the maximum allowed step-size, if $\gamma_k^{\star} \coloneqq  -\dfrac{ \nabla f(x_k)^T d_k  }{L\|d_k\|^2} \leq \gamma_{\max}$,  then $\gamma_k = \gamma_k^{\star}$ and 
$$
\gamma_k  \nabla f(x_k)^T d_k  + \dfrac{\gamma_k^2}{2} L \| d_k \|^2 = -\dfrac{  \left(\nabla f(x_k)^T\hat{d}_k \right) ^2}{2L},
$$
where $\hat{d}_k \coloneqq  d_k/\|d_k\|$. 
Thus, from the last inequality, we obtain 
\begin{equation}\label{eq:diffh}
h(x_k) - h(x_{k+1}) \geq \dfrac{\left(\nabla f(x_k)^T\hat{d}_k \right) ^2}{2L}.
\end{equation}
This result, along with \eqref{eq:minush},  yields 
\begin{equation}\label{eq:thekey}
h(x_{k+1}) \leq \left[ 1-\frac{\mu}{L}\left(\dfrac{ \nabla f(x_k)^T \hat{d}_k }{ \nabla f(x_k)^T \hat{e}_k } \right)^2 \right] h(x_k).
\end{equation}

From \eqref{eq:thekey}, it is clear that the linear convergence will depend on our ability to bound the ratio between $\nabla f(x_k)^T \hat{d}_k$ 
and $\nabla f(x_k)^T \hat{e}_k$. This can be done in certain situations for the classic FW or variants of FW using different search directions $d_k$. 
For instance, \citet[Theorem~2]{guelat1986some}  show that if $\xs$ is in the relative interior of $C$, with distance to the boundary of $C$ at least $\rho>0$  then, for $k$ sufficiently large, $  \nabla f(x_k)^T d_k  \leq -(\rho/2) \norm{\nabla f(x_k)}$ and $\gamma^{\star}_k \leq 1$, which implies that 
\begin{equation}\label{eq:casoespecialfwinteriorrelativo}
    h(x_{k+1})\leq \left[  1 - \dfrac{\mu\rho^2}{LD^2} \right] h(x_k).
\end{equation} 
However, for the classic Frank-Wolfe, if $\xs$ lies on the boundary of $C$, then $ \nabla f(x_k)^T \hat{d}_k$ can get arbitrarily close to zero and the method may exhibit a zigzagging behavior ({see Figure~\ref{fig:fw-a}}).

\begin{figure*}
	\centering
	\begin{subfigure}[t]{.49\linewidth}
		\centering
		\includegraphics[height=.17\textheight]{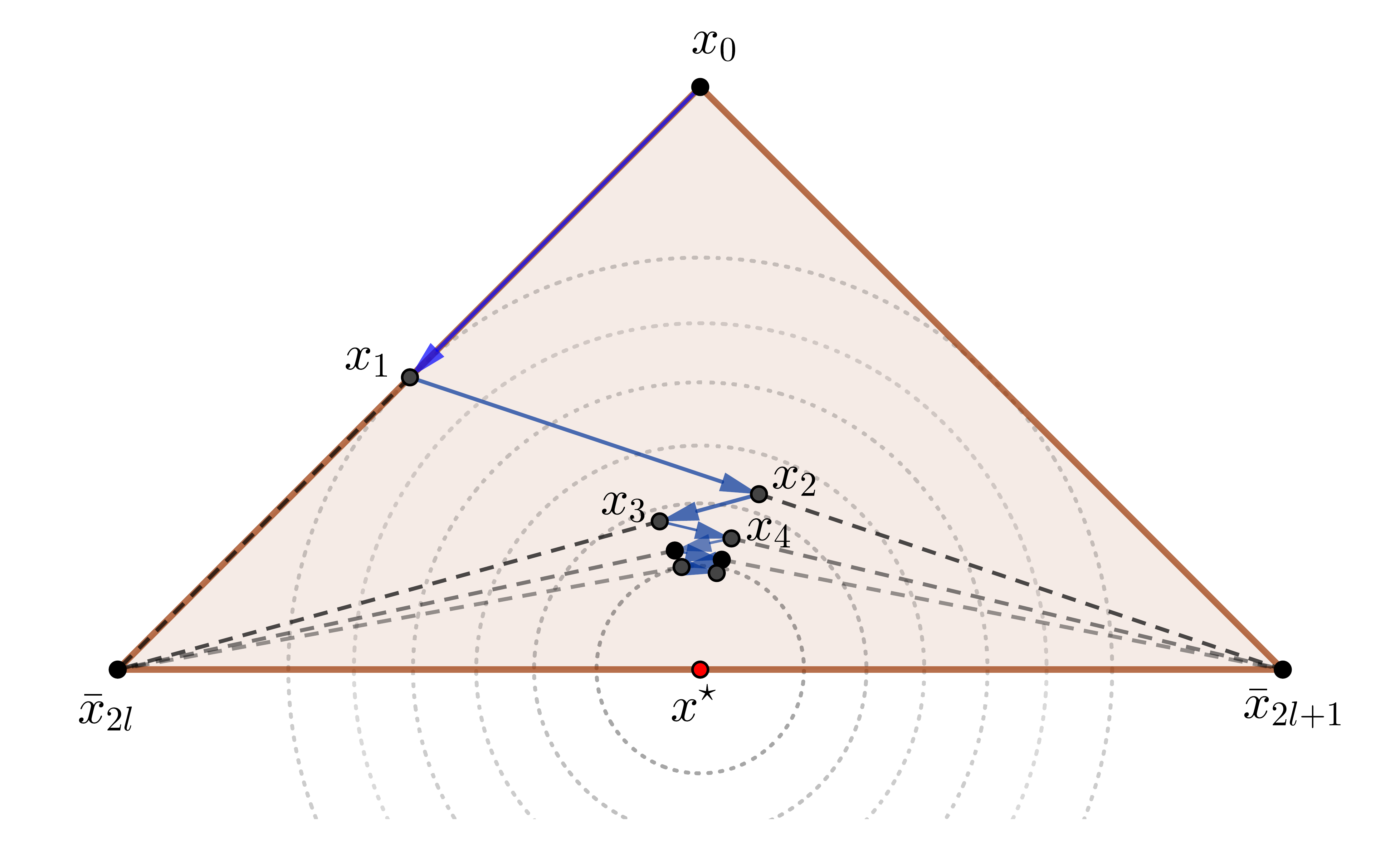}
		\caption{\label{fig:fw-a} Zigzagging behavior of Frank-Wolfe.}
	\end{subfigure}
	\begin{subfigure}[t]{.49\linewidth}
		\centering
		\includegraphics[height=.17\textheight]{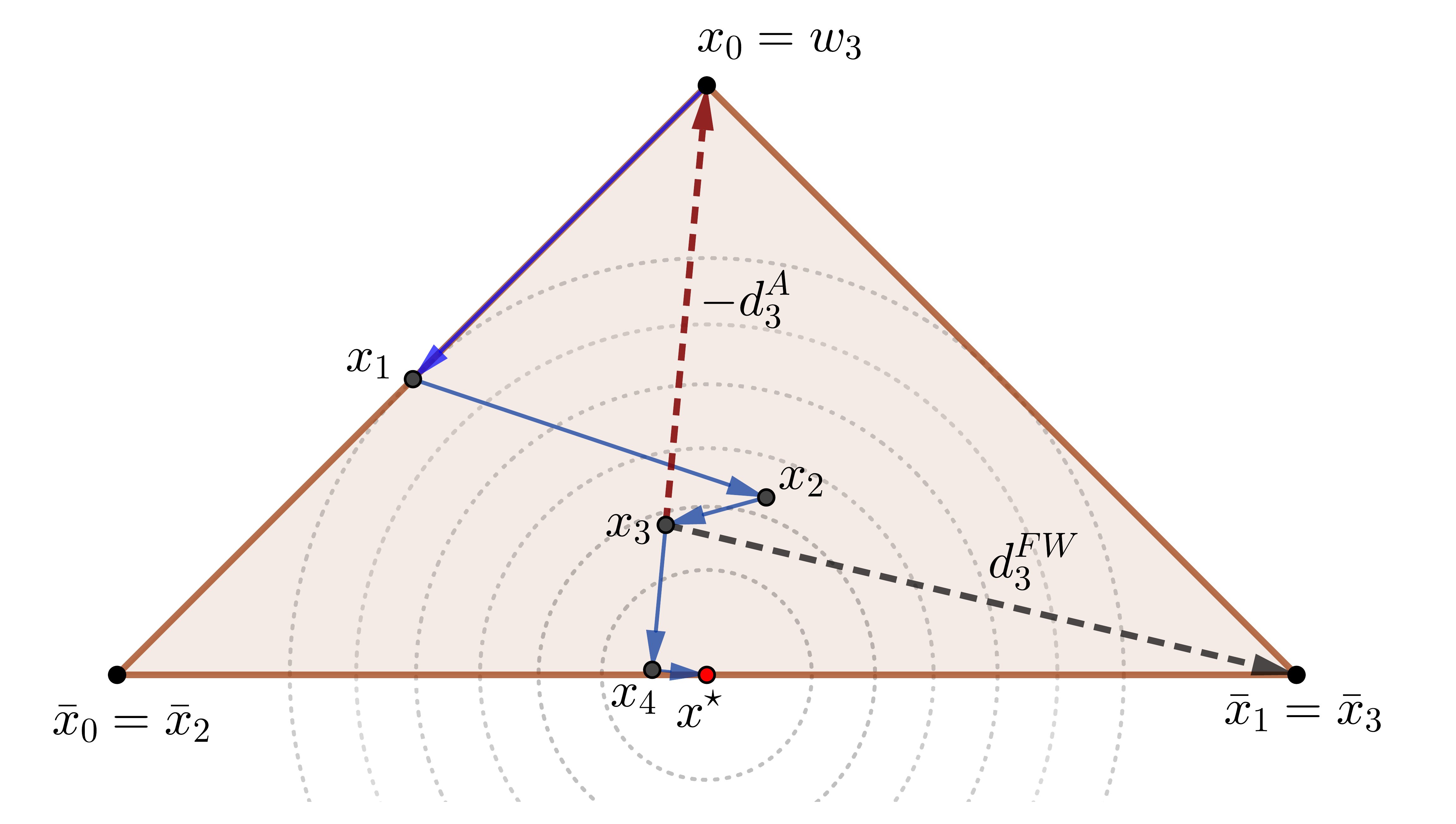}
		\caption{\label{fig:asfw-b} Iterations of Away Step Frank-Wolfe.}
	\end{subfigure} 
	\caption{Example of FW and ASFW iterations.}	\label{fig:geofwAsfw}
\end{figure*}

In next section we discuss a Frank-Wolfe variant, called Away-Step Frank-Wolfe (ASFW), which exhibits linear convergence independently of the location of $\xs \in C$. 

\subsection{Away-Step Frank-Wolfe}\label{sec:asfw}
We proceed to briefly reviewing ASFW algorithm presented in  \citet[Algorithm~1]{lacoste2015global}, which is summarized in Algorithm~\ref{alg:asfw}, and we provide a discussion on its main properties. Let $V$ be the set of extreme points of $C\coloneqq \convA$ and $n = |\A|$. In each iteration $k$ of Algorithm~\ref{alg:asfw}, the iteration point $x_k$ is a convex combination of elements of $\A$: $x_{k} = \sum_{v \in \A} \alpha_v^k v$. Let $\alpha^k \in \Delta_{n}$ be a vector containing the coefficients of such combination and define $U^k \coloneqq \{ v \in \A \mid  \alpha_v^k > 0 \}$. Without loss of generality, let us suppose that the elements of $\A$ are ordered in such a way that $x_0$ is the first one: then $\alpha^0 = e_1$, the first canonical vector of $\Rn$.

\begin{algorithm}
\small
    \DontPrintSemicolon
	\SetAlgoLined
	\KwData{$x_0 \in \mathcal{A},\epsilon>0, U^0= \{x_0\}, \alpha^0= e_1 \in \Delta _{n}$} 
	\For{$k= 0,1,2, \ldots$}{
		Set $\x_k \in \argmin_{x\in \mathcal{A}}\nabla f(x_k)^T(x-x_k) $ \label{stepalg:asfw-xbar} \;
		 $d_{k}^{FW}\gets \x_k - x_k$\; 
		Set $w_k \in  \argmax_{x\in U^k} \nabla f(x_k)^T(x-x_k) $ \;
		 $d_{k}^{A}\gets x_k - w_k$\; \label{step:awaydirection} 
		\lIf{$\nabla f(x_k)^T d_{k}^{FW} \geq -\epsilon$}{stop and return $x_k$} \label{step:criterioparadaASFW}
		\eIf{$\nabla f(x_k)^T d_{k}^{FW} \leq \nabla f(x_k)^T d_{k}^{A}$} {$d_k \gets  d_k^{FW}$, $\gamma_{\max}\gets 1$}		
		{$d_k\gets d_k^{A}$, $\gamma_{\max}\gets \alpha_{w_k}/(1-\alpha_{w_k})$ \label{step:gammamaxaway} 
		}
		
		Set $\gamma_k \in  \argmin_{\gamma\in [0,\gamma_{\max}]}f(x_k+ \gamma d_k)$
		
		$x_{k+1}\gets  x_k+\gamma_k d_k$.  
		
		Update $\alpha^{k+1}$ accordingly.
		
		$U^{k+1}\gets  \{v \in \mathcal{A}\mid  \alpha_v^{k+1}>0\}$.
	}
	\caption{Away-Step Frank-Wolfe (ASFW)}\label{alg:asfw}
\end{algorithm}

\begin{remark}
  In each iteration of ASFW, two search directions are computed: the FW direction $d_k^{FW}$ and an ``away direction'' $d_k^A$. 
If $\nabla f(x_k)^T d_{k}^{FW} \leq \nabla f(x_k)^T d_{k}^{A}$, the classic FW direction is taken and the maximum allowed step-size is  $\gamma_{\max}=1$. Otherwise, $\gamma_{\max}=\alpha_{w_k}/(1-\alpha_{w_k})$ gives a conservative upper bound for the step-size in the direction $d_k^A$ in order to ensure $x_{k+1} \in C = \convA$; see \cite{lacoste2015global} for details. 
\end{remark}

  \begin{remark}\label{rem:asfw}
    {The purpose of away-steps is to reduce the weight of elements in the current active set $U^k$ that contribute to increase the objective function (see Figure~\ref{fig:asfw-b})}. 
  When $d_k = d_k^A$ and $\gamma_k = \gamma_{\max}$, then $U^{k+1} = U^k \setminus \{w_k\}$. If the step-size in the away direction is smaller than $\gamma_{\max}$, $U_k$ remains unchanged, but the weight of $w_k$ is reduced. 
When $d_k = d_k^{FW}$, if $\gamma_k=1$, then $U^{k+1} = \{ \bar{x}_k \}$ and $U^{k+1} = U^k \cup \{\bar{x}_k\}$ otherwise, where $\bar{x}_k$ is given in Step \ref{stepalg:asfw-xbar} of Algorithm \ref{alg:asfw}. This means that the cardinality of $U^k$ either decreases, or increases by one in each iteration. Therefore, if $s^k$ is the number of iterations, until iteration $k$, in which $d_k = d_k^A$ and $\gamma_k = \gamma_{\max}$, and $\ell^k$ is the number of iterations in which an element was added to $U^k$, we have $s^k + \ell^k \leq k-1$, $s^k \leq \ell^k$ and thus $s^k \leq (k-1)/2$.
  \end{remark}
  
  {In the sequel, we present the global linear convergence of Algorithm~\ref{alg:asfw}.} 
  Let us define $d_k^P = \bar{x}_k - w_k$, which is known in the literature as \emph{pairwise direction}. Clearly,  
  \begin{equation}\label{eq:dirp}
      - \nabla f(x_k)^T d_k^P \leq -2  \nabla f(x_k)^T d_k .
  \end{equation}
  Moreover, for $d_k^P$ we can establish a bound on the ratio $(\nabla f(x_k)^T \hat{d}_k^P)/( \nabla f(x_k)^T \hat{e}_k) $. For that, we consider the next lemma, which holds if the convex set  $C$ is  a compact polyhedron, that is,  $C  \coloneqq \{ x \in \Rm \mid  a_i^T x \leq b_i, i=1,\ldots, \ell \}$. 
In the following, for $U\subset \Rm$, define $\ISet_C(U)$ the set of active constraints of the polyhedron $C$ for points in $U$, \emph{i.e}, $\ISet_C(U)\coloneqq  \{ i \in\{1,\ldots, \ell\} \mid  a_i^T u  = b_i, \forall u \in U \}$.
  
  \begin{lema}[{\citet[Lemma~3.1]{beck2017linearly}}]\label{lem:bslemma} Let  $C  \coloneqq \{ x \in \Rm \mid a_i^T x \leq b_i, i=1,\ldots, \ell \}$ be a compact polyhedron with $V$ being the set of its  extreme points, $c \in \Rm$, and consider $U \subset V$. If there exists $ z \in \Rm$ such that  $a_i^T z \leq 0$, for all $i\in \ISet_C(U)$, and $ c^T z  > 0$, then
      \[
      \max_{p \in V, u \in U}  c^T (p - u) \geq \dfrac{\Omega_C}{m+1} \dfrac{ c^T z }{\|z\|},
      \]
      where \[\Omega_C \coloneqq \displaystyle \min_{v\in V,  \,i\in \{1,\ldots, \ell\}\mid b_i> a_i^T v}  \dfrac{b_i -  a_i^T v}{\|a_i\|}.\]
  \end{lema}
Lemma \ref{lem:bslemma} holds for the case where $C \coloneqq \convA$, as $\convA$ is a polyhedron,  even though we usually do not know explicitly the linear inequalities that define $\convA$. Thus, if we assume that $\bar{x}_k \in V \subset \A$ in Algorithm~\ref{alg:asfw}, we may write 
  \[
  - \nabla f(x_k)^T d_k^P  = - \nabla f(x_k)^T( \bar{x}_k - w_k)  = \max_{p \in V, u \in U^k}  - \nabla f(x_k)^T (p - u) , 
  \]
  and for $z \coloneqq e_k = \xs - x_k$ and $c\coloneqq -\nabla f(x_k)$, we can apply Lemma~\ref{lem:bslemma} to obtain 
\[  -2  \nabla f(x_k)^T d_k  \geq -  \nabla f(x_k)^T d_k^P  \geq \dfrac{\Omega_C}{m+1} \left( -\nabla f(x_k) \right)^T \hat{e}_k  > 0,\]
  which shows that, in each iteration of ASFW for which $\gamma_k^{\star} \leq  \gamma_{\max}$, we have
 \[ h(x_{k+1}) \leq \left[ 1-\frac{\mu}{4L}\left( \dfrac{\Omega_C}{D(m+1)} \right)^2 \right] h(x_k).\]
  When $d_k=d_k^A$ and $\gamma_k^{\star} > \gamma_{\max}$, we can only ensure that $h(x_{k+1}) \leq h(x_k)$. However,  as discussed before (see Remark~\ref{rem:asfw}), this type of iteration cannot occur too often. Therefore, for the ASFW we have
  \begin{equation}\label{eq:convasfw}
      h(x_k) \leq h(x_0) \left[ 1-\frac{\mu}{4L}\left( \dfrac{\Omega_C}{D(m+1)} \right)^2 \right]^{(k-1)/2}, \quad \forall k\in \N,
  \end{equation}
  which implies global linear convergence (since $\Omega_C \leq D$ and $\mu \leq L$). We summarize the above discussion, which combines the analysis of \cite{lacoste2015global} with Lemma~\ref{lem:bslemma}, in the following result.
  
  \begin{theorem}\label{teo:linear_asfw}
       Let $f$ be $\mu$-strongly convex, $\nabla f$ $L$-Lipschitz, $C$ a compact polyhedral set (with diameter $D$) and let $\fstar$ be the optimal value for \eqref{prob:minf}. 
       Let $( x_k )_{k\in \mathbb{N}}$ be the sequence generated by ASFW for problem \eqref{prob:minf} and assume that $\{\bar{x}_k\} \subset V$. 
       Then, for all $k \geq 1$
       \[
       f(x_k) - \fstar \leq (1 - \eta)^{(k-1)/2} (f(x_0) - \fstar),
       \]
       where $\eta\coloneqq \frac{\mu}{4L}\left( \dfrac{\Omega_C}{D(m+1)} \right)^2$, with $\Omega_C$ from Lemma~\ref{lem:bslemma}. 
  \end{theorem}
  
\begin{remark}
{Note that, for problem  \eqref{prob:proj}, $f \equiv \Psi$, and such function satisfies the assumptions in Theorem~\ref{teo:linear_asfw}, with $\mu = L = 1$.}
\end{remark}

\subsection{Frank-Wolfe similarities with the Triangle Algorithm}\label{sec:rel}

If we apply the Frank-Wolfe algorithm (Algorithm~\ref{alg:fw}) to  formulation~\eqref{prob:proj}, in each iteration, we must solve the subproblem
\begin{equation} \label{subproblemaFW}
\begin{aligned}
& \underset{y}{\text{min}}
& & (y_k - p)^T (y-y_k) \\
& \text{s.t.} & & y \in \A.
\end{aligned} 
\end{equation}
Since $\A = \{ v_1,\dots,v_n\}$ is a finite set, $\bar{y}_k$ solution of \eqref{subproblemaFW}
can be chosen as 
$\bar{y}_k = v_j$, where 
$v_j^T  (y_k - p) = \min_{1 \leq i \leq n} v_i^T  (y_k - p)$. Taking into account that  $y_k \in \convA$ for every $k$, we may associate $y_k$ with iterate $p_k$ of Algorithm~\ref{alg:ag} to obtain 
\begin{equation}\label{eq:vjfw}
v_j^T  (p_k - p) = \min_{1 \leq i \leq n} v_i^T  (p_k - p).
\end{equation}

In comparison, in the $k$-th iteration,  Algorithm~\ref{alg:ag} chooses $v_i \in \A$ such that  $v_i^T (p_k - p) \le (\| p_k\|^2 - \|p\|^2)/2$, 
which, by Lemma~\ref{lema:pivo}, characterizes a pivot. Thus, it is clear by \eqref{eq:vjfw}, that Algorithm~\ref{alg:fw} not only chooses a pivot, but one that minimizes the product $v_i^T (p_k-p)$. In this sense, Frank-Wolfe (Algorithm~\ref{alg:fw}) applied to \eqref{prob:proj}, with \emph{exact line-search}, can be seen as a \emph{greedy} version of the Triangle Algorithm. 

On the other hand, the Triangle Algorithm can be interpreted as an \emph{inexact} version of Frank-Wolfe, in the following sense. 
Similar to \cite{jaggi}, we consider $\tilde{y}_k \in C$ an inexact solution of the FW subproblem 
(see Step~\ref{stepalg:asfw-xbar} in Algorithm~\ref{alg:fw}, replacing $x$ by $y$, and $f$ by $\Psi$), 
if the following inequality holds
\begin{equation}\label{eq:FWinexato}
    (\tilde{y}_k-y_k)^T\nabla \Psi(y_k)\leq \min_{y \in C}(y-y_k)^T\nabla \Psi(y_k)+ \frac{1}{2}\hat{\delta}_k \bar{\gamma}_k C_\Psi,
\end{equation}
where $\hat{\delta}_k \geq 0$ is an inexactness parameter, $\bar{\gamma}_k$ is given in \eqref{eq:alpha} and $C_\Psi$ is the \emph{curvature constant} of $\Psi$ with respect to the set $C$ (see \cite{jaggi} for a formal definition and detailed discussion). 
If we consider \eqref{prob:proj}, we have $\Psi(y) = \|y - p\|^2/2$, $\nabla \Psi(y) = y - p$, and $C = \convA$ (for which $C_\Psi = D^2$). Then, by identifying $y_k$ with $p_k$ and setting the inexactness parameter as  
\begin{equation}\label{eq:deltakcerto}
    \hat{\delta}_k = \frac{2}{ D^2}\norm{v_j-p_{k}}^2, 
\end{equation}
where $j \in \argmin \{v_i^T(p_k - p)\mid 1 \leq i \leq n \}$, if $v_i$ is a pivot, we have 
\begin{align}
\min_{y \in C}(y -p_k)^T\nabla \Psi(p_k)+ \frac{1}{2} \bar{\gamma}_k C_\Psi \hat{\delta}_k &= (v_{j}-p_k)^T(p_k-p)+ \frac{1}{2}  \bar{\gamma}_k D^2 \frac{2}{ D^2}\norm{v_j-p_{k}}^2 \nonumber  \\
&=(v_{j}-p_k)^T(p_k-p) -\frac{(p_k-p)^T(v_j-p_k)}{\norm{v_j-p_k}^2}\norm{v_j-p_k}^2  = 0\nonumber \\
&\geq -\dfrac{1}{2} \| p_k - p \|^2  \geq  (p_k - p)^T (v_i - p_k), \label{eq:tainexactfw}
\end{align}
 where in the second inequality we use Lemma~\ref{lema:pivo}(iii). Thus, \eqref{eq:tainexactfw} shows that pivot $v_i$ in the $k$-th iteration of Algorithm~\ref{alg:ag}, is an inexact solution to the FW subproblem for $\hat{\delta}_k$ as in \eqref{eq:deltakcerto}.

Therefore, it is not surprising that the complexity results for both algorithms are alike. 
As discussed in Section~\ref{sec:fw} (see inequality~\eqref{eq:sublinearconvFW}), the iterations of the Frank-Wolfe method for the problem \eqref{prob:minf} satisfy 
	\[
	\Psi(y_k) - \Psi(\ys)  = \mathcal{O}\left(  \dfrac{1}{k} \right).
	\]
	In particular, for the convex hull membership problem with $p \in \convA$, it means
	\[\dfrac{1}{2} \| y_k - p \|^2 = \mathcal{O}\left(  \dfrac{1}{k} \right).\]
	
	Furthermore, by Theorem~\ref{teo:complex1}, when $p \in \convA$, the Triangle Algorithm generates a sequence of points $p_k \in \convA$ such that 
	$$
	\| p_k - p \| < \sqrt{\frac{48}{k}} R, 
	$$
thereby exhibiting the same iteration complexity as Frank-Wolfe, if we identify $y_k \equiv p_k$.

Interestingly, the iteration complexity is improved for both algorithms when $p$ is in the relative interior of $\convA$ as stated in Theorem~\ref{teo:complexp} for the Triangle algorithm with strict pivots and as can be devised from \eqref{eq:casoespecialfwinteriorrelativo}  for Frank-Wolfe  (see also  \citet[Theorem~2]{guelat1986some}).
In Table~\ref{tab:compTAeFW}, we present a comparison between the iteration complexity of the Triangle Algorithm for the case when $p \in \convA$ with those of Frank-Wolfe and Away Step Frank-Wolfe (ASFW) under different assumptions. 
{The constants in the table for FW and ASFW are derived from \eqref{eq:sublinearconvFW}, \eqref{eq:thekey} and \eqref{eq:casoespecialfwinteriorrelativo}.}

\begin{table}
    \centering
	\caption{Iterations complexity to find $p_k \in \convA$ such that $d(p_k,p) < \varepsilon R$. The last assumption applies only to TA.}
    \label{tab:compTAeFW}
    \setlength\extrarowheight{2pt} 
\begin{tabularx}{\textwidth}{CCCC}
\toprule
         \bf{Assumptions} & \bf{Triangle Algorithm} & \bf{Frank-Wolfe} & \bf{ASFW}\\ 
		 \midrule 
         $p \in \convA$ & $\mathcal{O}(1/\varepsilon^2)$  & 
       $\mathcal{O}(D^2/\varepsilon^2)$ & $\mathcal{O}\Big(\frac{D^2(m+1)^2}{\Omega_C^2}\ln \frac{\delta_0}{\varepsilon}\Big)$  \\
       and Assumption~\ref{assump:pnointerior} & $\mathcal{O}(1/\varepsilon^2)$  & $\mathcal{O}\Big(\frac{
       	D^2}{\rho^2}\ln \frac{\delta_0}{\varepsilon}\Big)$ & $\mathcal{O}\Big(\frac{D^2(m+1)^2}{\Omega_C^2}\ln \frac{\delta_0}{\varepsilon}\Big)$ \\
       and Assumption~\ref{assump:strong_pivot}  & $\mathcal{O} \left( \frac{R^2}{\rho^2} \ln{\frac{\delta_0}{\varepsilon R}} \right)$  & $\mathcal{O}\Big(\frac{D^2}{\rho^2}\ln \frac{\delta_0}{\varepsilon}\Big)$ & $\mathcal{O}\Big(\frac{D^2(m+1)^2}{\Omega_C^2}\ln \frac{\delta_0}{\varepsilon}\Big)$ \\ \bottomrule
    \end{tabularx}

\end{table}

In addition, when $p \notin \convA$, according to Theorem~\ref{teo:complex1}(ii), the iteration complexity of the Triangle Algorithm depends only on the ``problem geometry'', that is, it does not depend on the tolerance $\varepsilon$. Since FW can be viewed as a Greedy TA, this assertion also holds, as long as stopping criteria based on distance duality are integrated to FW (see Section~\ref{sec:stop}).

When it comes to the cost per iteration, although in the worst case both algorithms have to evaluate the entire list of inner  products $v_i^T(p_k - p)$, it is expected that, in general, the iteration of the Algorithm~\ref{alg:ag} requires fewer evaluations because it only asks for a simple pivot, in contrast to the Frank-Wolfe method that needs a pivot which minimizes the inner product $v_i^T(p_k - p)$.

\subsection{Frank-Wolfe as a Greedy Triangle algorithm}\label{sec:new}	

We have just investigated the relationship between the Triangle Algorithm and the Frank-Wolfe method applied to the problem~\eqref{prob:proj}: FW is equivalent to the Triangle Algorithm  when in TA one chooses a pivot $v_j$ such that 
\begin{equation}\label{eq:greedy}
v_j \in \argmin \{ v_i^T (p_{k} - p) \mid v_i \in \A \setminus \{ p_{k} \} \}.
\end{equation}
In view of Lemma~\ref{lema:pivo}(ii), problem \eqref{eq:greedy} means that the Frank-Wolfe method (with the peculiarities of the CHMP) is equivalent to a \emph{greedy} Triangle Algorithm.  However, instead of choosing $v_j$ such that
\begin{equation}\label{eq:pivob}
v_j^T (p_{k} - p) \le \dfrac{\| p_{k} \|^2 - \| p \|^2}{2},
\end{equation}
we select $v_j \in {\cal A} \setminus \{ p_k \}$ such that the left side of \eqref{eq:pivob} is minimal. 
We call this algorithm \emph{Greedy Triangle} (GT) and summarize it in Algorithm~\ref{alg:aggreedy}. 
\begin{algorithm}
\small
	\DontPrintSemicolon
	\SetAlgoLined
	\KwData{$\A, p, R,\varepsilon \in (0,1)$} 
		Choose $p_0 \in \argmin \{ d(v_j,p) \mid  v_j \in \A \}$ \; 
		\For{$k= 0,1,2, \ldots$}{
			\lIf{ $d(p_k,p)< \varepsilon R$}{stop.}  \label{step:aggcriteriodeparada} 
			Choose $v_j \in \argmin \{ v_i^T (p_{k} - p) \mid v_i \in \A \setminus \{ p_{k} \} \}$.\label{step:gananciosopivo}\\
	\lIf{{$v_j^T (p_{k} - p) > p^T (p_k - p)$}}{stop: $p_k$ is a $p$-witness.}
		Set $\bar{\gamma}_k \in \argmin \{d(p, (1-\gamma)p_k + \gamma v_j )\mid  {0 \le \gamma \le 1}\}$. \\
		 $p_{k+1} \gets  (1-\bar{\gamma}_k)p_k + \bar{\gamma}_k v_j$.
		 }
	\caption{\textsc{Greedy Triangle (GT)}}\label{alg:aggreedy}
\end{algorithm}

\begin{remark}
{We point out that, for $p=0$, Algorithm~\ref{alg:aggreedy} coincides with the von Neumann\footnote{According to Dantzig,  von Neumann communicated that algorithm to him in the 1940s and Dantzig  studied it later in an unpublished manuscript \citep{dantzig92}.} algorithm  \citep{epelman2000,goncalves2009,pena2016}.} 
\end{remark}

Since the inequality  $v_j^T (p_{k} - p) \le p^T (p_{k} - p)$ characterizes a strict pivot, in view of Step~\ref{step:gananciosopivo} of   Algorithm~\ref{alg:aggreedy} it follows that a strict pivot is chosen whenever possible. 
Furthermore, if there are strict pivots at  $p_{k}$ verifying condition \eqref{eq:strongpivot} (see Assumption~\ref{assump:strong_pivot}),   Algorithm~\ref{alg:aggreedy} will employ one of them.

\begin{prop}\label{prop: greedychoosebest}
Assume that Assumption~\ref{assump:pnointerior} holds. 
\begin{listi}
    \item If $d(p_k,p) > \varepsilon R$, then there exists a strict pivot $w \in \A$ satisfying \eqref{eq:strongpivot}. 
    \item Algorithm~\ref{alg:aggreedy} always chooses $v_j \in \A$ satisfying condition \eqref{eq:strongpivot}. 
\end{listi}
\end{prop}
\begin{proof} (i) Let $p_k \in \convA$ be  such that $d(p,p_k) > \varepsilon R$. Assumption~\ref{assump:pnointerior} implies that there exists $\tau>0$ such that $q = p + \tau(p - p_k) \in \convA$ and $d(q,p)=\rho$.  Since $q \in \convA$, applying  Theorem~\ref{teo:dual2} for $q$ (instead of $p$), we conclude that there exists a strict $q$-pivot $w$ at $p$. Thus,  inequality~\eqref{eq:strongpivot} follows from Remark~\ref{rem:strict}. \\
    (ii)  Consider $v_j$ of  Step~\ref{step:gananciosopivo} in Algorithm~\ref{alg:aggreedy} and let $w \in \A$ be a strict pivot satisfying \eqref{eq:strongpivot}. 
    Then, 
   \[ \begin{aligned}
        (p-q)^T(v_j-q) &= -\tau(p-p_k)^T(v_j-p-\tau(p-p_k)) = -\tau(p-p_k)^Tv_j +\tau(p-p_k)^Tp+\tau^2\norm{p-p_k}^2  \\ 
        &\leq -\tau(p-p_k)^Tw +\tau(p-p_k)^Tp+\tau^2\norm{p-p_k}^2  = (p-q)^T(w-q) \leq 0.
	\end{aligned}\]
    Hence, $v_j$ also satisfies the inequality~\eqref{eq:strongpivot}.
\end{proof}
A possible drawback of Algorithm~\ref{alg:aggreedy}, in comparison with Algorithm~\ref{alg:ag}, is that, at each iteration, it is necessary to go through the entire list $\A$ to assure the minimum of \eqref{eq:greedy}. Therefore, the cost per iteration is always $\mathcal{O}(nm)$.

\section{Projected gradient methods for CHMP} \label{sec:spg}

Another class of first-order methods that can be used to solve the CHMP is that of projected gradient methods. 
The projected gradient method is well known in the literature to solve problems as \eqref{prob:minf}. Given a feasible point $x_k\in C$, the method considers the search direction $d_k\coloneqq \x_k-x_k$, where  $\x_k \coloneqq  P_{C}(x_k - \nabla f(x_k))$ and $P_C$ denotes the orthogonal projection onto the closed and convex set $C$. It is not difficult to show that if $x_k$ is not a stationary point for \eqref{prob:minf}, then $d_k$ is a descent direction, and that $x_k + \gamma d_k \in C$, for all $\gamma \in [0,1]$. On the other hand, one can show that  $\| d_k \| = 0$ if, and only if, $x_k$ is a stationary point for \eqref{prob:minf}. 

Over the last decades, nonmonotone strategies for nonlinear optimization began to become popular \citep{gll86,zh2004,gs2017} and, along with scaling ideas, are present in enhanced versions of the projected gradient method. An example is  the Spectral Projected Gradient (SPG) method, proposed by  \cite{birgin2009spectral,birgin2000nonmonotone}, which uses Barzilai-Borwein (spectral, BB-type) scaling of the gradient direction combined with nonmonotone line-search. 

The SPG iteration is given by $x_{k+1} =  x_k+ \gamma_k d_k$, where the search direction $d_k$ is defined as $d_k\coloneqq  P_C(x_k- \lambda_k \nabla f(x_k))- x_k$. Here, $\lambda_k$ is the spectral  scaling parameter 
\begin{equation}\label{eq:spec}
\lambda_k = \frac{s_{k-1}^Ts_{k-1}}{s_{k-1}^T u_{k-1}},
\end{equation}
where $s_{k-1}\coloneqq x_k-x_{k-1}$ and $u_{k-1}\coloneqq \nabla f(x_k)- \nabla f(x_{k-1})$. Rather than imposing a sufficient decrease at each iteration, a characteristic of the SPG is to employ a nonmonotone line-search in order to favor  the acceptance of the full-step $\gamma_k=1$, while still guaranteeing global convergence \cite[Section~2]{birgin2000nonmonotone}.  

For instance, the nonmonotone line-search proposed by ~\cite{gll86} depends on an integer parameter $M \geq 1$ and imposes a functional decrease with respect to the highest functional value over the last $M$ iterations (if $M=1$ the line-search is monotone).  

Algorithm \ref{alg:spg} summarizes the Spectral Projected Gradient method as in \cite{birgin2014spectral}. 

 \begin{algorithm}
 \small
 	\DontPrintSemicolon
 	\SetAlgoLined
 	\KwData{$x_0 \in C, \epsilon > 0, M \geq 1, \eta \in (0,1), 0< \lambda_{\min} \leq \lambda_{\max} < \infty$ and  $\lambda_0 \in  [\lambda_{\min},\lambda_{\max}] .$} 
 	\For{$k= 0,1,2, \ldots$} { 
 	    $\bar{x}_k= P_{C}(x_k- \lambda_k \nabla f(x_k))$ and define $d_k= \bar{x}_k -x_k$. \label{step:spgxbar}\\
 		\lIf{$\norm{d_k}< \epsilon$}{stop: return $x_k$.\label{step:spgstop}} 
 		Compute $f_{\max} = \max \{f(x_{k-j}) \mid \,  0 \leq j \leq \min\{k, M-1\}\}$ and set $\gamma \leftarrow 1$. \\
 		\While{ $f(x_k+ \gamma d_k) > f_{\max} + \eta \gamma \nabla f(x_k)^T d_k$\label{step:testnonmonotone}}{
 		$\gamma \leftarrow \gamma/2$}
 		$\gamma_k \leftarrow \gamma$ \;
 		$x_{k+1}\gets  x_k+\gamma_k d_k$.\;
 		 $s_k\gets x_{k+1}-x_{k} $ \;
 		 $u_k\gets \nabla f(x_{k+1})-  \nabla f(x_{k})$. \;
 		\eIf{$s_k^T u_k \leq 0$} {$\lambda_{k+1}\gets  \lambda_{\max}.$}{
 		$\lambda_{k+1}\gets  \max \{ \lambda_{\min}, \min \{s_k^Ts_k/s_k^T u_k, \lambda_{\max}\}\}$.}
 		} 	
 	\caption{\textsc{Spectral Projected Gradient (SPG)}}\label{alg:spg}
 \end{algorithm}
 
 \begin{remark}
{The main cost per iteration of projected gradient methods is the cost of computing the projection $P_C(\cdot )$. Fortunately, if we consider formulation \eqref{prob:quad} for the CHMP, then $f \equiv \Phi$ and $C = \Delta_n$, the unit simplex, for which the projection can be computed in ${\cal O}(n \log n)$ operations \citep{Condat2016}.}
 \end{remark}

\section{Specific stopping criteria for CHMP}\label{sec:stop}
The distance duality (see Theorem~\ref{teo:dd}) is a remarkable result for the CHMP, and it is the keystone of the Triangle Algorithm. Thus, a natural question that arises is: how other first-order methods applied to CHMP can benefit from this result  in order to employ specialized stopping criteria?

From the discussion in Section~\ref{sec:new}, we established that Frank-Wolfe applied to CHMP can be seen as a Greedy Triangle Algorithm, in the sense that it chooses $v_j \in \A \setminus \{p_k\}$ that minimizes $v_i^T (p_k - p)$. Hence, in case $v_j^T (p_k - p) > (\|p_k\|^2 - \|p\|^2)/2$, we can assert that there is no pivot at $p_k$ and $p \notin \convA$. This way, as in the Triangle Algorithm, it is possible to stop FW (and its variants) as soon as a witness is detected. Furthermore, as in TA, we also stop FW iterations as soon as $\|p_k-p\| \le \varepsilon R$.

Another important question is how to set the tolerance $\epsilon$ in Algorithm~\ref{alg:fw} (Step~\ref{step:FWstop}) in order to ensure $\| p_k - p \| \leq \varepsilon R$. 
Note that the classic stopping criterion of Algorithm~\ref{alg:fw} (and Algorithm~\ref{alg:asfw}) is $\nabla \Psi(y_k)^T(\bar{y}_k-y_k) \geq -\epsilon$ (where we denote an iterate by $y_k$ instead of $x_k$ because we are considering problem \eqref{prob:proj}). 
Let $\ys \in \convA$ be the solution of \eqref{prob:proj}. 
From the discussion in Section~\ref{sec:fw}, we know that 
\[\Psi(y_k) - \Psi(\ys) \le \nabla \Psi(y_k)^T (y_k - \bar{y}_k) \le \epsilon,\]
where $\bar{y}_k$ is the solution of the FW subproblem. 
For problem \eqref{prob:proj}, this inequality is equivalent to 
\[\| y_k - p\|^2 - \| \ys - p \|^2 \le 2 \epsilon,\]
which yields $\| y_k - p\| - \| \ys - p \| \le 2 \epsilon/\| y_k - p\|$. Thus, imposing $\epsilon \coloneqq  \| y_k - p\| \varepsilon R/2$, we  have
\begin{equation}\| y_k - p\| \le \| \ys - p \| + \varepsilon R.\label{eq:stoppingCrit}\end{equation}
If $p \in \convS$, $\ys$, solution of \eqref{prob:proj}, is such that $\norm{\ys - p} = 0$ and then inequality \eqref{eq:stoppingCrit} coincides with $\|p_k - p \| \leq \varepsilon R$ (with $p_k = y_k$). 
On the other hand, if $\nabla \Psi(y_k)^T (y_k - \bar{y}_k) \le \| y_k - p \| \varepsilon R/2$ and $\norm{y_k - p} > \varepsilon R$ then, from  \eqref{eq:stoppingCrit}, $\norm{\ys - p} > 0$ and we can conclude that $p \notin \convS$. 
Therefore, instead of the classic FW stopping criterion, we can use the \emph{relative error} criterion 
\begin{equation}\label{eq:pstopfw}
\nabla \Psi(y_k)^T (y_k - \bar{y}_k) \le \| y_k - p \| \dfrac{\varepsilon R}{2}.
\end{equation}

When it comes to projected gradient methods applied to problem \eqref{prob:quad},   
we recall that $\nabla \Phi(x_k) = A^T(A x_k - p)$ and, since $x_k \in \Delta_n$, we can write $p_k = Ax_k \in \convA$. Thus, each gradient component can be written as $[\nabla \Phi(x_k)]_i = v_i^T (p_k - p)$ and, in case each component is greater than $(\|p_k\|^2 - \|p\|^2)/2$, we can stop and return $p_k = A x_k$ as a witness that $p \notin \convA$. 

Next, we discuss how to set the tolerance $\epsilon$ in Algorithm~\ref{alg:spg}  (Step~\ref{step:spgstop}) in order to guarantee that $\norm{A \bar{x}_k - p} \leq \varepsilon R$ (here, $\bar{x}_k$ is from Step~\ref{step:spgxbar}). For that, first we need the following lemma.

\begin{lema}\label{lema:auxspgcriterio}
Let $f: \Rn \rightarrow \R $ be a continuously differentiable convex function with  $L$-Lipschitz gradient and $C \subset \Rn$ be a non-empty, convex and compact set with diameter $D$. If $\xs \in C$ is a minimizer of $f$ over $C$ and  $\bar{x} \coloneqq P_C(x-\lambda\nabla f(x))$, for some $\lambda>0$ and $x \in C$, then 
\[\label{eq:auxspgcriterio}
f(\bar{x})- f(\xs) \leq D\left(\frac{L}{2}+ \frac{1}{\lambda}\right)\norm{\bar{x} -x }.
\]
\end{lema}

\begin{proof}
Since $\bar{x}$ is the projection of $x - \lambda \nabla f(x)$ onto $C$, by the characterization of projections onto convex sets we have  $(w - \bar{x})^T(x-\lambda\nabla f(x) - \bar{x}) \leq 0$, for all $w \in C$. In particular, as  $\xs\in C$ and $\lambda >0$, we obtain
\begin{equation}\label{eq:minxbar}
    \nabla f(x)^T(\bar{x}-\xs) \leq \frac{1}{\lambda} (\bar{x}-\xs)^T(x- \bar{x}).
\end{equation}
Note also that
\begin{align} 
(\bar{x} - \xs)^T(x - \bar{x}) &= (\bar{x}-x)^T(x- \bar{x}) + (x - \xs)^T(x - \bar{x}) \leq -\norm{x- \bar{x}}^2 + \norm{x-\xs}\norm{x-\bar{x}}.\label{eq:minxbar1}
\end{align}
Therefore, using \eqref{eq:minxbar} and \eqref{eq:minxbar1}, we get, after some manipulations, 
\begin{equation}\label{eq:minxbar2}
\nabla f(x)^T(\bar{x} - x ) \leq \nabla f(x )^T(\xs - x ) + \frac{1}{\lambda }( -\norm{x - \bar{x} }^2 + \norm{x - \xs}\norm{x -\bar{x}}).
\end{equation}
Since  $\nabla f$ is $L$-Lipschitz, 
by using \eqref{eq:minxbar2} and convexity of $f$, we obtain 
\begin{equation}
f(\bar{x} )- f(x )- \frac{L}{2}\norm{\bar{x} -x }^2  \leq f(\xs) -f(x )+\frac{1}{\lambda}( -\norm{x - \bar{x} }^2 + \norm{x - \xs}\norm{x -\bar{x}}),
\end{equation}
which implies that 
\[
f(\bar{x} ) - f(\xs) \leq  \frac{L}{2}\norm{\bar{x} - x }^2 + \frac{1}{\lambda}( -\norm{x - \bar{x} }^2 + \norm{x - \xs}\norm{x -\bar{x}}) \leq D\left(\frac{L}{2} + \frac{1}{\lambda}\right)\norm{\bar{x} -x },
\]
where, in the last inequality, we used the fact that both $x, \xs \in C$, which has, by hypothesis, diameter $D$.
\end{proof}

Now, we recall that for problem \eqref{prob:quad}, $f(x) := \Phi(x) = \frac{1}{2}\norm{Ax - p}^2$ and $C = \Delta_n$. In this case, $L = \norm{A}^2$ and $D = \sqrt{2}$. From the definition of the spectral parameter (see \eqref{eq:spec}), we have $1/\lambda_k \leq L$. Then, we can rewrite  inequality \eqref{eq:auxspgcriterio} from Lemma~\ref{lema:auxspgcriterio} as
\begin{equation}\label{eq:auxspgstop}
\Phi(\bar{x}_k) - \Phi(\xs) \leq \frac{3}{2}LD\norm{\bar{x}_k- x_k},
\end{equation}
by taking $x_k$ as $x$ and $\bar{x}_k$ as  $\bar{x}$. 

Since, in Algorithm~\ref{alg:spg}, $d_k \coloneqq 
\bar{x}_k - x_k$, if one imposes $\epsilon \coloneqq \frac{\norm{A\bar{x}_k - p}}{3LD} \varepsilon R$, inequality \eqref{eq:auxspgstop} implies that 
\(
\| A\bar{x}_k - p\| \le \| A \xs - p \| + \varepsilon R.
\)
If $p \in \convS$, $\xs$, solution of \eqref{prob:quad}, is such that $\norm{A \xs - p} = 0$ and then the last inequality coincides with $\norm{A\bar{x}_k - p} \leq \varepsilon R$. We remark that computing a $\bar{x}_k$ such that the last inequality is satisfied  is similar to finding an $\varepsilon$-solution for the Triangle Algorithm because $A\bar{x}_k \in \convA$. 
On the other hand, if 
\(\label{eq:pstopspg}
	\norm{d_k} \le \frac{\norm{A\bar{ x}_k - p}}{3LD} \varepsilon R
	\)
	and $\norm{A\bar{x}_k - p} > \varepsilon R$,  then $\norm{Ax^* - p} > 0$ and we can conclude that $p \notin \convA$.


\section{Numerical Experiments} \label{sec:numerics}
%
{In order to evaluate the performance of the first-order methods described in Sections \ref{sec:triangle}, \ref{sec:fw} and \ref{sec:spg} when solving CHMP, 
equipped with the stopping criteria from Section~\ref{sec:stop},  
 we present three sets of computational experiments.} 

{In Section~\ref{sec:random}, we consider randomly generated instances of CHMP organized in 4 different scenarios according to the distribution of the points of $\A$ and the relative position of the query point $p$. This way, we can cover different ranges for values of the \emph{geometric constants}, namely $D$, $R$, $\rho$, $\delta_0$, $\Delta$ and $c$, 
and  evaluate  the performance of the algorithms in favorable and unfavorable conditions. 
Section~\ref{sec:lpfeas} discuss how linear programming feasibility problems can be cast as CHMPs under suitable assumptions and compares the studied first-order methods for CHMP with the dual-simplex applied to the original problem. 
Finally, in Section~\ref{sec:mnist} we apply the studied methods to an image classification problem, in order to determine membership of (or distance from) elements in the testing set to the convex hull of the elements in the training set (or its subclasses).}

All the tests were run on an Intel Core i7-8565U 1.80~GHz with 8~GB of RAM running Windows 10 whilst the algorithms were implemented in Matlab R2019a. The codes of our experiments are fully available at \url{https://github.com/rafaelafilippozzi/First-order-methods-for-the-CHMP}.

\subsection{Artificial instances of CHMP}\label{sec:random}
The artificial instances of CHMP were generated following a procedure similar to the one proposed by \cite{KalantariComparacao}.  
Each element of ${\cal A} = \{v_1,\dots,v_n\}$ is randomly generated according to a uniform distribution in the unit ball of $\R^m$ \citep{harman2010}. For  that,  each of the $n$ points of $\A$ is generated as follows: first we sample a  $\hat v \in \Rm$  from the standard normal distribution; then we set $v\in \A$ as $v \coloneqq \sqrt[m]{u} (\hat v / \norm{\hat v})$, where $u$ is sampled from the uniform distribution in $[0,1]$.

The query point $p\in \Rm$ is generated in the following four cases:
\begin{lista}
\item $p\in \convA $ in the relative interior of $\convA$;
\item $ p \in \convA $ with visibility factor close to zero;
\item $p\notin \convA $ away from the boundary;
\item $ p \notin \convA $ with visibility factor close to zero.
\end{lista}

The cases above were selected aiming to evaluate the performance of the considered first-order algorithms in the best and worst cases for the Triangle Algorithm. 
For case (a) we point out the improved iteration complexity given by Theorem~\ref{teo:complexp} for the Triangle Algorithm with ``strong pivots'' (see Assumptions~\ref{assump:pnointerior}, \ref{assump:strong_pivot} and Proposition~\ref{prop: greedychoosebest}). 
 In the cases where the point $p$ is outside $\convA$, the relative position of $p$ affects the values of $\Delta$, $R$ and $\delta_0$ in Theorems~\ref{teo:complex1} and  \ref{teo:complexalternative}. In the cases (b) and (d) we have tried to force the visibility factor $c$ to be close to zero  in order to obtain harder instances for the Triangle algorithms (TA and GT). Details are given in Sections~\ref{sec:caseb}~and~\ref{sec:cased}.

For all the tests in this section we consider the following tolerances,  algorithmic parameters and implementation details:
\begin{listi}
	\item The maximum number of iterations was set to
	$	\texttt{maxit} \coloneqq \min \{ \max \{1000n, 10000 \}, 10^{6} \}
	$.
	
	\item The usual stopping criteria of both Away Step Frank-Wolfe (ASFW) and Spectral Projected Gradient (SPG) have been changed according to the discussions in Section~\ref{sec:stop}.

    \item The tolerance for an $\varepsilon$-approximate solution of CHMP was set to $\varepsilon := \num{e-4}$ for all tested algorithms.
 
	\item For the nonmonotone line search in SPG we used $M=15$ because it obtained the best performance in preliminary tests. Moreover, we set $\lambda_{\min} := 10^{-8}$ and  $\lambda_{\max} := 10^{8}$.
	
	\item 
	As the variables $x \in \Rn$ in the quadratic formulation \eqref{prob:quad} are the coefficients of the convex combination (of the elements of $\A$), 
	we consider the starting point $x_0 = e_i \in \Rm$, where $e_i$ denote the $i$-th canonical vector of $\Rn$ and the index $i$ is such that 
	\begin{align}
	i= \displaystyle \argmin_{1 \leq j \leq n} \{\norm{v_j-p}\}, 
	\end{align}
	so $Ax_0 = p_0$, where $p_0$ is the starting point for Algorithm~\ref{alg:ag}. 
	
	\item In our implementation of the  Triangle Algorithm, at each iteration, we build the full list of values $v_i^T (p_k - p)$ (see Lemma~\ref{lema:pivo}(ii)),  
	find all the pivots (if any) and select one randomly. The random choice of a pivot (instead of a pivot $v_i$ with the smallest index $i$) is inspired by \cite{zhangtrianglerandom} and turns out to be very effective in reducing the number of TA iterations across the four different scenarios  considered in our experiments. Though we could have tested sequentially whether $v_i$, for $i=1,2,\dots,n$, is a pivot and evaluated the products $v_i^T (p_k - p)$ one-by-one, stopping the search as soon as the first pivot is found, we instead ``vectorize'' the search which usually has a superior performance in Matlab than loops.  
	We emphasize that in TA we have used simple pivots (see Lemma~\ref{lema:pivo}) at each iteration rather than strict pivots as in Definition~\ref{def:pivoestrito}. 
\end{listi}



The computational cost per iteration of both TA and GT is $\mathcal{O}(mn)$ arithmetic operations, corresponding to $n$  inner products of vectors in $\Rm$.
In ASFW, we also have  $\mathcal{O}(mn)$ arithmetic operations per iteration to determine Frank-Wolfe and Away directions. SPG is the algorithm that has the highest computational cost per iteration,  $\mathcal{O}(2mn + n \log n)$, since it requires two matrix-vector products and the projection onto the unit simplex $\Delta_n$; see Section~\ref{sec:spg}. 

\begin{remark}
{Concerning the Triangle Algorithm, it was shown in \citep{Kalantari:2019a} that, with a pre-processing cost of ${\cal O}(mn^2)$, the iteration cost of TA can be reduced to ${\cal O}(m+n)$. Such alternative implementation may be interesting when the expected number of TA iterations is much larger than $n$. 
However, in the experiments reported here we do not use pre-processing.}
\end{remark}

 In the experiments described in this  subsection we fix the dimension as $m=100$ and vary the number of points $n$ in $\A$ from $500$ to $10^5$. For each value of $n$, 10 random instances are generated. Tables in \ref{sec:apa} show the \emph{average} number of iterations required by each algorithm.

\subsubsection{Case(a): $p$ in the relative interior of $\convA$}\label{sec:casea}

To ensure that the query point lies in $\textrm{relint} (\convA)$, we select it as $p=0\in \Rm$, the center of the unit ball. Since every $v_j \in \A$ originates from a uniform distribution on the unit ball, whenever  the number $n$ of points in $\A$ is sufficiently large, it is rather true that  $p\in \convA$. In addition, it is also likely that $B_\rho(p) \subset \convA$, for $\rho>0$ away from zero and thus, in this case,  Assumption~\ref{assump:pnointerior} is satisfied. 

In this scenario, we expect GT to perform better than TA. Bear in mind, from Theorems \ref{teo:complex1} and \ref{teo:complexalternative}, that the iteration complexity of TA is given by 
\[
\mathcal{O}\left(\min \left\{\frac{1}{\varepsilon^2}, \frac{1}{c}\ln \frac{\delta_0}{\varepsilon R}\right\}\right), 
\]
where $\delta_0= d(p,p_0), R:=\max \{d(v_i,p) \mid v_i\in \A\}$ and $c$ is from Assumption~\ref{assump:cfatorvisibi}. However, in view of Proposition~\ref{prop: greedychoosebest}, Assumption~\ref{assump:strong_pivot} is satisfied by GT and its iteration complexity is given by Theorem~\ref{teo:complexp}:
\[
\mathcal{O}\left( \frac{R^2}{\rho^2}\ln \frac{\delta_0}{\varepsilon R}\right).
\]

Figure~\ref{fig:casea} shows the performance of the algorithms in terms of running time, while the number of iterations is listed in Table~\ref{tab:iterationspin} (in the appendix). As anticipated, GT requires fewer iterations than TA to retrieve $p_{\varepsilon}$ such that $d(p_{\varepsilon},p) \leq \varepsilon R$. Moreover, when $n$ is large, GT also achieves the best performance in terms of running time, followed closely by SPG and ASFW. 
Even though the cost per iteration of SPG is higher, it required only 12 iterations on average while GT and ASFW used up to 118 iterations.

\begin{figure}
	\centering
	\includegraphics[scale=0.53, trim=100 200 100 200, clip]{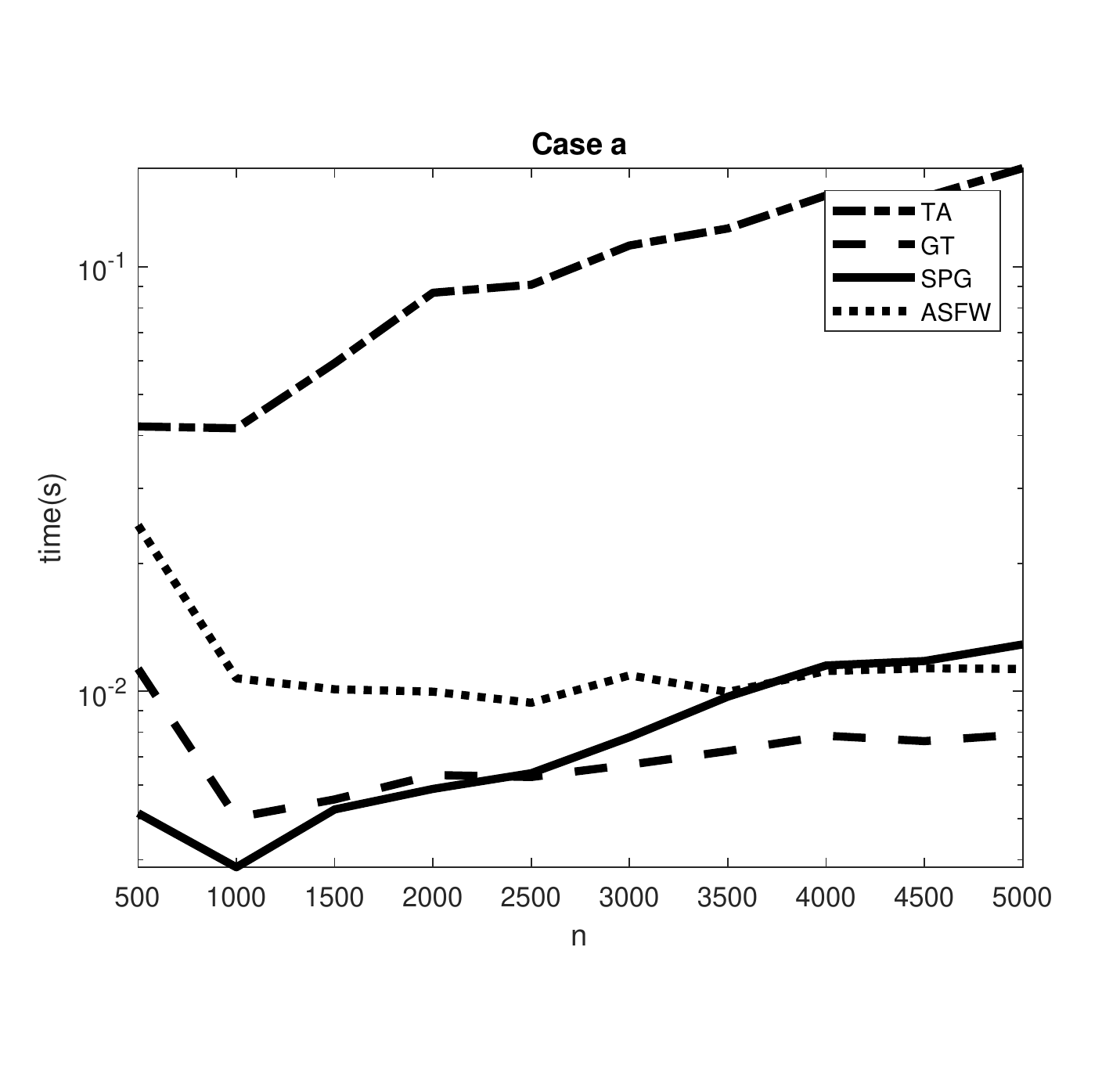}
    \hfill 
	\includegraphics[scale=0.53, trim=100 200 100 200, clip]{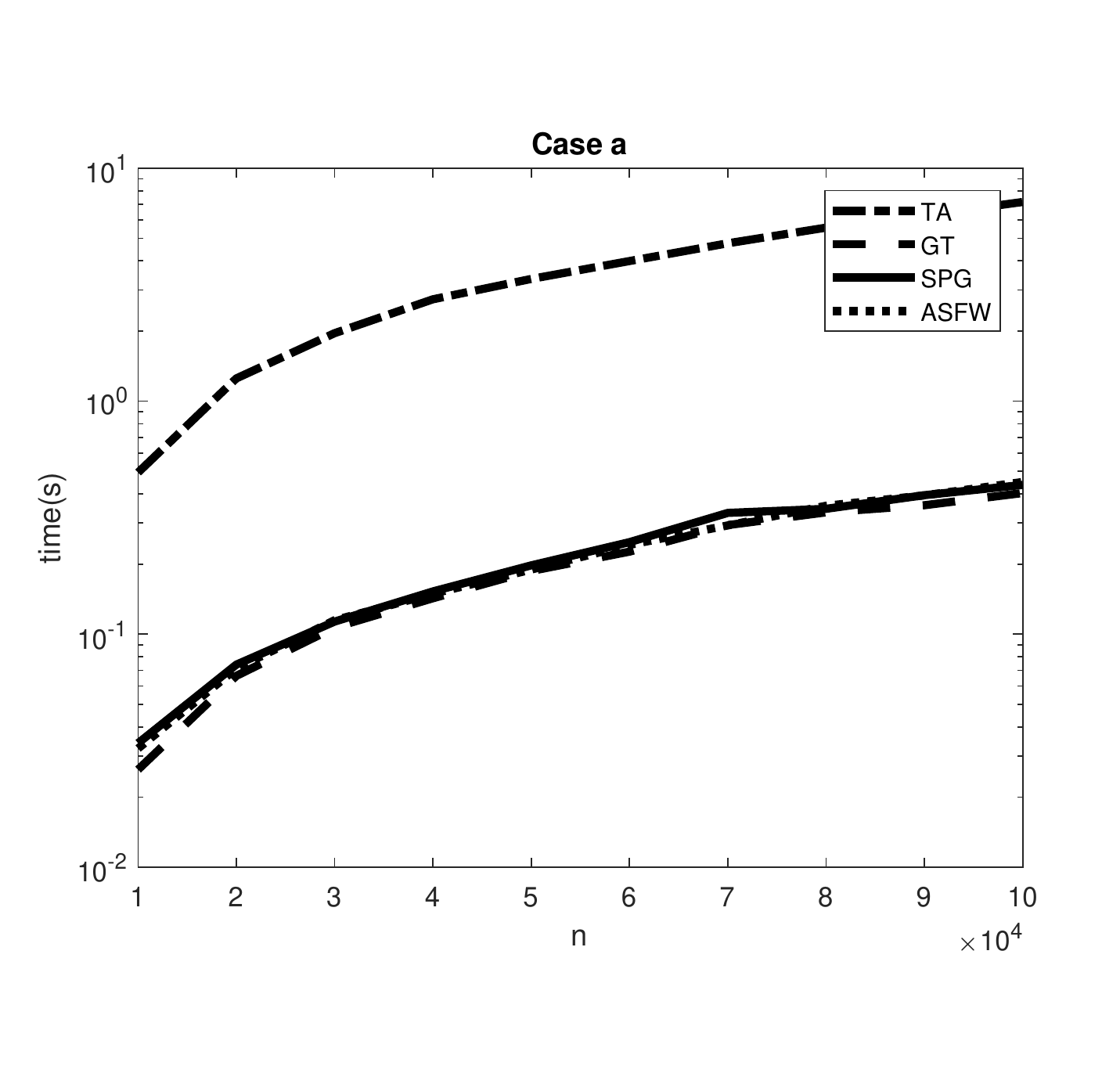}
	\caption{Case (a). Running times (in seconds) for dimension $m=100$ and increasing $n$, the number of points in $\A$. }\label{fig:casea}
\end{figure}

\subsubsection{Case (b): $p \in \convA$ with visibility factor close to zero}\label{sec:caseb}

Here, we generate the query point $p \in \convA$ on the boundary of the convex hull, similar to Example~\ref{exe:quadrado}.  For that, first we determine two points of $\A$, say  $v_{\ell}$ and $v_{q}$, that have the highest values of the linear functional $\psi(v) = e^T v$: since $\convA$ is a polytope and from the way the points of $\A$ were randomly generated, we expect, for a sufficiently large $n$, that $v_{\ell}$ and $v_q$ are extreme points of $\convA$ (the maximum of $\psi(v) $ is achieved in at least one of these points).
Next, we set $p = v_{\ell}/2 + v_{q}/2$. 
In addition, to prevent $v_{\ell}$ and $v_{q}$ from being points of $\A$ closest to $p$, we add to $\A$ a new point $v_{s}$ such that $d(v_s,p) < d(v_{\ell},v_{q})/2$. More precisely,  
\[
v_s = p - \dfrac{\beta}{2}\dfrac{\| v_{\ell} - v_q \|}{\| p \|}p,
\]
where $\beta=0.9$. This last step is necessary because if $p_0 = v_q = \mbox{argmin} \{ d(v,p) \mid v \in \A \}$, for example, $v_{\ell}$ is a strict pivot since $$(p_0-p)^T(v_q-p) = (v_{\ell}-p)^T(v_q-p) = \frac{1}{4}(v_{\ell}-v_q)^T(v_q-v_{\ell})\leq 0.
$$ 
In this situation, GT  could select $v_{\ell}$ as pivot in the first iteration and $p$ would be retrieved after the exact line search. This actually happened when we did not add $v_s$.

The instances considered in this subsection are more difficult for TA and GT algorithms because the upper bound $1/\sqrt{1+c}$ in \eqref{eq:hip1} can be quite close to 1. 
Since $p = \frac{1}{2}(v_q + v_{\ell})$, either $v_q$ or $v_{\ell}$ is a pivot for any $p' \in \convA \setminus \{p\}$. In fact, otherwise
\[
    (p'- p)^T (v_{\ell} - p)  > \frac{1}{2} \norm{p'-p}^2  \quad \mbox{ and } \quad 
    (p'- p)^T (v_{q} - p)  > \frac{1}{2} \norm{p'-p}^2 
\]
would imply
\[
0 = (p'- p)^T (v_{\ell} + v_q - 2p) > \norm{p'-p}^2, 
\]
a contradiction. Moreover, if $v_{\ell}$ was the pivot chosen to obtain $p_k$ at the iteration $k-1$, from the exact line search $(v_{\ell} - p_k)^T (p-p_k) = 0$, thus $v_{\ell}$ is not a pivot at iteration $k$, which implies that $v_q$ is. Therefore, starting at $p_0 \in \convA \setminus \{v_{\ell}, v_q, p \}$, we could have in each iteration of TA (or GT)  $p_k$ as a convex combination of $p_0$, $v_{\ell}$ and $v_q$, with $v_q$ and $v_{\ell}$ alternating as pivots. This phenomenon is similar to the one of Example~\ref{exe:quadrado}.
Hence,
since $1/c$ can be relatively large,
we expect to observe the complexity of the
Theorem~\ref{teo:complex1} instead of the Theorem~\ref{teo:complexalternative} for TA and GT (in contrast to case (a)). 
Nevertheless, the assumptions of Theorem~\ref{teo:linear_asfw}, concerning ASFW, are still satisfied. 

Both Triangle and Greedy Triangle algorithms reached the maximum number of iterations and, therefore, 
are not reported in Table~\ref{tab:iterationspin} and Figure~\ref{fig:caseb}. 
On average,  the ASFW converged in 12 iterations and 7 of such iterations have used the away direction. 
From Figure~\ref{fig:caseb} we see that ASFW outperformed SPG. 
This happened because, although SPG achieved the desired precision in a maximum of $10$ iterations, on average (see Table~\ref{tab:iterationspin}), 
its iteration cost is higher than ASFW. 
%
\\

\begin{figure}
\centering
	\includegraphics[scale=0.53, trim=100 200 100 200, clip]{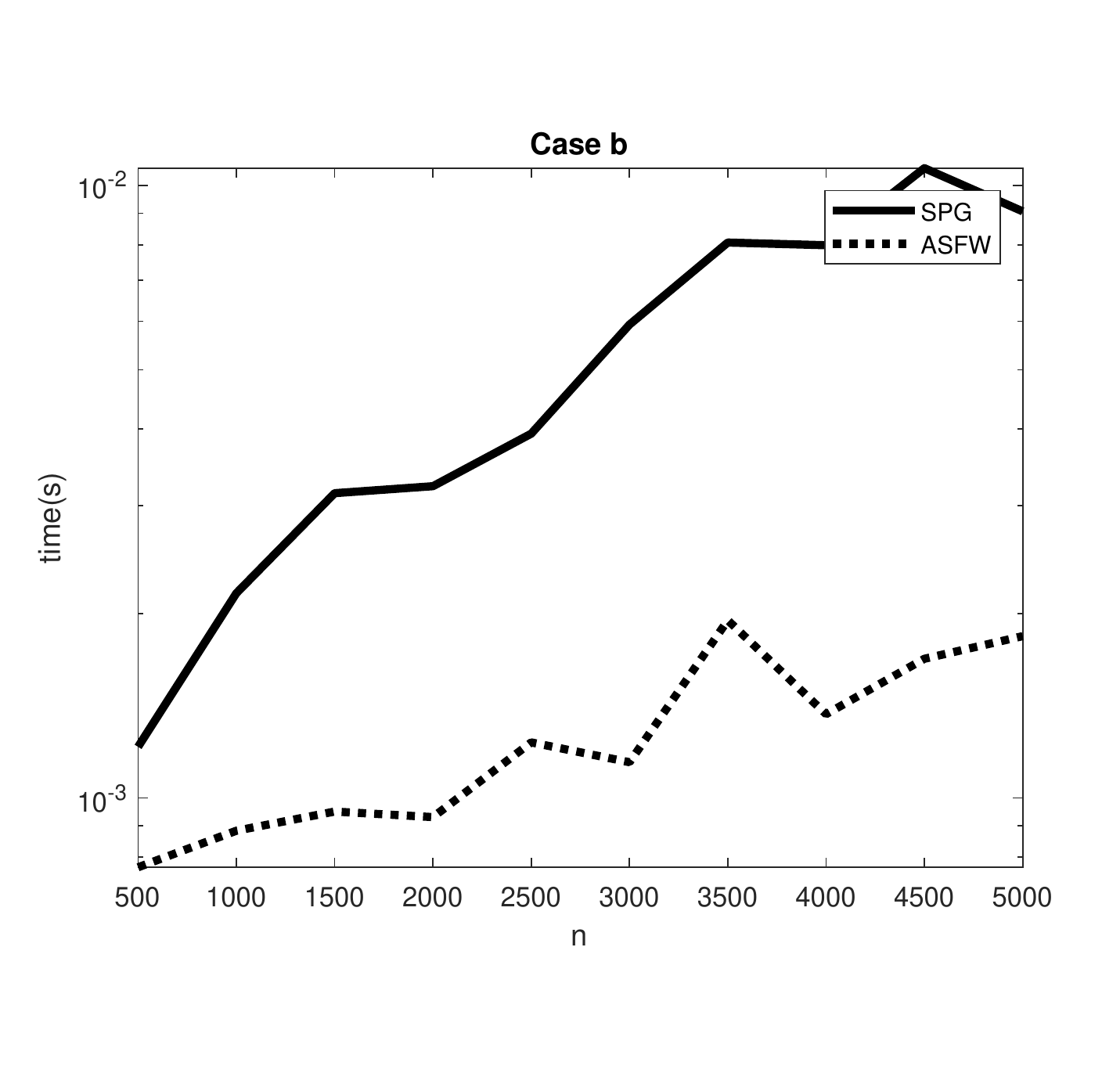} \hfill 
	\includegraphics[scale=0.53, trim=100 200 100 200, clip]{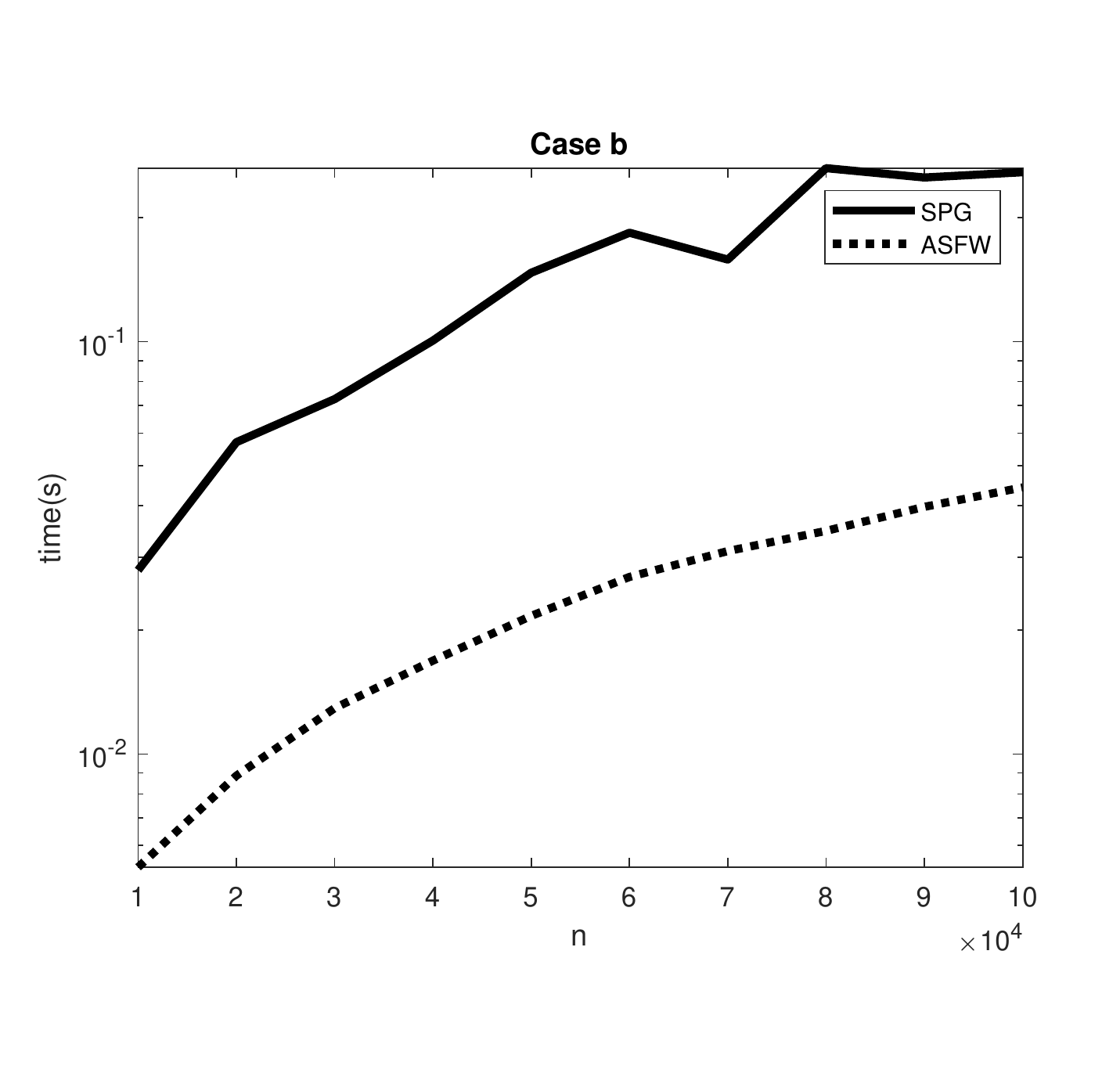}
	\caption{Case (b). Running times (in seconds) for dimension $m=100$ and increasing $n$, the number of points in $\A$. }\label{fig:caseb}
\end{figure}


\subsubsection{Case (c): $p \notin \convA$ and  away from the boundary}\label{sec:casec}
Now, the query point $p$ is located outside $\convA$. We generate  $p \in \Rm$ as in Case (b), followed by a dilation: it is multiplied by $1.5$. Differently from the previous case, we do not add another point $v_s$ to $\convA$ to prevent $p_0$ from being $v_q$ or $v_{\ell}$. Note that, in this way, as $\A \subset B_1(0)$, $p \notin \convA$ and $p$ is sufficiently far from the boundary. This ensures that $\Delta := \min \{ d(v,p) \mid v \in \convA \}$ is sufficiently greater than zero.

As can be seen in Table~\ref{tab:iterationspout} (\ref{sec:apa}), all the algorithms have found a witness in no more than $4$ iterations. Bear in mind that Theorem~\ref{teo:complex1} states that the iteration complexity, in this case, does not depend on the tolerance $\varepsilon$, but on the constants $\Delta$ and $R$; it is not hard to show that the ratio $R/\Delta$ approaches 1 as $\Delta$ increases. Figure~\ref{fig:casec} depicts the running times for this scenario.

Recall from Theorems \ref{teo:complex1} and \ref{teo:complexalternative} that the iteration complexity of TA and GT, when $p \notin \convA$, is
\[\label{eq:GT_TAcost}
\mathcal{O}\left(\min \left\{\frac{R^2}{\Delta^2}, \frac{1}{c}\ln \frac{\delta_0}{\Delta}\right\}\right), 
\] 
where $c$ is from \eqref{eq:hip1} and depends on the choice of pivot. Since GT looks for a pivot $v$ that minimizes $v^T (p_{k}-p)$, usually $\sin \theta_k$ will be smaller for GT than TA (see Figure~\ref{fig:pivoestrito}) and thus, we expect a larger constant $c$ for GT than TA which explains the lower number of iterations for the former.

For the instances considered in this subsection, we observed $1.58 \leq R \leq 1.8 $,  $0.32 \leq \Delta \leq 0.41$ and $ 0.68 \leq \delta_0 \leq 0.81$. Moreover, 
by evaluating $\sin{\theta_k}$ throughout the iterations, we  deduced that if the visibility factor $c$ exists, its value must be less than $0.02$. 

\begin{figure}
	\centering
	\includegraphics[scale=0.53, trim=100 200 100 200, clip]{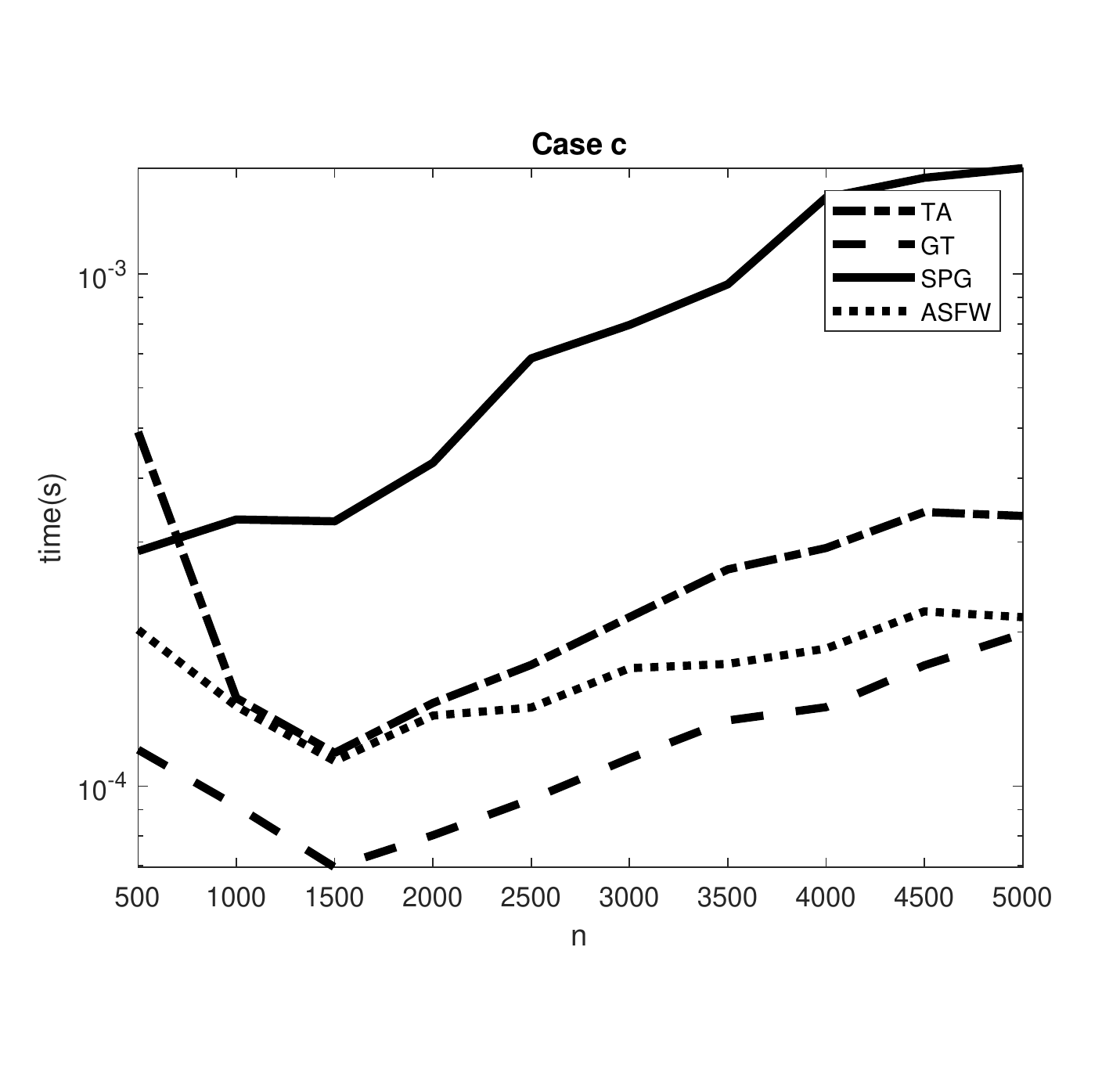} \hfill
	\includegraphics[scale=0.53, trim=100 200 100 200, clip]{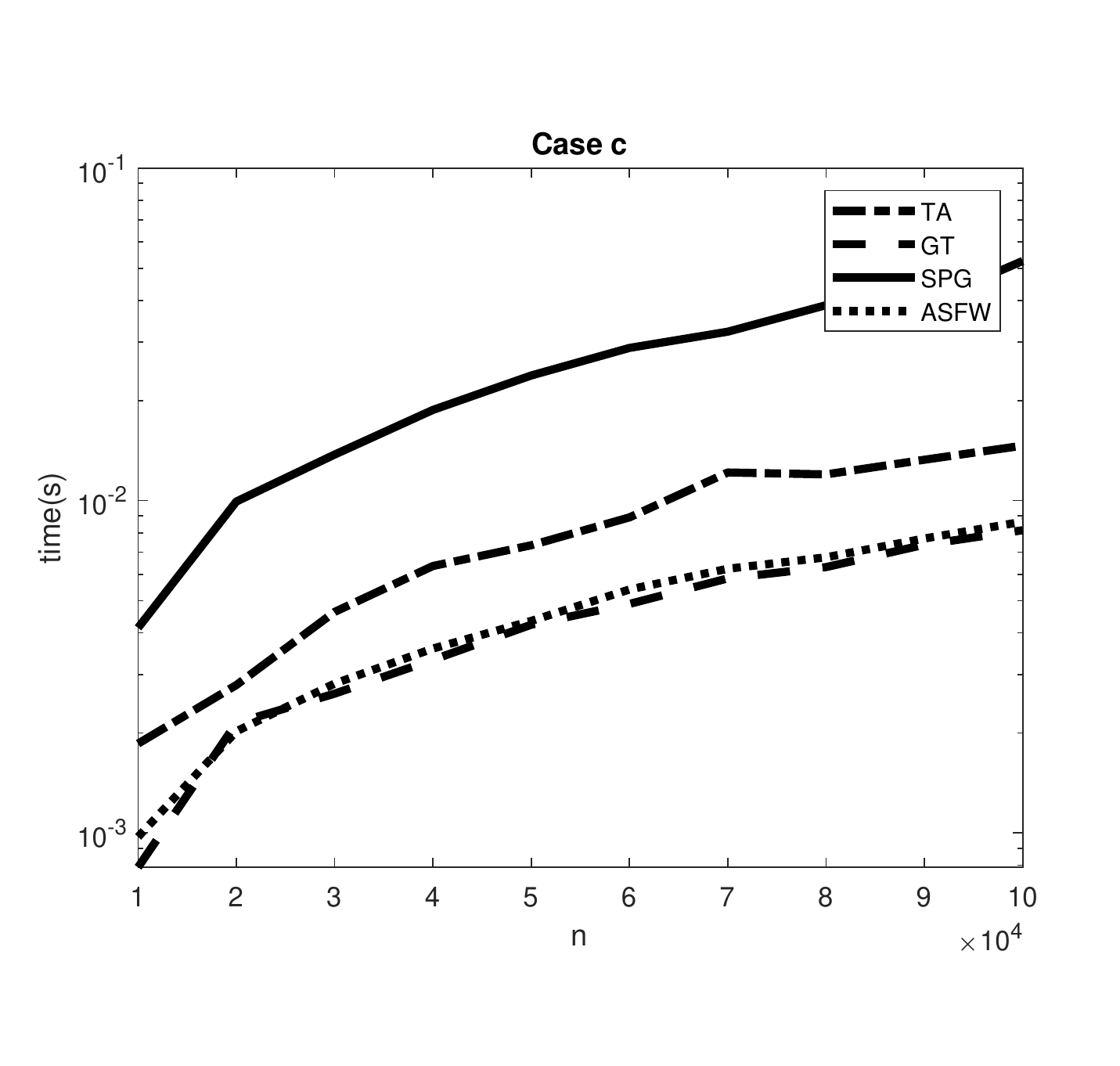}
	\caption{Case (c). Running times (in seconds) for dimension $m=100$ and increasing $n$, the number of points in $\A$. }\label{fig:casec}
\end{figure}

In addition, ASFW and GT perform better than  SPG. The gap in running time between ASFW (and GT) and SPG  increases as the number of  points increases. Although the number of iterations required by SPG is similar to the others, remember that its cost per iteration is higher due to the projection cost of ${\cal O}(n \log n)$.

We remark that in these instances $p_0$, the point of $\A$ closest to $p$, can be either $v_q$ or $v_\ell$. As argued in the previous section,  if $p_0=v_q$, then $v_{\ell}$ is a strict pivot and can be chosen by GT and ASFW in the first iteration. Since, for such instances, the projection of $p$ onto $\convA$ is a convex combination of $v_q$ and $v_\ell$, these two algorithms can find it at the first iteration, explaining the results in Table~\ref{tab:iterationspout}.

\subsubsection{Case (d): $p \notin \convA$ with visibility factor close to zero.}\label{sec:cased}

Motivated again by Example~\ref{exe:quadrado}, now we generate $p$ as in case (b) but we multiply its coordinates by $1.01$. Here we add $v_s$, as in case (b) and, as explained in Section~\ref{sec:caseb}, we expect to obtain instances with visibility factor close to zero. 
In fact, by analyzing the TA iterations for this set of instances, we observed that if a positive factor $c$ exists its value will be less than  $\mathcal{O}(10^{-4})$. 
Moreover, for the instances generated in this case, we have $0.007 \leq \Delta \leq 0.008$, $1.34 \leq R \leq 1.54$, and $ 0.57 \leq \delta_0 \leq 0.75$. 
Thus, from Theorem~\ref{teo:complex1} (and Theorem~\ref{teo:complexalternative}) in the worst case TA may spend  $\mathcal{O}(10^4)$ iterations to find a witness.

As can be seen in Table~\ref{tab:iterationspout}, the average number of TA and GT iterations is very high in the comparison to the other algorithms and this translates to the running times shown in Figure~\ref{fig:cased}.  
We remark that the number of SPG and ASFW iterations also increased in this case when compared to case (c). Note that the Triangle Algorithm takes practically the same time as GT. 

Remarkably, the away directions enabled ASFW  to avoid the zigzag phenomenon seen in the Example~\ref{exe:quadrado}: it was the fastest algorithm, followed by SPG.

\begin{figure}
	\centering
	\includegraphics[scale=0.53, trim=100 200 100 200, clip]{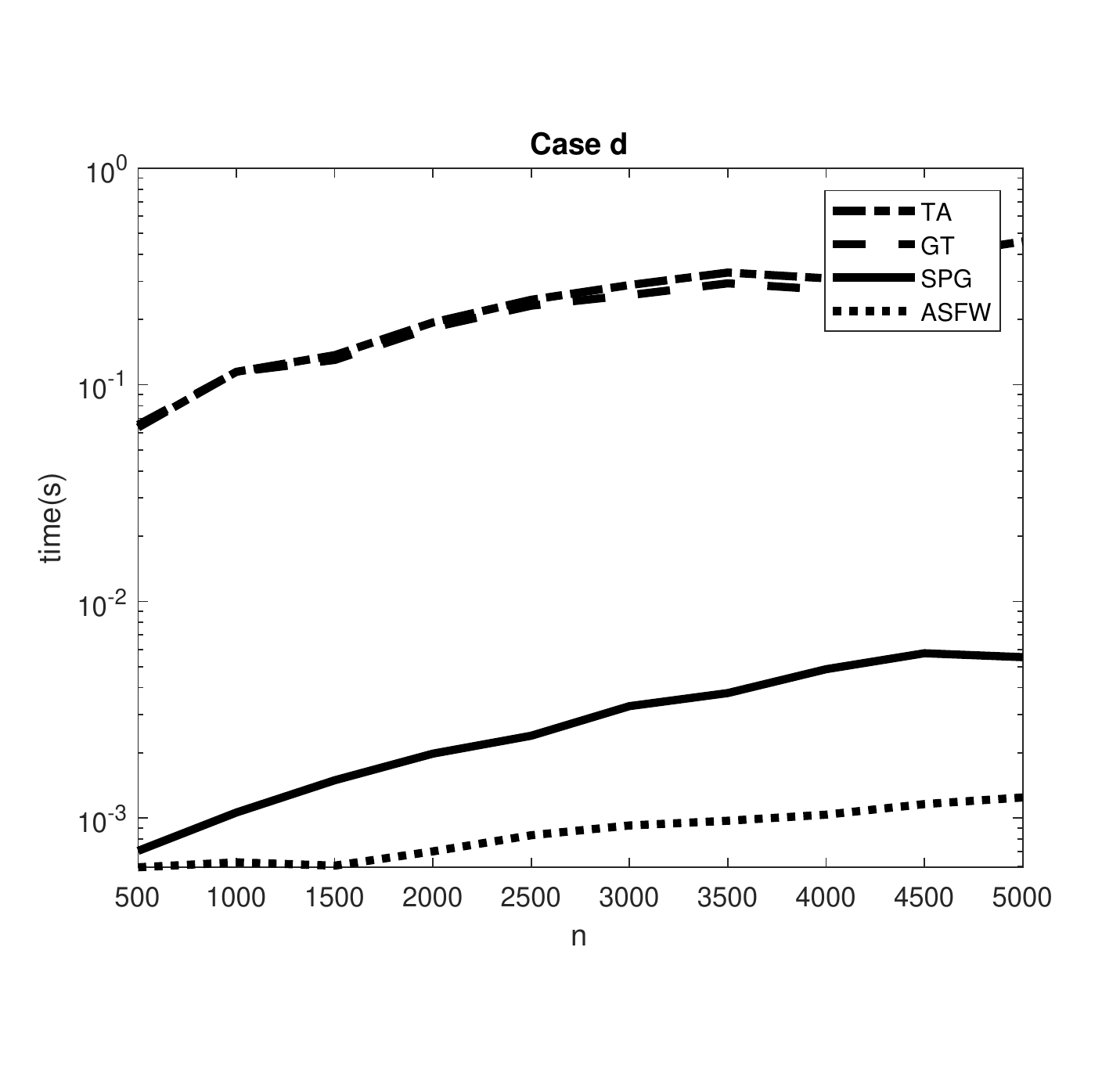} \hfill 
	\includegraphics[scale=0.53, trim=100 200 100 200, clip]{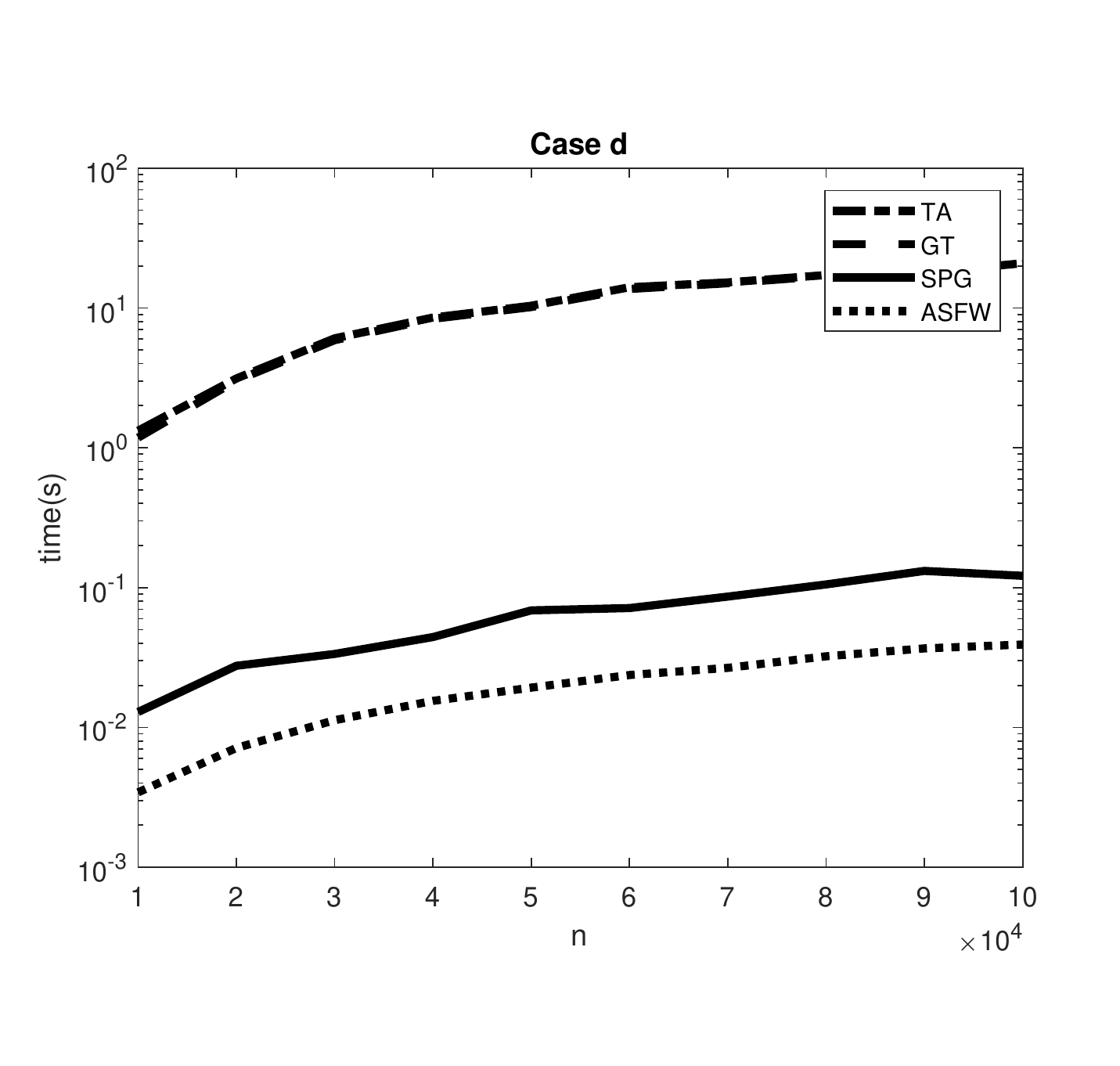}
	\caption{Case (d). Running times (in seconds) for dimension $m=100$ and increasing $n$, the number of points in $\A$. }\label{fig:cased}
\end{figure}


\subsection{Linear Programming Feasibility Problem}\label{sec:lpfeas}
The convex hull membership problem is also related to the linear programming feasibility problem.
The feasible set of a linear programming problem can be  given by $\Omega = \{x \in \Rn \mid Ax=b, x \geq 0\}$, with matrix $A=[a_1\ a_2 \ \cdots\  a_n],$ and vectors  $ b, a_i\in \mathbb{R}^m$. We are interested in elements of $\Omega$ with bounded norm, let us say $\| x \|_1 \leq N$. In other words, we would like to decide whether the intersection of $\Omega$ with the half-space defined by $H_N = \{ x \in \Rn \mid e^Tx \leq N \}$, for a given $N>0$, is non-empty. 
This (more specific) LP feasibility problem can be cast as 
\begin{equation} \label{eq:sistemafacPL}
\begin{aligned} \begin{pmatrix}
A &  0 & -b \\
e^T & 1 & -N \\
0^T & 0 & 1
\end{pmatrix}\begin{pmatrix}
\alpha \\
\beta \\
\gamma \end{pmatrix} = \begin{pmatrix}
0 \\
0 \\
\frac{1}{N+1} \end{pmatrix}, \\
e^T\alpha + \beta + \gamma = 1, \\
\alpha, \beta, \gamma \geq 0,
\end{aligned}
\end{equation}
which is equivalent to the CHMP $p \in \conv{(\tilde{\A})}$, where $p=\left(0^T,0,\frac{1}{N+1}\right)^T \in \R^{m+2}$ and the elements of $\tilde{\A}$ are the columns of the coefficient matrix in \eqref{eq:sistemafacPL}. It is not hard to show that $\Omega \cap H_N$ is nonempty if and only if $p \in \conv{(\tilde{\A})}$.

\begin{remark}
Let $p_k \in \conv{(\tilde{\A})}$ be given such that $d(p_k,p) = \norm{p_k - p} \leq \varepsilon R$, where $R$ is the diameter of $\conv{(\tilde{\A})}$, from \eqref{eq:sistemafacPL}. Then,
\(\hat{x} := \left(\frac{\alpha_1}{\gamma}, \dots, \frac{\alpha_n}{\gamma} \right)^T
\)
satisfies  $\norm{A \hat{x} - b} \leq \varepsilon R/\gamma \eqqcolon \varepsilon'$, $\abs{e^T \hat{x} + \beta/\gamma -N} \leq \varepsilon'$, $\abs{\gamma-\frac{1}{N+1}} \leq \varepsilon R$ and $\hat{x} \geq 0$. 
\end{remark}

We consider instances of the LP feasibility problem
\begin{equation}\label{eq:lpfeas}
Ax=b, \quad x \geq 0 \quad \textrm{and} \ \quad e^T x \leq N, 
\end{equation}
and solve them as CHMPs by using TA, GT, SPG and ASFW, using the stopping criteria described in Section~\ref{sec:stop}. We  compared as well their performances against the Matlab LP solver {\tt linprog} (which is an implementation of the Dual-Simplex algorithm) applied to  \eqref{eq:lpfeas}. 

For each instance, 
the columns of $A$ were 
sampled from the uniform distribution on the unit sphere centered at $e \in \R^m$ (such that $a_{ij} \geq 0$). 
To generate feasible instances of \eqref{eq:lpfeas}, we select a random vector $x \in \Rn$, where each entry is  sampled from the uniform distribution in $(0, 1)$, and compute $b$ as $b = Ax$. 
Since the sum of $n$ independent random variables with uniform distribution has expected value $n/2$ and variance $n/12$, and in these experiments we consider $n\leq 2000$,  to ensure $e^T x \leq N$, we set $N=1200$. 
For infeasible instances, we use the same procedure but at the end we multiply the first coordinate of $b$ by $-1$. In the following tables, for each pair $m$ and $n$ we considered 10 random instances.

For {\tt linprog}, we test the feasibility of \eqref{eq:lpfeas} considering a constraint violation tolerance of $10^{-3}$ (the largest value allowed by {\tt linprog}) and we also disabled  preprocessing procedures. For the other algorithms applied to the equivalent CHMP formulation, we use the same parameters of the experiments in Section~\ref{sec:random} except for the tolerance $\varepsilon$, which here takes values in $\{10^{-6}, 10^{-7}\}$, and the maximum number of iterations, which is  set to $10^6$. In the SPG implementation, based on preliminary experiments, we set $\lambda_{\min} = 10^{-10}$ and $\lambda_{\max} = 10^{10}$, and  we use $M=60$ for the nonmonotone line-search. 
It is worth mention that, for the instances we generated,  the constant $R$, that depends on the problem geometry, is around $10^{3}$.

The running times and  respective iterations count are shown in Tables~\ref{tab:timeiteraothesfeasible}, \ref{tab:linprogtime} and \ref{tab:timeiteraothesinfeasible}. In Table~\ref{tab:timeiteraothesfeasible}, we report time and iterations of TA, GT, ASFW and SPG for feasible problems,  varying dimensions and tolerance $\varepsilon$.  Table~\ref{tab:linprogtime} shows the performance of {\tt linprog} (for which the constraint violation tolerance is fixed).

Note that even  GT and ASFW, which presented the best performances in the experiments of Section~\ref{sec:random},  struggled in these experiments, requiring many iterations and consequently presenting higher running times. 
This was somehow expected because Theorem~\ref{teo:complex1} states the iteration complexity of TA and its variants as $O(\varepsilon^{-2})$ and now $\varepsilon$ is three orders of magnitude smaller than in the previous section.

As for ASFW, we remark that the linear convergence rate appearing in  equation~\eqref{eq:convasfw} depends on $\mu$, $L$, $D$, $\Omega_C$ and the dimension. In these experiments, $C := \conv{\tilde{\A}}$ and the dimension is $m+2$.  The objective function of  \eqref{prob:proj} yields that $\mu = L =1$. From \eqref{eq:sistemafacPL}, 
we deduce that 
$D \geq N+1$ and $\Omega_C \leq 1$. Therefore,
\[
\dfrac{\Omega_C}{D(m+3)} \leq \dfrac{1}{(N+1)(m+3)}.
\]
Since, in our experiments, $N=1200$ and $m=O(10^2)$, the above ratio is smaller than $10^{-5}$ and by replacing it in \eqref{eq:convasfw} we see that convergence factor is very close to $1$, justifying the slow convergence of ASFW for these instances. 

\begin{table}
\centering\small
\caption{CPU time (in seconds) and iterations count for feasible problems. Best times are in bold.}\label{tab:timeiteraothesfeasible}
\begin{tabular}{lrrccccccccc}\toprule
\ & \textbf{} &\textbf{} &\multicolumn{2}{c}{\textbf{TA}} &\multicolumn{2}{c}{\textbf{GT}} &\multicolumn{2}{c}{\textbf{ASFW}} &\multicolumn{2}{c}{\textbf{SPG}} \\  \cmidrule(lr){4-5}  \cmidrule(lr){6-7}  \cmidrule(lr){8-9}  \cmidrule(lr){10-11}
\ & \textbf{$m$} &\textbf{$n$} &\texttt{time } &\texttt{iter} &\texttt{time } &\texttt{iter} &\texttt{time } &\texttt{iter} &\texttt{time } &\texttt{iter} \\\midrule
\multirow{4}{*}{$\varepsilon=10^{-6}$} & 50 &200 &0.5137 &86491.4 &0.3272 &55952.9 &0.1507 &6286.7 &\textbf{0.0053} &56.8 \\
\ & 50 &2000 &0.2669 &10318 &0.0809 &4472.4 &0.0869 &2190.3 &0.2971 &54.4 \\
\ & 100 &500 &1.0028 &74248.7 &0.5599 &46952.8 &1.0131 &27998.3 &\textbf{0.0083} &122 \\
\ & 200 &2000 &1.5370 &29734.1 &0.8193 &17676.6 &1.2437 &18007.5 &\textbf{0.1833} &268.5 \\\midrule
\multirow{4}{*}{$\varepsilon=10^{-7}$} & 50 &200 &2.1938 &330132 &1.2349 &189164.6 &0.7452 &26773.2 &\textbf{0.0080} &58.5 \\
\ & 50 &2000 &0.6298 &24588.7 &0.1503 &8173.1 &0.2172 &5847.9 &0.5040 &55.9 \\
\ & 100 &500 &4.0366 &313218.6 &2.1494 &186011.7 &4.0959 &112368.9 &\textbf{0.0148} &125.2 \\
\ & 200 &2000 &13.7279 &245975.6 &1.8809 &39395.9 &8.8099 &123920.4 &\textbf{0.1923} &266.2 \\\bottomrule
\end{tabular}
\end{table}

\begin{table}
\centering\small 
\caption{Time (in seconds) for  \texttt{linprog}}\label{tab:linprogtime}
\begin{tabular}{ccrrr}\toprule
\bf{$m$} & \bf{$n$} & \bf{feasible} & \bf{infeasible} \\
\cmidrule(lr){1-2}
\cmidrule(lr){3-4}
50 &200 &0.0162 &0.0103 \\
50 &2000 &\textbf{0.0294} &0.0189 \\
100 &500 &0.0297 &0.0128 \\
200 &2000 &0.2180 &0.0424 \\
\bottomrule
\end{tabular}
\end{table}

\begin{table}
\centering \small
	\caption{CPU time (in seconds) and iterations count for infeasible problems.}\label{tab:timeiteraothesinfeasible}
	\begin{tabular}{cclclclclcl}\toprule
	&\textbf{} &\multicolumn{2}{c}{\textbf{TA}} &\multicolumn{2}{c}{\textbf{GT}} &\multicolumn{2}{c}{\textbf{ASFW}} &\multicolumn{2}{c}{\textbf{SPG}} \\ \cmidrule(lr){3-4} \cmidrule(lr){5-6}\cmidrule(lr){7-8} \cmidrule(lr){9-10}
	\textbf{$m$} &\textbf{$n$} &\texttt{time} &\texttt{iter} &\texttt{time} &\texttt{iter} &\texttt{time } &\texttt{iter} &\texttt{time } &\texttt{iter} \\\midrule
	50 &200 &0.0011 &21.1 &\bf{0.0003} &13 &0.0084 &253.9 &0.0018 &3.5 \\
	50 &2000 &0.0016 &35 &\bf{0.0007} &33.2 &0.0015 &33.2 &0.0352 &11.3 \\
	100 &500 &0.0014 &56.9 &\bf{0.0009} &45.4 &0.0015 &45.4 &0.0053 &3.7 \\
	200 &2000 &0.0053 &109.2 &\bf{0.0045} &106.8 &0.0060 &106.8 &0.0513 &7.6 \\\bottomrule
	\end{tabular}
	\end{table}

Table~\ref{tab:timeiteraothesinfeasible} shows the results (on average) for infeasible problems. We remark that in infeasible problems, apart from {\tt linprog}, the  algorithms stop when a witness is found and this depends only on the geometry of the problem, not on the tolerance $\varepsilon$. 
For the instances considered in this section we observed that $R$ is around $10^3$ and $0.14\leq\Delta \leq 0.99$.
We can see that the algorithm with the best performance with respect to time in feasible cases was SPG, followed by {\tt linprog}. For infeasible problems, GT found a witness faster than the other benchmarked algorithms. 

\subsection{An example of CHMP in image classification}\label{sec:mnist}
{The MNIST database is a large database of handwritten digits, which is a very popular dataset for validating image classification algorithms \citep{lecun1998gradient}. 
It contains 60,000 samples of $28\times 28$ grayscale images in its training set, and 10,000 images in its testing set. 
In the interesting study of \cite{yousefzadeh2021} it was shown that, for the MNIST dataset, all test points are considerably outside the convex hull of the set of training points. Furthermore, in such a case, the author suggested the following procedure for classification: 
given a testing point, we estimate its distance to the convex hull of each class\footnote{For MNIST, the classes are $0,1,\dots,9$.} and classify it according to the smallest of the distances. \citet{yousefzadeh2020} and \citet{yousefzadeh2021} reported  an accuracy of $98.5\%$  using this procedure.}

{Inspired by these studies, we apply the algorithms discussed in this work to obtain an estimate for the distance from a test point to the convex hull of (subclasses of) the training set. The difference from the work of \citet{yousefzadeh2021} is that we do not solve problem~\eqref{prob:quad} exactly, but stop the algorithms as soon as a witness is found. Recall from \eqref{eq:witproj} that the distance from a testing point to a witness is not greater than twice the distance from this test point to the convex hull. Thus, we wonder how the use of this estimate would impact the results/conclusions of  \citet{yousefzadeh2021}.}

{Firstly, we consider $\A \subset \mathbb{R}^{784}$ as the training set (with 60,000 vectorized images). 
Then, for each of the $10,000$ test points, we solve the corresponding CHMPs with the same algorithms\footnote{Here, $M=3$ is used in the nonmonotone line search for SPG.} tested in the Section~\ref{sec:random}.  As in \citep{yousefzadeh2021}, we found that all elements of the test set are outside $\convA$. 
Similarly to Section~\ref{sec:casec}, the fastest algorithms to obtain a witness for the 10,000 test points were GT and ASFW, which took approximately a total time of 12 and 13 minutes, respectively. The slowest was SPG which took 114 minutes.} 
{In order to obtain the exact distance from a testing point to the convex hull we also considered a classic version of SPG (without the stopping criteria of Section~\ref{sec:stop}), which we call PROJ. PROJ took a total time of almost 9 hours to find the projection of all test points onto $\convA$.}

{To compare the distance estimates with the exact distances to the convex hull of the training set, we present in Figure~\ref{fig:distmnist} 
a graph with the distance (divided by $255$ which is the largest pixel value) from all test points to their respective witnesses (depending on the method used) 
and a histogram to compare the (estimate) distance distributions for TA and PROJ. As expected, the distances between the witnesses and the test points do not exceed twice the distance to the projection.    
}

\begin{figure}
	\centering
		\includegraphics[scale=0.31,trim=0 100 0 100,clip]{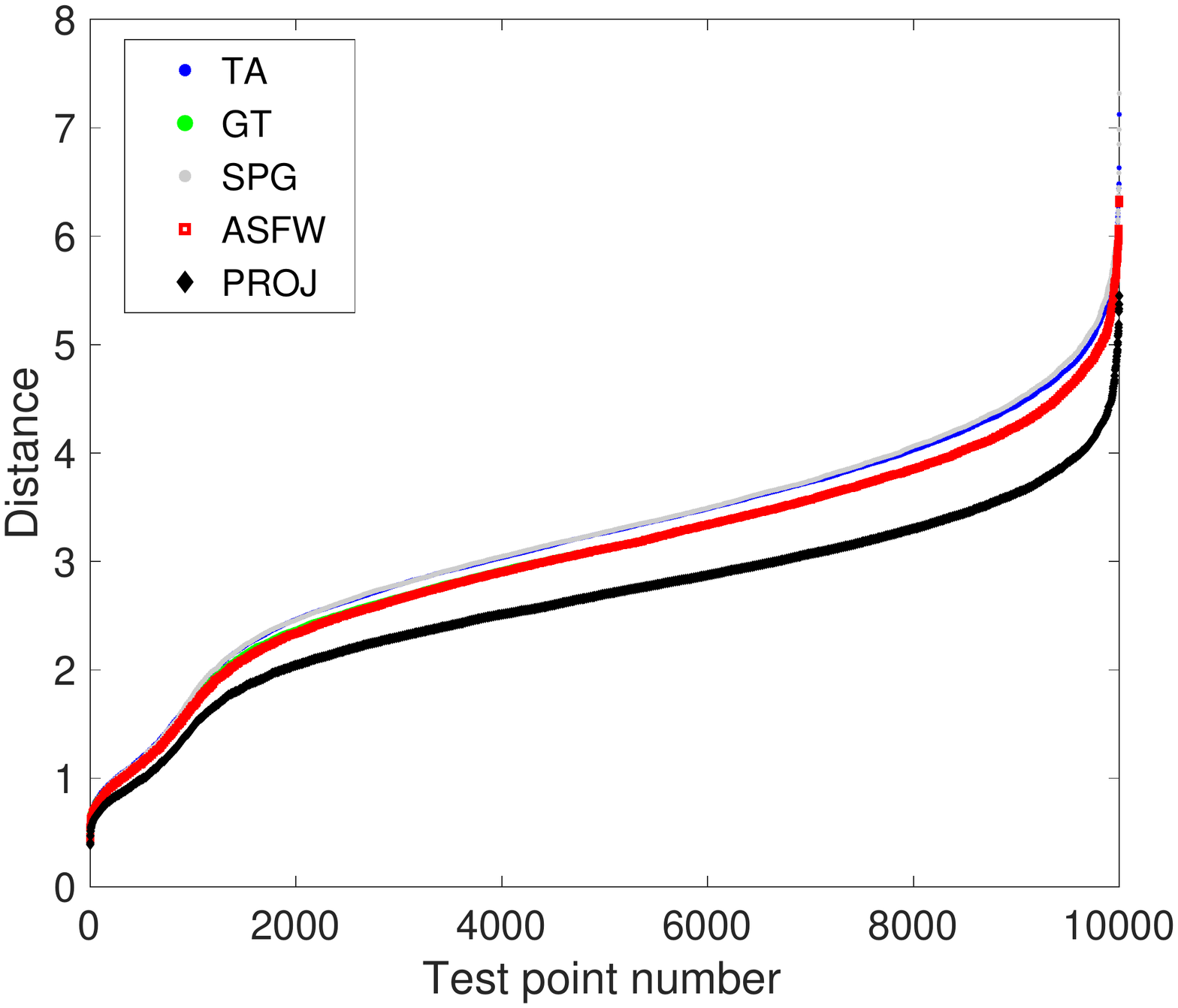} \hfill 
		\includegraphics[scale=0.31,trim=0 100 0 100,clip]{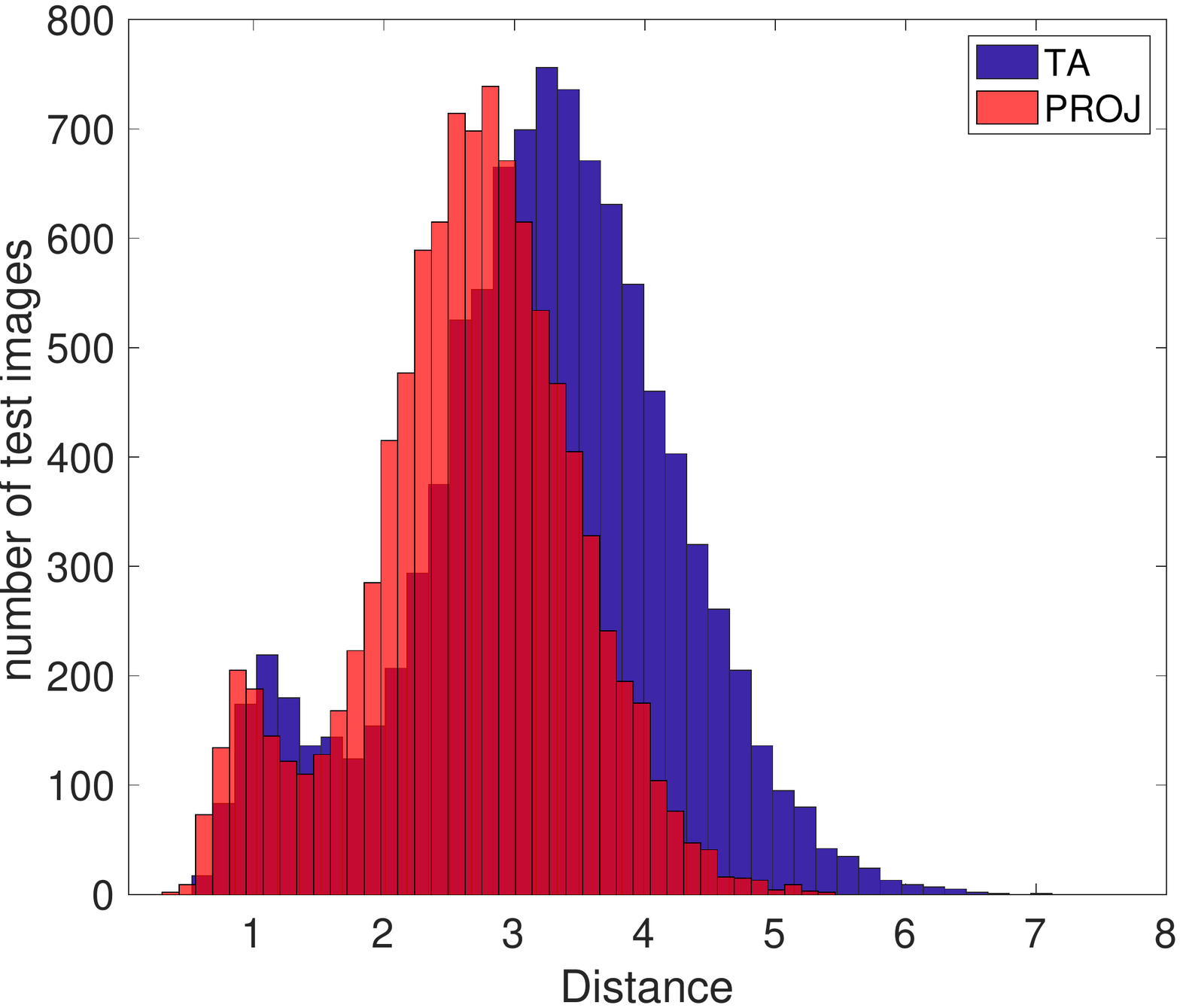}
	\caption{(Left) Distance estimates compared with the distance to the convex hull (PROJ). (Right) Histogram with the distances between the witness and the test point (TA), and the distances between the test point and its projection (PROJ).}\label{fig:distmnist}
\end{figure}

{Secondly, we apply the same procedure for classification as proposed in \citep{yousefzadeh2020}, however we calculate the distance from the testing points to the witness in the convex hull of each class,  and we classify the testing point according to the smallest of the distances. The best accuracy of $98\%$ is achieved by TA in less than 10 minutes, which is not that far from the $98.5\%$ accuracy of PROJ, that required more than 4 hours to process all test points.}

\section{Concluding remarks}\label{sec:end}

{We have shown that first-order methods, such as Frank-Wolfe-type and Projected Gradient-type, 
equipped with appropriate stopping criteria, 
are suitable for solving convex hull membership problems and  are competitive with the recently proposed Triangle Algorithm, which is specific for CHMP.} 

{By exploring the characterization of the so-called pivots, we showed that the Triangle Algorithm can be viewed as an inexact version of the Frank-Wolfe algorithm applied to CHMP whereas the latter, with stopping rules based on distance duality, lead us to a Greedy Triangle Algorithm. This algorithm coincides with an old one, due to von Neumann.} 

{The proposed stopping criteria based on distance duality were essential for the first-order methods to handle the case $p \notin \convA$, avoiding the computation of the projection of the query point onto the convex hull, and also allowing for a fair numerical comparison with the Triangle Algorithm and its variants.}

{A comprehensive set of numerical experiments indicate which algorithm is preferable according to the geometry of the convex hull and the relative position of the query point. Overall, for the random instances considered in this paper, ASFW had a better performance, especially in ``harder to solve'' instances of CHMP, as those of Sections~\ref{sec:caseb} and \ref{sec:cased}, where the classic versions of TA and FW usually present a zigzagging behavior.}

We also gave examples of two potential applications of algorithms for CHMP, one in the linear programming feasibility problem and another in image classification problems.
After casting LP feasibility problem as a CHMP, our experiments show that first-order methods can be superior to the classic dual-simplex algorithm and support the Greedy Triangle algorithm as an efficient candidate to detect infeasibility.
The image classification experiment shows that the use of the distance from a query point to a witness instead of the distance from the point to the convex hull did not change considerably the classification accuracy, and it is about twenty times faster than computing the exact projection.
A further computational study on other image classification datasets is subject of future investigation.

The performance of Frank-Wolfe-type algorithms in feasible instances of LP feasibility problems was different from what we observed in Sections~\ref{sec:casea} and \ref{sec:caseb}.  
As discussed in Section~\ref{sec:lpfeas}, this has to do with the conditioning of the problem (which depends on the geometry of the convex hull), which makes things harder even for the linearly convergent ASFW.  
In future works we plan to investigate this precisely as well as possible acceleration strategies for FW-type methods applied to CHMP.

{A deeper study on Spherical-TA \citep{Kalantari:2019b} and a comparison with the first-order methods considered in this paper shall be subject of our future investigations as well as extensions and applications of distance duality to semidefinite and conic programming \citep{Kalantari:2019,Kalantari:2020a}.}

 \section*{Acknowledgments}
 The authors are grateful to the anonymous referees for meaningful suggestions that helped to improve this paper. 
The authors want to acknowledge Brazilian agency CAPES (Coordenação para Aperfeiçoamento  Pessoal do Ensino Superior) for financial support. \textbf{RF} thanks CAPES for the doctoral scholarship; \textbf{DG} thanks Brazilian agency CNPq (Conselho Nacional de Desenvolvimento Científico e Tecnológico) for grant  305213/2021-0; \textbf{LRS} thanks CNPq for grant 113190/2022-0.
\bibliographystyle{model5-names}\biboptions{authoryear}
\bibliography{Referencias} 

\begin{thebibliography}{31}
\expandafter\ifx\csname natexlab\endcsname\relax\def\natexlab#1{#1}\fi
\providecommand{\url}[1]{\texttt{#1}}
\providecommand{\href}[2]{#2}
\providecommand{\path}[1]{#1}
\providecommand{\DOIprefix}{doi:}
\providecommand{\ArXivprefix}{arXiv:}
\providecommand{\URLprefix}{URL: }
\providecommand{\Pubmedprefix}{pmid:}
\providecommand{\doi}[1]{\href{http://dx.doi.org/#1}{\path{#1}}}
\providecommand{\Pubmed}[1]{\href{pmid:#1}{\path{#1}}}
\providecommand{\bibinfo}[2]{#2}
\ifx\xfnm\relax \def\xfnm[#1]{\unskip,\space#1}\fi
\bibitem[{Awasthi et~al.(2020)Awasthi, Kalantari \& Zhang}]{awasthi2020}
\bibinfo{author}{Awasthi, P.}, \bibinfo{author}{Kalantari, B.}, \&
  \bibinfo{author}{Zhang, Y.} (\bibinfo{year}{2020}).
\newblock \bibinfo{title}{Robust vertex enumeration for convex hulls in high
  dimensions}.
\newblock {\it \bibinfo{journal}{Annals of Operations Research}\/},  {\it
  \bibinfo{volume}{295}\/}, \bibinfo{pages}{37--73}.
  \DOIprefix\doi{10.1007/s10479-020-03557-0}.
\bibitem[{Beck \& Shtern(2017)}]{beck2017linearly}
\bibinfo{author}{Beck, A.}, \& \bibinfo{author}{Shtern, S.}
  (\bibinfo{year}{2017}).
\newblock \bibinfo{title}{Linearly convergent away-step conditional gradient
  for non-strongly convex functions}.
\newblock {\it \bibinfo{journal}{Mathematical Programming}\/},  {\it
  \bibinfo{volume}{164}\/}, \bibinfo{pages}{1--27}.
  \DOIprefix\doi{10.1007/s10107-016-1069-4}.
\bibitem[{Birgin et~al.(2000)Birgin, Mart{\'\i}nez \&
  Raydan}]{birgin2000nonmonotone}
\bibinfo{author}{Birgin, E.~G.}, \bibinfo{author}{Mart{\'\i}nez, J.~M.}, \&
  \bibinfo{author}{Raydan, M.} (\bibinfo{year}{2000}).
\newblock \bibinfo{title}{Nonmonotone {{Spectral Projected Gradient Methods}}
  on {{Convex Sets}}}.
\newblock {\it \bibinfo{journal}{SIAM Journal on Optimization}\/},  {\it
  \bibinfo{volume}{10}\/}, \bibinfo{pages}{1196--1211}.
  \DOIprefix\doi{10.1137/S1052623497330963}.
\bibitem[{Birgin et~al.(2009)Birgin, Mart{\'\i}nez \&
  Raydan}]{birgin2009spectral}
\bibinfo{author}{Birgin, E.~G.}, \bibinfo{author}{Mart{\'\i}nez, J.~M.}, \&
  \bibinfo{author}{Raydan, M.} (\bibinfo{year}{2009}).
\newblock \bibinfo{title}{Spectral projected gradient methods}.
\newblock In \bibinfo{editor}{C.~A. Floudas}, \& \bibinfo{editor}{P.~M.
  Pardalos} (Eds.), {\it \bibinfo{booktitle}{Encyclopedia of
  {{Optimization}}}\/} (pp. \bibinfo{pages}{3652--3659}).
\newblock \bibinfo{address}{{Boston, MA}}: \bibinfo{publisher}{{Springer US}}.
\newblock \DOIprefix\doi{10.1007/978-0-387-74759-0_629}.
\bibitem[{Birgin et~al.(2014)Birgin, Mart{\'\i}nez \&
  Raydan}]{birgin2014spectral}
\bibinfo{author}{Birgin, E.~G.}, \bibinfo{author}{Mart{\'\i}nez, J.~M.}, \&
  \bibinfo{author}{Raydan, M.} (\bibinfo{year}{2014}).
\newblock \bibinfo{title}{Spectral {{Projected Gradient Methods}}: Review and
  {{Perspectives}}}.
\newblock {\it \bibinfo{journal}{Journal of Statistical Software}\/},  {\it
  \bibinfo{volume}{60}\/}. \DOIprefix\doi{10.18637/jss.v060.i03}.
\bibitem[{Canon \& Cullum(1968)}]{canon1968}
\bibinfo{author}{Canon, M.~D.}, \& \bibinfo{author}{Cullum, C.~D.}
  (\bibinfo{year}{1968}).
\newblock \bibinfo{title}{A {{Tight Upper Bound}} on the {{Rate}} of
  {{Convergence}} of {{Frank}}-{{Wolfe Algorithm}}}.
\newblock {\it \bibinfo{journal}{SIAM Journal on Control}\/},  {\it
  \bibinfo{volume}{6}\/}, \bibinfo{pages}{509--516}.
  \DOIprefix\doi{10.1137/0306032}.
\bibitem[{Clarkson(2010)}]{Clarkson}
\bibinfo{author}{Clarkson, K.~L.} (\bibinfo{year}{2010}).
\newblock \bibinfo{title}{Coresets, sparse greedy approximation, and the
  {{Frank}}-{{Wolfe}} algorithm}.
\newblock {\it \bibinfo{journal}{ACM Transactions on Algorithms}\/},  {\it
  \bibinfo{volume}{6}\/}, \bibinfo{pages}{1--30}.
  \DOIprefix\doi{10.1145/1824777.1824783}.
\bibitem[{Condat(2016)}]{Condat2016}
\bibinfo{author}{Condat, L.} (\bibinfo{year}{2016}).
\newblock \bibinfo{title}{Fast projection onto the simplex and the $\ell_1$
  ball}.
\newblock {\it \bibinfo{journal}{Mathematical Programming}\/},  {\it
  \bibinfo{volume}{158}\/}, \bibinfo{pages}{575--585}.
  \DOIprefix\doi{10.1007/s10107-015-0946-6}.
\bibitem[{Dantzig(1992)}]{dantzig92}
\bibinfo{author}{Dantzig, G.~B.} (\bibinfo{year}{1992}).
\newblock {\it \bibinfo{title}{An $\epsilon$-Precise Feasible Solution to a
  Linear Program with a Convexity Constraint in $1/\epsilon^2$ Iterations
  Independent of Problem Size}\/}.
\newblock \bibinfo{type}{Technical Report} \bibinfo{number}{SOL 92-5} {Systems
  Optimization Laboratory, Dept of Management Science and Engineering, Stanford
  University} \bibinfo{address}{{Palo Alto, CA}}.
\newblock \URLprefix \url{https://stanford.edu/group/SOL/reports/SOL-92-5.pdf}.
\bibitem[{Epelman \& Freund(2000)}]{epelman2000}
\bibinfo{author}{Epelman, M.}, \& \bibinfo{author}{Freund, R.~M.}
  (\bibinfo{year}{2000}).
\newblock \bibinfo{title}{Condition number complexity of an elementary
  algorithm for computing a reliable solution of a conic linear system}.
\newblock {\it \bibinfo{journal}{Mathematical Programming}\/},  {\it
  \bibinfo{volume}{88}\/}, \bibinfo{pages}{451--485}.
  \DOIprefix\doi{10.1007/s101070000136}.
\bibitem[{Frank \& Wolfe(1956)}]{fw1956}
\bibinfo{author}{Frank, M.}, \& \bibinfo{author}{Wolfe, P.}
  (\bibinfo{year}{1956}).
\newblock \bibinfo{title}{An algorithm for quadratic programming}.
\newblock {\it \bibinfo{journal}{Naval Research Logistics Quarterly}\/},  {\it
  \bibinfo{volume}{3}\/}, \bibinfo{pages}{95--110}.
  \DOIprefix\doi{10.1002/nav.3800030109}.
\bibitem[{Gon{\c c}alves et~al.(2009)Gon{\c c}alves, Storer \&
  Gondzio}]{goncalves2009}
\bibinfo{author}{Gon{\c c}alves, J.~P.}, \bibinfo{author}{Storer, R.~H.}, \&
  \bibinfo{author}{Gondzio, J.} (\bibinfo{year}{2009}).
\newblock \bibinfo{title}{A family of linear programming algorithms based on an
  algorithm by von {{Neumann}}}.
\newblock {\it \bibinfo{journal}{Optimization Methods and Software}\/},  {\it
  \bibinfo{volume}{24}\/}, \bibinfo{pages}{461--478}.
  \DOIprefix\doi{10.1080/10556780902797236}.
\bibitem[{Grapiglia \& Sachs(2017)}]{gs2017}
\bibinfo{author}{Grapiglia, G.~N.}, \& \bibinfo{author}{Sachs, E.~W.}
  (\bibinfo{year}{2017}).
\newblock \bibinfo{title}{On the worst-case evaluation complexity of
  non-monotone line search algorithms}.
\newblock {\it \bibinfo{journal}{Computational Optimization and
  Applications}\/},  {\it \bibinfo{volume}{68}\/}, \bibinfo{pages}{555--577}.
  \DOIprefix\doi{10.1007/s10589-017-9928-3}.
\bibitem[{Grippo et~al.(1986)Grippo, Lampariello \& Lucidi}]{gll86}
\bibinfo{author}{Grippo, L.}, \bibinfo{author}{Lampariello, F.}, \&
  \bibinfo{author}{Lucidi, S.} (\bibinfo{year}{1986}).
\newblock \bibinfo{title}{A {Nonmonotone} {Line} {Search} {Technique} for
  {Newton}'s {Method}}.
\newblock {\it \bibinfo{journal}{SIAM Journal on Numerical Analysis}\/},  {\it
  \bibinfo{volume}{23}\/}, \bibinfo{pages}{707--716}.
  \DOIprefix\doi{10.1137/0723046}.
\bibitem[{Gu\'{e}lat \& Marcotte(1986)}]{guelat1986some}
\bibinfo{author}{Gu\'{e}lat, J.}, \& \bibinfo{author}{Marcotte, P.}
  (\bibinfo{year}{1986}).
\newblock \bibinfo{title}{Some comments on {Wolfe}'s `away step'}.
\newblock {\it \bibinfo{journal}{Mathematical Programming}\/},  {\it
  \bibinfo{volume}{35}\/}, \bibinfo{pages}{110--119}.
  \DOIprefix\doi{10.1007/BF01589445}.
\bibitem[{Harman \& Lacko(2010)}]{harman2010}
\bibinfo{author}{Harman, R.}, \& \bibinfo{author}{Lacko, V.}
  (\bibinfo{year}{2010}).
\newblock \bibinfo{title}{On decompositional algorithms for uniform sampling
  from $n$-spheres and $n$-balls}.
\newblock {\it \bibinfo{journal}{Journal of Multivariate Analysis}\/},  {\it
  \bibinfo{volume}{101}\/}, \bibinfo{pages}{2297--2304}.
  \DOIprefix\doi{10.1016/j.jmva.2010.06.002}.
\bibitem[{Jaggi(2013)}]{jaggi}
\bibinfo{author}{Jaggi, M.} (\bibinfo{year}{2013}).
\newblock \bibinfo{title}{Revisiting {{Frank}}-{{Wolfe}}: Projection-free
  sparse convex optimization}.
\newblock In \bibinfo{editor}{S.~Dasgupta}, \& \bibinfo{editor}{D.~McAllester}
  (Eds.), {\it \bibinfo{booktitle}{Proceedings of the 30th International
  Conference on Machine Learning}\/} (pp. \bibinfo{pages}{427--435}).
\newblock \bibinfo{address}{{Atlanta, Georgia, USA}}:
  \bibinfo{publisher}{{PMLR}} volume~\bibinfo{volume}{28} of {\it
  \bibinfo{series}{Proceedings of Machine Learning Research}\/}.
\newblock \URLprefix \url{http://proceedings.mlr.press/v28/jaggi13.html}.
\bibitem[{Kalantari(2014)}]{Kalantari}
\bibinfo{author}{Kalantari, B.} (\bibinfo{year}{2014}).
\newblock \bibinfo{title}{A characterization theorem and an algorithm for a
  convex hull problem}.
\newblock {\it \bibinfo{journal}{Annals of Operations Research}\/},  {\it
  \bibinfo{volume}{226}\/}, \bibinfo{pages}{301--349}.
  \DOIprefix\doi{10.1007/s10479-014-1707-2}.
\bibitem[{Kalantari(2019{\natexlab{a}})}]{Kalantari:2019a}
\bibinfo{author}{Kalantari, B.} (\bibinfo{year}{2019}{\natexlab{a}}).
\newblock \bibinfo{title}{An algorithmic separating hyperplane theorem and its
  applications}.
\newblock {\it \bibinfo{journal}{Discrete Applied Mathematics}\/},  {\it
  \bibinfo{volume}{256}\/}, \bibinfo{pages}{59--82}.
  \DOIprefix\doi{10.1016/j.dam.2018.05.009}.
\bibitem[{Kalantari(2019{\natexlab{b}})}]{Kalantari:2019}
\bibinfo{author}{Kalantari, B.} (\bibinfo{year}{2019}{\natexlab{b}}).
\newblock \bibinfo{title}{A {{Triangle Algorithm}} for {{Semidefinite Version}}
  of {{Convex Hull Membership Problem}}}.
\newblock {\it \bibinfo{journal}{arXiv:1904.09854v2 [cs, math]}\/}, .
  \href{http://arxiv.org/abs/1904.09854v2}{\tt arXiv:1904.09854v2}.
\bibitem[{Kalantari(2020)}]{Kalantari:2020a}
\bibinfo{author}{Kalantari, B.} (\bibinfo{year}{2020}).
\newblock \bibinfo{title}{On the {{Equivalence}} of {{SDP Feasibility}} and a
  {{Convex Hull Relaxation}} for {{System}} of {{Quadratic Equations}}}.
\newblock {\it \bibinfo{journal}{arXiv:1911.03989v2 [cs, math]}\/}, .
  \href{http://arxiv.org/abs/1911.03989v2}{\tt arXiv:1911.03989v2}.
\bibitem[{Kalantari \& Park(2014)}]{Kalantari:2014b}
\bibinfo{author}{Kalantari, B.}, \& \bibinfo{author}{Park, J.~Y.}
  (\bibinfo{year}{2014}).
\newblock \bibinfo{title}{Tree convex hull theorems on triangles and circles}.
\newblock {\it \bibinfo{journal}{Honam Mathematical J.}\/},  {\it
  \bibinfo{volume}{36}\/}, \bibinfo{pages}{787--794}.
  \DOIprefix\doi{10.5831/HMJ.2014.36.4.787}.
\bibitem[{Kalantari \& Zhang(2022)}]{Kalantari:2019b}
\bibinfo{author}{Kalantari, B.}, \& \bibinfo{author}{Zhang, Y.}
  (\bibinfo{year}{2022}).
\newblock \bibinfo{title}{Algorithm 1024: {{Spherical Triangle Algorithm}}: {{A
  Fast Oracle}} for {{Convex Hull Membership Queries}}}.
\newblock {\it \bibinfo{journal}{ACM Transactions on Mathematical Software}\/},
   {\it \bibinfo{volume}{48}\/}, \bibinfo{pages}{23:1--32}.
  \DOIprefix\doi{10.1145/3516520}.
\bibitem[{Lacoste-Julien \& Jaggi(2015)}]{lacoste2015global}
\bibinfo{author}{Lacoste-Julien, S.}, \& \bibinfo{author}{Jaggi, M.}
  (\bibinfo{year}{2015}).
\newblock \bibinfo{title}{On the global linear convergence of {Frank}-{Wolfe}
  optimization variants}.
\newblock In {\it \bibinfo{booktitle}{Proceedings of the 28th {International}
  {Conference} on {Neural} {Information} {Processing} {Systems}}\/} (pp.
  \bibinfo{pages}{496--504}).
\newblock \bibinfo{address}{Cambridge, MA, USA}: \bibinfo{publisher}{MIT Press}
  volume~\bibinfo{volume}{1} of {\it \bibinfo{series}{{NIPS}'15}\/}.
\newblock \DOIprefix\doi{10.5555/2969239.2969295}.
\bibitem[{LeCun et~al.(1998)LeCun, Bottou, Bengio \&
  Haffner}]{lecun1998gradient}
\bibinfo{author}{LeCun, Y.}, \bibinfo{author}{Bottou, L.},
  \bibinfo{author}{Bengio, Y.}, \& \bibinfo{author}{Haffner, P.}
  (\bibinfo{year}{1998}).
\newblock \bibinfo{title}{Gradient-based learning applied to document
  recognition}.
\newblock {\it \bibinfo{journal}{Proceedings of the IEEE}\/},  {\it
  \bibinfo{volume}{86}\/}, \bibinfo{pages}{2278--2324}.
  \DOIprefix\doi{10.1109/5.726791}.
\bibitem[{Li \& Kalantari(2013)}]{KalantariComparacao}
\bibinfo{author}{Li, M.}, \& \bibinfo{author}{Kalantari, B.}
  (\bibinfo{year}{2013}).
\newblock \bibinfo{title}{Experimental study of the convex hull decision
  problem via a new geometric algorithm}.
\newblock In {\it \bibinfo{booktitle}{23rd Annual Fall Workshop on
  Computational Geometry, City College of New York}\/}.
\bibitem[{Pe{\~n}a et~al.(2016)Pe{\~n}a, Rodr{\'\i}guez \& Soheili}]{pena2016}
\bibinfo{author}{Pe{\~n}a, J.}, \bibinfo{author}{Rodr{\'\i}guez, D.}, \&
  \bibinfo{author}{Soheili, N.} (\bibinfo{year}{2016}).
\newblock \bibinfo{title}{On the von {{Neumann}} and {{Frank--Wolfe
  Algorithms}} with {{Away Steps}}}.
\newblock {\it \bibinfo{journal}{SIAM Journal on Optimization}\/},  {\it
  \bibinfo{volume}{26}\/}, \bibinfo{pages}{499--512}.
  \DOIprefix\doi{10.1137/15M1009937}.
\bibitem[{Yousefzadeh(2021)}]{yousefzadeh2021}
\bibinfo{author}{Yousefzadeh, R.} (\bibinfo{year}{2021}).
\newblock \bibinfo{title}{Deep {{Learning Generalization}} and the {{Convex
  Hull}} of {{Training Sets}}}.
\newblock {\it \bibinfo{journal}{arXiv:2101.09849 [cs, math]}\/}, .
  \href{http://arxiv.org/abs/2101.09849}{\tt arXiv:2101.09849}.
\bibitem[{Yousefzadeh \& Huang(2020)}]{yousefzadeh2020}
\bibinfo{author}{Yousefzadeh, R.}, \& \bibinfo{author}{Huang, F.}
  (\bibinfo{year}{2020}).
\newblock \bibinfo{title}{Using {{Wavelets}} and {{Spectral Methods}} to
  {{Study Patterns}} in {{Image-Classification Datasets}}}.
\newblock {\it \bibinfo{journal}{arXiv:2006.09879 [cs, eess, math]}\/}, .
  \href{http://arxiv.org/abs/2006.09879}{\tt arXiv:2006.09879}.
\bibitem[{Zhang \& Hager(2004)}]{zh2004}
\bibinfo{author}{Zhang, H.}, \& \bibinfo{author}{Hager, W.~W.}
  (\bibinfo{year}{2004}).
\newblock \bibinfo{title}{A {{Nonmonotone Line Search Technique}} and {{Its
  Application}} to {{Unconstrained Optimization}}}.
\newblock {\it \bibinfo{journal}{SIAM Journal on Optimization}\/},  {\it
  \bibinfo{volume}{14}\/}, \bibinfo{pages}{1043--1056}.
  \DOIprefix\doi{10.1137/S1052623403428208}.
\bibitem[{Zhang \& Kalantari(2016)}]{zhangtrianglerandom}
\bibinfo{author}{Zhang, Y.}, \& \bibinfo{author}{Kalantari, B.}
  (\bibinfo{year}{2016}).
\newblock \bibinfo{title}{{The Triangle Algorithm with Relaxed and Randomized
  Pivots}}.
\newblock In {\it \bibinfo{booktitle}{Proceedings of 26th Fall Workshop on
  Computational Geometry}\/}.
\newblock \bibinfo{address}{New York, NY}: \bibinfo{organization}{CUNY Graduate
  Center}.
\newblock
  \bibinfo{note}{\url{https://matthewpjohnson.org/fwcg2016/FWCG_2016_paper_30.pdf}}.

\end{thebibliography}

%
\newpage
\appendix
\section{Tables with numerical results from Section~\ref{sec:random}}\label{sec:apa}
\begin{table}[h]
\centering
\caption{Iterations count (average)  for the case where $p\in \convA$}\label{tab:iterationspin}
\begin{tabular}{rccccccc}\toprule
&\multicolumn{4}{c}{\textbf{Case (a)}} &\multicolumn{2}{c}{\textbf{Case (b)}} \\\cmidrule(lr){2-5} \cmidrule(lr){6-7}
${n}$ &\textbf{TA} &\textbf{ASFW} &\textbf{GT} &\textbf{SPG} &\textbf{ASFW} &\textbf{SPG} \\
500 &2557.3 &573.9 &662.2 &23.7 &12 &8 \\
1000 &1544.8 &242.8 &247.9 &15.9 &12.5 &8.8 \\
1500 &1428.1 &195.0 &196.5 &15.0 &13 &9.3 \\
2000 &1373.0 &167.4 &169.7 &13.8 &12 &8.0 \\
2500 &1317.2 &153.4 &153.7 &13.1 &12 &7.8 \\
3000 &1345.9 &147.3 &146.1 &13.1 &12.5 &10 \\
3500 &1288.7 &139.1 &139.5 &13.1 &12 &8.6 \\
4000 &1295.5 &132.3 &133.3 &13.0 &12 &8.9 \\
4500 &1270.0 &129.8 &129.8 &13.0 &13 &10.3 \\
5000 &1309.7 &126.1 &126.6 &13.0 &12 &8.4 \\
$1\times 10^4$ &1261.5 &110.7 &110.5 &12.0 &13 &11.8 \\
$2\times 10^4$ &1254.6 &99.0 &99.0 &11.9 &13 &12.4 \\
$3\times 10^4$ &1220.8 &93.4 &93.4 &12.0 &13 &10.6 \\
$4\times 10^4$ &1251.6 &88.0 &87.8 &12.0 &12 &11.0 \\
$5\times 10^4$ &1235.0 &86.7 &86.7 &12.0 &12 &12.6 \\
$6\times 10^4$ &1243.1 &84.1 &84.1 &12.0 &13 &12.8 \\
$7\times 10^4$ &1265.8 &83.5 &83.4 &12.0 &12 &9.4 \\
$8\times 10^4$ &1247.3 &82.2 &82.2 &12.0 &13 &13.1 \\
$9\times 10^4$ &1233.6 &80.7 &80.8 &12.0 &13 &11.0 \\
$10^5$ &1225.5 &79.3 &79.4 &12.0 &12 &10.2 \\\bottomrule
\end{tabular}
\end{table}

\begin{table}
\centering
\caption{Iterations count (average) for the cases where $p \notin \convA$}\label{tab:iterationspout}
\begin{tabular}{rccccccccc}\toprule
\textbf{} &\multicolumn{4}{c}{\textbf{Case (c)}} &\multicolumn{4}{c}{\textbf{Case (d)}} \\\cmidrule(lr){2-5} \cmidrule(lr){6-9}
$n$ &\textbf{TA} &\textbf{ASFW} &\textbf{GT} &\textbf{SPG} &\textbf{TA} &\textbf{ASFW} &\textbf{GT} &\textbf{SPG} \\
500 &2.3 &1 &1 &1.3 &6570.6 &9.2 &6575.2 &4.6 \\
1000 &3 &1 &1 &1.3 &7347.5 &9.1 &7358.4 &4.0 \\
1500 &1.9 &1 &1 &1.2 &6474 &9.2 &6483.2 &4.8 \\
2000 &2.6 &1 &1 &1.3 &7233.5 &9.0 &7246.2 &4.5 \\
2500 &2.2 &1 &1 &1.5 &6604 &9.1 &6616.4 &4.1 \\
3000 &2.2 &1 &1 &1.5 &6657.9 &9.2 &6670.2 &4.6 \\
3500 &3 &1 &1 &1.4 &6519.3 &9.2 &6529.5 &4.0 \\
4000 &2.0 &1 &1 &1.7 &5553.0 &9.2 &5560.4 &4.8 \\
4500 &2.7 &1 &1 &1.5 &6275.2 &9.2 &6283.2 &5.1 \\
5000 &2.3 &1 &1 &1.5 &6628.0 &9.1 &6638.5 &4.3 \\
$1\times 10^4$ &3.3 &1 &1 &1.3 &6154.4 &9.1 &6165.8 &4.9 \\
$2\times 10^4$ &2.1 &1 &1 &1.8 &5671.3 &9.2 &5680.0 &5.2 \\
$3\times 10^4$ &2.6 &1 &1 &1.6 &6184.6 &9.0 &6201.9 &4.0 \\
$4\times 10^4$ &2.9 &1 &1 &1.6 &6146.2 &9.0 &6159.5 &4.0 \\
$5\times 10^4$ &2.6 &1 &1 &1.6 &5795.8 &9.2 &5811.4 &5 \\
$6\times 10^4$ &2.6 &1 &1 &1.6 &6390.5 &9.1 &6416.2 &4.4 \\
$7\times 10^4$ &3.4 &1 &1 &1.5 &5890.2 &9.0 &5903.4 &4.6 \\
$8\times 10^4$ &2.7 &1 &1 &1.6 &5783.6 &9.1 &5803.0 &4.9 \\
$9\times 10^4$ &2.6 &1 &1 &1.4 &5324 &9.2 &5334.4 &5 \\
$10^5$ &2.6 &1 &1 &1.7 &5542.4 &9.0 &5560.3 &4.4 \\
\bottomrule
\end{tabular}
\end{table}
\end{document}